\documentclass[12pt]{amsart}

\usepackage{amsmath, amsopn, amssymb}

\usepackage{a4wide,xspace,natbib,color}

\usepackage{latexsym, amsbsy, epsfig, enumerate}

\usepackage{hyperref}

\usepackage{mathsym}
\newcommand{\suman}{\ensuremath{\sum_{i=1}^{n}}}
\newcommand{\E}{\ensuremath{\mathsf{E}}}
\newcommand{\Es}{\ensuremath{\mathsf{E}\,}}
\newcommand{\EX}{\ensuremath{\mathsf{E}_{\Xb}}}
\newcommand{\var}{\ensuremath{\mathsf{var}}}
\newcommand{\ind}{\ensuremath{\mathbf{1}}}
\newcommand{\Fhat}{\ensuremath{\widehat{F}}}
\newcommand{\Uhat}{\ensuremath{\widehat{U}}\xspace}
\newcommand{\Utilde}{\ensuremath{\widetilde{U}}\xspace}
\newcommand{\Ubhat}{\ensuremath{\widehat{\Ub}}\xspace}
\newcommand{\Ubtilde}{\ensuremath{\widetilde{\Ub}}\xspace}
\newcommand{\Hx}{\ensuremath{H_{\xb}}\xspace}
\newcommand{\hateps}{\ensuremath{\widehat{\eps}}\xspace}

\newcommand{\skorb}{\boldsymbol{\psi}}

\renewcommand{\Pr}{\ensuremath{\mathsf{P}}\xspace}

\newcommand{\epsb}{\ensuremath{\boldsymbol{\eps}}}

\newcommand{\tbn}{\ensuremath{\tb^{(n)}}}

\newcommand{\hatthetab}{\ensuremath{\boldsymbol{\widehat{\theta}}}}

\newcommand{\hatF}{\ensuremath{\widehat{F}}}

\DeclareRobustCommand{\inDist}{\ensuremath{\xrightarrow[n\rightarrow\infty]{d}}}
\DeclareRobustCommand{\inPr}{\ensuremath{\xrightarrow[n\rightarrow\infty]{\mathrm{P}}}}

\newtheorem{theorem}{Theorem}
\newtheorem{corollary}{Corollary}
\newtheorem{lemma}{Lemma}

\theoremstyle{remark}
\newtheorem{remark}{Remark}
\theoremstyle{definition}

\newtheorem*{definition*}{Definition}

\newtheorem{assumpC}{}

\def\boxit#1{\vbox{\hrule\hbox{\vrule\kern6pt
          \vbox{\kern6pt#1\kern6pt}\kern6pt\vrule}\hrule}}

\begin{document}

\vspace*{-1 cm}

\noindent
\title[MPLE in copula semiparametric models]{Maximum pseudo-likelihood estimation based on estimated residuals in copula semiparametric models}
% \author[M.-P. C\^ot\'e, C. Genest, M. Omelka]{Marie-Pier C\^ot\'e$^1$, Christian Genest$^1$, Marek Omelka$^{2}$
% }
\author{Marek  Omelka$^{1}$, \v{S}\'{a}rka Hudecov\'{a}$^{1}$, Natalie Neumeyer$^2$}

\maketitle

\begin{center}
$^1$ 
Department of Probability and Statistics,
Faculty of Mathematics and Physics,  
Charles University, 
%  in Prague
Sokolovsk\'a 83,
186\,75 Praha 8, Czech Republic
\\
$^2$
Department of Mathematics, University of Hamburg, Bundesstrasse 55, 20146 Hamburg, Germany \\
\end{center}
\vspace*{0.6 cm}

\maketitle

%  \begin{center}
%  $^1$  
%  Department of Mathematics and Statistics, McGill University, \\
%  805, rue Sherbrooke ouest, Montr\'eal (Qu\'ebec) Canada H3A 0B9\\
%  $^2$
%  % Katedra pravd\v{e}podobnosti a matematick\'e statistiky,
% % Matematicko-Fyzik\'aln\'i Fakulta,  Univerzita Karlova v Praze,
% %  Sokolovsk\'a 83, 18675 Praha 8, \v{C}esk\'a Republika
% Department of Probability and Mathematical Statistics,  Faculty of 
% Mathematics and Physics, Charles University, 
%  Sokolovsk\'a 83, 186\,75 Praha 8, Czech Republic 
% % 
%  \end{center}

\bigskip
\begin{center}
\textcolor{red}{\today}
\end{center}

% % % % Abstract % % %
\begin{abstract}
This paper deals with a situation when one is interested in 
the dependence structure of a multidimensional response variable in the 
presence of a multivariate covariate. It is assumed that the covariate affects only 
the marginal distributions through regression models 
while the dependence structure, which is described 
by a copula, is unaffected.
A parametric estimation of the copula function is considered with 
focus on the maximum pseudo-likelihood method.
It is proved
that under some appropriate regularity assumptions the estimator 
calculated from the residuals is asymptotically equivalent to the estimator 
based on the unobserved errors.  In such case  one can  ignore the fact 
that the response is first adjusted for the effect of the covariate. A Monte Carlo simulation 
study explores (among others) situations where the regularity assumptions 
are not satisfied and the claimed result does not hold. It shows that  in such situations 
the maximum pseudo-likelihood estimator may behave poorly and 
the moment estimation of the copula parameter is of interest.   Our results complement the  results available for nonparametric  estimation of the copula function.
\end{abstract}
% \vspace*{0.5 cm}
% 

% \begin{singlespace}
% 
% \begin{footnotetext}
% { $^*$ 
% }
% \end{footnotetext}
% \end{singlespace}

\date{\today}
\maketitle

\thispagestyle{empty}

\section{Introduction}

% Joint regression analysis of correlated data using copulas takes its roots in the work of \cite{Song:2007}. It can be broadly described as follows. 

Consider a $d$-dimensional vector $\Yb = (Y_1, \ldots , Y_d)\tr$ of responses and an associated 
$q$-dimensional vector of the covariates $\Xb = (X_1, \ldots,X_{q})\tr$. 
For instance in insurance applications 
one can consider that the response represents various type of payments related to a given 
car accident (medical benefits, income replacement benefits, and allocated expenses for a claimant) 
and the covariates present some additional information (claimant’s age, gravity of accident, 
number of people injured in the accident, \ldots). 

Often we are interested in the conditional distribution of $\Yb$ given the value 
of the covariate. To simplify the situation it is often assumed 
that $\Xb$ affects only the marginal 
distributions of~$Y_j\ (j=1,\dotsc,d)$,  but does not affect the dependence structure of~$\Yb$. 
More formally, it is assumed that there exists a copula $C$ such that the joint conditional distribution 
of $\Yb$ given $\Xb = \xb$ can be for all $\xb \in S_{\Xb}$ (the support 
of $\Xb$) written as 
\begin{equation*}
% \label{eq: Hx}
\Hx(y_{1}, \ldots, y_{d}) = \Pr (Y_{1} \leq y_{1}, \ldots, Y_{d} \leq y_{d} \mid \Xb = \xb ) = C \big(F_{1\xb}(y_1),\ldots,F_{d\xb}(y_d)\big)
\end{equation*}
where $F_{j\xb}(y_j)=\Pr (Y_{j} \leq y_{j}\mid \Xb = \xb)$, $j=1,\dotsc,d$. 
Using this assumption one can proceed in two steps. In the first step one models the effect 
of the covariate on each of the marginal distributions separately (i.e.\ estimating $F_{j\xb}$ 
for each $j \in \{1,\dotsc,d\}$ separately). Having $\Fhat_{j\xb}$ one estimates 
the copula function~$C$ in the second step. 

% \boxit{Maybe other possibilities to estimate $H_{\xb}$ could be discussed to put 
% the problem in a more general context. 
% }

Nonparametric estimation of the copula function~$C$ (for $d=2$ and $q=1$) was in detail 
considered in \cite{ogv_sjs_2015}. The most interesting result is as follows. 
Suppose that the marginal distributions follow the parametric or even non-parametric 
location scale models, i.e. 
\begin{equation} \label{eq: location-scale model}
 Y_{j} = m_{j}(X) + s_{j}(X)\eps_{j},  \quad 
 \text{where } \eps_{j} \text{ is independent with } X.   
\end{equation}
Note that then $C$ is the copula function corresponding to the 
random vector $(\eps_{1},\eps_{2})\tr$. Then 
\cite{ogv_sjs_2015} proved that (under some regularity assumptions) 
the empirical copula~$\widehat{C}_{n}$ based on the estimated residuals 
from model \eqref{eq: location-scale model} 
% % 
% \[
%  \epshat_{ji} = \frac{Y_{ij} - \widehat{m}_{j}(X_{i})}{\widehat{\sigma}_{j}(X_{i})}, 
% \ i=1,\dotsc,d,\, j=1,2,
% \]
% % 
is asymptotically equivalent to the empirical copula~$\widetilde{C}_{n}$ calculated 
from the unobserved errors~$\eps_{ji}$. More precisely it was proved that 
\begin{equation} \label{eq: oracle property of Cnhat}
 \sup_{(u_1,u_2) \in [0,1]^2}\sqrt{n}\,\big|\widehat{C}_{n}(u_1, u_2) - \widetilde{C}_{n}(u_1, u_2)\big| 
 = o_{P}(1). 
\end{equation}
This result was generalized to time-series setting by \cite{noh_copula_nts_2019}. 
In \cite{portier2018weak} the authors were even able to drop 
the location-scale assumption~\eqref{eq: location-scale model} but at 
the cost of deriving only a slightly weaker result
(the supremum in~\eqref{eq: oracle property of Cnhat} 
is replaced with $\sup_{[\gamma,1-\gamma]^{2}}$ where $\gamma$ can be 
taken arbitrarily small but positive).   
On the other hand  \cite{cote_genest_omelka_2019} concentrated on 
the parametric form of the location scale model 
\eqref{eq: location-scale model} and generalized the results to $d > 2$, $q > 1$ 
and at the same time relaxed assumptions on~$f_{j\eps}$ (the density of $\eps_{ji}$). 

To complement the results on nonparametric estimation of~$C$ one is naturally interested 
if analogous results hold also for parametric estimation of~$C$. More precisely 
suppose that the copula function~$C$ belongs to the family $\mathcal{C} = \big\{C(\cdot; \ab) : \ab \in \Theta \big\}$ 
and we are interested in estimating the unknown parameter. 
Denote $\alfab$ the true value of the parameter,  
$\widehat{\alfab}_{n}$ the estimator based on the residuals ($\widehat{\eps}_{ji}$) 
and $\widetilde{\alfab}_{n}$ its counterpart based on the true 
(but unobserved) errors ($\eps_{ji}$) from the location-scale model (\ref{eq: location-scale model}). Then in analogy to 
\eqref{eq: oracle property of Cnhat} one would expect that 
$\widehat{\alfab}_{n}$ is (the first-order) asymptotically equivalent to $\widetilde{\alfab}_{n}$, 
i.e. 
\begin{equation} \label{eq: equiv of estim of parameters}
  \sqrt{n}\, \big(\widehat{\alfab}_{n} - \widetilde{\alfab}_{n}  \big) = o_{P}(1). 
\end{equation}
Although the conjecture \eqref{eq: equiv of estim of parameters} seems to be natural, 
to the best of our knowledge there are only limited results specifying the regularity assumptions 
that are needed so that \eqref{eq: equiv of estim of parameters} holds. Some results for the moment-like 
estimators that can be deduced from the convergence of the empirical copula $\widehat{C}_{n}$  
can be found in \cite{noh_copula_nts_2019} and \cite{cote_genest_omelka_2019}. 

In this paper \citep[similarly as in][]{cote_genest_omelka_2019} we 
assume the parametric form of the location-scale model \eqref{eq: location-scale model} 
and concentrate on \textbf{maximum pseudo-likelihood estimation}. This method 
of estimation was in the context of copula models popularised by \cite{genest_et_al_1995} and in more detail investigated in \cite{tsukahara_2005}. 
This method is often preferred to moment-like estimation because 
the resulting estimator has usually a lower asymptotic variance. 
% As we will see later, 
% when proving \eqref{eq: equiv of estim of parameters} the main difficulty comes 
% from the fact that score functions of many standard copula families are unbounded. 

In the econometric (time-series) literature the inference based on 
the residuals is also known as univariate (marginal) filtering 
\citep[see e.g.,][]{bucher2015nonparametric} and the 
result~\eqref{eq: equiv of estim of parameters} is supported 
by many simulation studies. The result is formulated already  
in \cite{chen2006estimation} but there it is presented more on an intuitive level and the precise assumptions (as well as reasoning) are missing. This lack of 
of rigorousness were to some extent  redeemed in the subsequent paper \cite{chan2009statistical} where the authors concentrated 
on the multivariate GARCH-models and presented a lot of interesting ideas how 
to deal with the technical difficulties. But a careful reading of the paper 
reveals that (probably due to the broad scope of the presented results) some of the 
crucial steps in the proofs are missing. 
% Thus the 
% belief tresult~\eqref{eq: equiv of estim of parameters}

In our paper we will explore in detail the assumptions that are needed so that 
\eqref{eq: equiv of estim of parameters} holds in the standard i.i.d. setting. 
Even in this relatively simply setting one has to handle many 
technical difficulties. The thing is that it is not 
clear how to make use to of the recent deep results in empirical copula estimation \citep[see e.g.,][]{berghaus2017weak,radulovic2017weak} 
as the densities of many standard copulas 
are unbounded. The only remarkable exception in this aspect is Theorem~3.3 
of \citep{berghaus2017weak}, but the authors 
considered only two dimensional copulas and no covariates. 

We show that although the assumptions that guarantess 
\eqref{eq: equiv of estim of parameters} are mild, they are not satisfied 
for some combinations of commonly used copula functions and marginal densities.
 Roughly speaking 
we illustrate that an unbounded copula density has to be compensated with marginal 
densities that are well behaved not only in the supports of the corresponding distributions, 
but also at the border points of the supports. 
We are convinced that exploring this problem  in this settings is not only of independence 
interest, but it provides also insights to understand what might go wrong when switching 
to more complicated econometric or time-series models (see also the discussion in Section~\ref{sec: discussion}). 

The paper is organised as follows. The main result and the needed assumptions are 
formulated in Section~\ref{sec: main results}.  The theoretical results are illustrated 
in a simulation study in Section~\ref{sec: sim study}.  
All the proofs are given in the Appendices.

\section{Main result} \label{sec: main results}
In what follows we assume that 
for each $j \in \{1,\dotsc,d\}$  there exists 
a \textbf{known} transformation $T_j$ 
increasing on the support of $Y_j$ and \textbf{known} functions $m_{j}(\xb;\thetab_{j})$ 
and $s_{j}(\xb; \thetab_j)$ depending only on an unknown (finite-dimensional) parameter~$\thetab_{j}$ 
such that the random variable 
\begin{equation*} % \label{eq: eps}
 \eps_{j} = \frac{T_{j}(Y_j) -m_{j}(\Xb; \thetab_{j})}{s_{j}(\Xb; \thetab_{j})},  
\end{equation*}
is independent of $\Xb$ with cumulative distribution function $F_{j\varepsilon}$. The distribution of 
the random vector $\epsb = (\eps_1,\dotsc,\eps_d)\tr$ has continuous margins and the 
copula corresponding to  $\epsb$  belongs to the 
families of copulas $\mathcal{C} = \big\{C(\cdot; \ab) : \ab \in \Theta \big\}$ 
and $\Theta \subset \RR^{p}$. 

Our task is to estimate the true value of the copula parameter (say $\alfab$) 
based on the observations $\binom{\Yb_1}{\Xb_1},\dotsc,\binom{\Yb_n}{\Xb_n}$ 
% $(\Yb_1,\Xb_1),\dotsc,(\Yb_n,\Xb_n)$ 
that are assumed to be mutually independent copies of the vector $\binom{\Yb}{\Xb}$.

% and  $\mu_j$ and $\sigma_j$ are either known functions up to an unknown parameter~$\thetab_j$ (i.e. of the form $m_{j}(\xb; \thetab_j)$ and $s_{j}(\xb; \thetab_j)$ with the functions $m_j$ and $s_j$ known) 
% or unknown but sufficiently smooth functions. 

Let $\Yb_{i} = (Y_{1i},\dotsc,Y_{di})\tr$. As the parameters $\thetab_j$ ($j\in \{1,\dotsc,d\}$)
% or even the entire functions $\mu_{j}$ 
% and $\sigma_{j}$ 
are in practice unknown, we work with the residuals 
\[
%  \widehat{\eps}_{ji} = \frac{T_{j}(Y_{ji}) -\mu_{j}(\Xb_{i};\widehat{\thetab}_j)}{\sigma_{j}(\Xb_{i};\widehat{\thetab}_j)}, \quad \text{or} \quad   
 \widehat{\eps}_{ji} = \frac{T_{j}(Y_{ji}) - m_{j}(\Xb_{i}; \widehat{\thetab}_{j})}{s_{j}(\Xb_{i}; \widehat{\thetab}_{j})}, 
 \quad i=1,\dotsc,n; \ j=1,\dotsc,d,
\]

where $\widehat{\thetab}_j$ is a suitable estimate of ${\thetab}_j$.
% % 
% where $\widehat{\mu}_{j}(\xb)$ and $\widehat{\sigma}_{j}(\xb)$ stand either for 
% $m(\xb; \widehat{\thetab}_{j})$ and $s(\xb; \widehat{\thetab}_{j})$ or for nonparametric estimators of $\mu_{j}(\xb)$ and $\sigma_{j}(\xb)$. 
%
% 
For $j \in \{1,\dotsc,d\}$ let $\hatF_{j\hateps}$ be the marginal empirical distribution function 
of the estimated residuals, i.e. 
\[
 \hatF_{j\hateps}(y) = \frac{1}{n} \suman \ind\{\hateps_{ji} \leq y\}.  
%  \quad j=1,\dotsc,d.
\]

Then the \textbf{maximum pseudo-likelihood estimator} based on the 
residuals is defined as  
% 
% \begin{equation*} %\label{eq: MPL based on residuals}
\[
 \widehat{\alfab}_{n} = \argmax_{\ab \in \Theta} \suman 
   %w_{n}(\Ubhat_{i})\,
 \log \big\{c\big(\Ubhat_{i}; \ab\big)\big\},  
\]
% \end{equation*}
% 
where 
\begin{equation} \label{eq: estimated pseudoobservations}
 \Ubhat_{i} = \big(\Uhat_{1i},\dotsc,\Uhat_{di} \big)^{\top} 
  = \tfrac{n}{n+1}\big(\widehat{F}_{1\hateps}(\hateps_{1i}),\dotsc, \widehat{F}_{d\hateps}(\hateps_{di})\big)^{\top}  
\end{equation}
are the estimated pseudo-observations and $c(\ub;\ab)$ is the density of the 
assumed copula family. As it is common in the maximum likelihood theory 
we will consider the estimator $\widehat{\alfab}_{n}$ to be an appropriately chosen 
root of the estimating equations 
\begin{equation} \label{eq: MPL based on residuals}
 \suman \skorb(\Ubhat_{i}; \widehat{\alfab}_{n}) = \mathbf{0}_{p}, 
 \quad \text{where} \quad \skorb(\ub;\ab) = \frac{\partial \log \{c\big(\ub; \ab\big)\}}{\partial \ab}.  
\end{equation} 
Analogously let $\widetilde{\alfab}_{n}$ be the corresponding estimator based on 
the true (but unobserved) errors $\eps_{ji}$. I.e.\ $\widetilde{\alfab}_{n}$ is defined as (an appropriately chosen) 
root of the estimating equations 
\begin{equation} \label{eq: MPL based on errors}
 \suman \skorb(\Ubtilde_{i}; \widetilde{\alfab}_{n}) = \mathbf{0}_{p},
\end{equation} 
where 
\begin{equation} \label{eq: true pseudoobservations}
 \Ubtilde_{i} = \big(\Utilde_{1i},\dotsc,\Utilde_{di} \big)^{\top} 
  = \tfrac{n}{n+1}\big(\widehat{F}_{1\eps}(\eps_{1i}),\dotsc, \widehat{F}_{d\eps}(\eps_{di})\big)^{\top}    
\end{equation}
and $\hatF_{j\eps}$ is the marginal empirical distribution function of the 
(unobserved) errors, i.e. 
\[
 \hatF_{j\eps}(y) = \frac{1}{n} \suman \ind\{\eps_{ji} \leq y\}, 
 \quad j=1,\dotsc,d. 
\]

\subsection{Regularity assumptions on the marginal distributions}

In general we need to assume that the density of the error term~$\eps_{j}$  
should be `well-behaved' on the border of its support. The 
following assumption is close  to assumption F(iii) in  
Appendix~A of \cite{einmahl2008specification}. But our assumption 
is weaker as it allows for distributions with supports 
different from a real line. 

% As far as I see this should be a milder assumption than assumption F(iii) in  
% Appendix~A of \cite{einmahl2008specification}. Note that (among) others our assumption 
% allows for distributions with support different from a real line. On the other hand 
% I am still not sure if our assumption cannot be expressed in a more natural way. 

% \medskip 

\noindent \textbf{Assumption} $(\mathbf{F}_{j\eps})$:  %\label{assumption Fjeps}
For each $j \in \{1,\dotsc,d\}$ the density function $f_{j\eps}$ of $\eps_{j}$ is continuous on the support of~$\eps_{j}$
% $f_{j\eps}\big(F_{j\eps}^{-1}(u)\big)$ and 
% $f_{j\eps}\big(F_{j\eps}^{-1}(u)\big)\,F_{j\eps}^{-1}(u)$ are continuous 
and there exists $\beta \in [0,\frac{1}{2})$ such that 
\begin{equation} \label{eq: fjeps at quantile} 
 \sup_{u \in (0,1)}\frac{f_{j\eps}\big(F_{j\eps}^{-1}(u)\big)\big(1 + |F_{j\eps}^{-1}(u)|\big)}{u^{\beta}(1-u)^{\beta}} < \infty 
% \quad \text{ and } \quad 
%  \inf_{u \in (0,1)}\frac{f_{j\eps}\big(F_{j\eps}^{-1}(u)\big)}{u^{\beta}(1-u)^{\beta}} > 0, 
%  \qquad  f_{j\eps}\big(F_{j\eps}^{-1}(u)\big) F_{j\eps}^{-1}(u) = O(u^{\beta}(1-u)^{\beta}),    
%  f_{j\eps}\big(F_{j\eps}^{-1}(u)\big) = O(u^{\beta}(1-u)^{\beta}), 
%  \qquad  f_{j\eps}\big(F_{j\eps}^{-1}(u)\big) F_{j\eps}^{-1}(u) = O(u^{\beta}(1-u)^{\beta}),    
\end{equation}
and 
\begin{equation*} %\label{eq: fjeps at quantile second} 
 \sup_{u \in (0,1/2)}\frac{f_{j\eps}\big(F_{j\eps}^{-1}(2u)\big)}{f_{j\eps}\big(F_{j\eps}^{-1}(u)\big)} < \infty
\quad \text{ and } \quad 
 \sup_{u \in (1/2,1)}\frac{f_{j\eps}\big(F_{j\eps}^{-1}(1-2u)\big)}{f_{j\eps}\big(F_{j\eps}^{-1}(1-u)\big)} < \infty. 
\end{equation*}
% 
% where $\beta$ is the same as in \eqref{eq: generalized chibisov reilly for resid}. 
Further for some $u_1$, $u_2$ in $(0,1)$ the function  $f_{j\eps}\big(F_{j\eps}^{-1}(u)\big)$ is non-decreasing on $(0,u_{1})$ and non-increasing on $(u_2,1)$.   
\medskip 

% \boxit{
% As far as I see this should be a milder assumption than assumption F(iii) in  
% Appendix~A of \cite{einmahl2008specification}. Note that (among) others our assumption 
% allows for distributions with support different from a real line. On the other hand 
% I am still not sure if our assumption cannot be expressed in a more natural way. 
% }

Note that assumption $(\mathbf{F}_{j\eps})$ with $\beta = 0$
allows also  for distributions with non-continuous but bounded densities 
(e.g.\ exponential and uniform). 
But as we 
show later, for copula families with 
unbounded densities one needs to assume that $\beta > 0$. 

\begin{remark}\label{remark1}
 The assumption $(\mathbf{F}_{j\eps})$ is formulated so that it covers the general 
case when both the conditional mean as well as the conditional variance 
of $T_{j}(Y_{ji})$ depends on $\Xb_{i}$. From the proofs given in the appendix 
it follows that if one rightly assumes that the conditional variance does not depend on $\Xb_{i}$, 
then one does only location adjustment (i.e.\ $\widehat{\eps}_{ji} = T_{j}(Y_{ji}) - m_{j}(\Xb_{i}; \widehat{\thetab}_{j})$) and assumption \eqref{eq: fjeps at quantile} simplifies to 
\begin{equation*} 
 \sup_{u \in (0,1)}\frac{f_{j\eps}\big(F_{j\eps}^{-1}(u)\big)}{u^{\beta}(1-u)^{\beta}} < \infty.  
% \quad \text{ and } \quad 
%  \inf_{u \in (0,1)}\frac{f_{j\eps}\big(F_{j\eps}^{-1}(u)\big)}{u^{\beta}(1-u)^{\beta}} > 0, 
%  \qquad  f_{j\eps}\big(F_{j\eps}^{-1}(u)\big) F_{j\eps}^{-1}(u) = O(u^{\beta}(1-u)^{\beta}),    
%  f_{j\eps}\big(F_{j\eps}^{-1}(u)\big) = O(u^{\beta}(1-u)^{\beta}), 
%  \qquad  f_{j\eps}\big(F_{j\eps}^{-1}(u)\big) F_{j\eps}^{-1}(u) = O(u^{\beta}(1-u)^{\beta}),    
\end{equation*}
On the other hand if one rightly assumes that the conditional mean is zero 
then one does only scale adjustment (i.e.\ $\widehat{\eps}_{ji} = \frac{T_{j}(Y_{ji})}{s_{j}(\Xb_{i}; \widehat{\thetab}_{j})}$) and it is sufficient to assume 
\begin{equation*} 
 \sup_{u \in (0,1)}\frac{f_{j\eps}\big(F_{j\eps}^{-1}(u)\big)\,\big|F_{j\eps}^{-1}(u)\big|}{u^{\beta}(1-u)^{\beta}} < \infty.  
% \quad \text{ and } \quad 
%  \inf_{u \in (0,1)}\frac{f_{j\eps}\big(F_{j\eps}^{-1}(u)\big)}{u^{\beta}(1-u)^{\beta}} > 0, 
%  \qquad  f_{j\eps}\big(F_{j\eps}^{-1}(u)\big) F_{j\eps}^{-1}(u) = O(u^{\beta}(1-u)^{\beta}),    
%  f_{j\eps}\big(F_{j\eps}^{-1}(u)\big) = O(u^{\beta}(1-u)^{\beta}), 
%  \qquad  f_{j\eps}\big(F_{j\eps}^{-1}(u)\big) F_{j\eps}^{-1}(u) = O(u^{\beta}(1-u)^{\beta}),    
\end{equation*}
This last assumption is close to the assumption 2. formulated just before Theorem~2.1 
of \cite{chan2009statistical}. But similarly as when comparing with 
assumption F(iii) in  Appendix~A of \cite{einmahl2008specification}, 
our assumption does not require that the support of the distribution is a real line. 
\end{remark}

\begin{remark} \label{remark about beta equal to zero}
%
% % 
As in assumption $(\mathbf{F}_{j\eps})$ 
the function $f_{j\eps}\big(F_{j\eps}^{-1}(u)\big)$ is supposed to be monotone when 
$u$ is close to zero or close to one, then the integrability of $f_{j\eps}$ 
(see Lemma~\ref{lemma density small at infinity}) implies that 
\[
 \lim_{|x| \to \infty} |x|f_{j\eps}(x) = 0. 
\]
% 
% \boxit{Is it really true? I think it is, . 
% But on the other hand I am not sure if such a lemma should be included in the paper which is already 
% rather long.}

Thus if
\begin{equation} \label{eq: infinite quantiles}
 \lim_{u \to 0_{+}} F_{j\eps}^{-1}(u) = -\infty
 \quad \Big(\; \lim_{u \to 1_{-}} F_{j\eps}^{-1}(u) = \infty \;\Big),   
\end{equation}
then one gets 
\begin{equation} \label{eq: fjeps diminishing at borders}
 \lim_{u \to 0_{+}} f_{j\eps}\big(F_{j\eps}^{-1}(u)\big)\big(1+|F_{j\eps}^{-1}(u)|\big) = 0 
 \quad \Big(\; \lim_{u \to 1_{-}} f_{j\eps}\big(F_{j\eps}^{-1}(u)\big)\big(1+|F_{j\eps}^{-1}(u)|\big)= 0\;\Big).   
\end{equation}
Note that the above equations are also automatically satisfied if $\beta > 0$
even if \eqref{eq: infinite quantiles} does not hold. Thus one can conclude that 
if \eqref{eq: fjeps diminishing at borders} does not hold, then  $\beta = 0$ 
and the corresponding border of the support is finite, i.e., 
\[
%  \lim_{u \to 0_{+}} f_{j\eps}\big(F_{j\eps}^{-1}(u)\big)F_{j\eps}^{-1}(u) = 0
%  \text{ or }
 \lim_{u \to 0_{+}} F_{j\eps}^{-1}(u) > -\infty 
 \quad \Big(\; \lim_{u \to 1_{-}} F_{j\eps}^{-1}(u) < \infty \;\Big).  
\]
% 
% and analogously for $u \to 1_{-}$. 
%  
% % 
% As the function $f_{j\eps}\big(F_{j\eps}^{-1}(u)\big)$ when monotone if 
% $u$ is close to zero or close to one, than the integrability of $f_{j\eps}$ 
% implies that 
% % 
% \[
%  \lim_{|x| \to \infty} |x|f_{j\eps}(x) = 0. 
% \]
% % 
% Thus if $\beta = 0$, then either 
% % 
% \[
%  \lim_{u \to 0_{+}} f_{j\eps}\big(F_{j\eps}^{-1}(u)\big)F_{j\eps}^{-1}(u) = 0
%  \text{ or }
%  \lim_{u \to 0_{+}} F_{j\eps}^{-1}(u) > -\infty 
% \]
% % 
% and analogously for $u \to 1_{-}$. 
% %  
\end{remark}

% Note that assumption $(\mathbf{F}_{j\eps})$ 
% does not allow for distribution with non-continuous densities. Thus for 
% instance exponential and uniform distributions are excluded. As illustrated 
% in our simulation study (Section~\ref{sec: sim study}) the asymptotic 
% equivalence~\eqref{eq: equiv of estim of parameters} is in danger for 
% distributions with non-continuous densities. On the other for copula 
% families with nicely behaved copula density it is sufficient to assume 
% that 

\medskip 

% \noindent \textbf{Assumption} $(\widetilde{\mathbf{F}}_{j\eps})$:  \label{assumption Fjeps}
% % 
% The functions $f_{j\eps}\big(F_{j\eps}^{-1}(u)\big)$ 
% and $f_{j\eps}\big(F_{j\eps}^{-1}(u)\big)F_{j\eps}^{-1}(u)$  
% are continuous and bounded on $(0,1)$. 

\subsection{Regularity assumptions on \texorpdfstring{$m_{j}$ and $s_{j}$}{mj and sj}}

The next assumption states that the parametric models can be estimated 
at the standard $\sqrt{n}$-rate and that the location and scale functions 
are sufficiently smooth and integrable.  

\medskip 

\noindent \textbf{Assumption} $\boldsymbol{(ms)}$:  %\label{assumption mu sigma} 
For each $j \in \{1,\dotsc,d\}$ 
 $\widehat{\thetab}_{j}$ is a $\sqrt{n}$-consistent estimate
of the parameter~$\thetab_j \in \RR^{p_{j}}$. 
The functions  
$m_{j}(\xb;\tb)$ and $s_j(\xb;\tb)$ are (once) differentiable 
with respect to~$\tb$ and  the derivatives are denoted as $m'_{j}(\xb;\tb)$ and $s'_j(\xb;\tb)$.  
% satisfy 
% % 
% \[
%  \Es\Big|\tfrac{m'_{j}(\Xb; \thetab_{j}) 
%      }{s_{j}(\Xb; \thetab_{j})}\Big| < \infty, 
% \qquad 
% \Es\Big|\tfrac{s'_{j}(\Xb; \thetab_{j}) 
%      }{s_{j}(\Xb; \thetab_{j})}\Big| < \infty.  
% \]
% % 
% Finally 
Further there exists a neighborhood~$U(\thetab_{j})$  
of the true value of the parameter~$\thetab_{j}$   
such that $\inf_{\xb \in S_{\Xb}, \tb \in U(\thetab_{j})} s_j(\xb;\tb) > 0$ 
% and either: 
% % 
% \[
%  \text{(A)} \quad  \sup_{\xb \in S_{\Xb}, \tb \in U(\thetab_{j})} \|m'_j(\xb;\tb)\| < \infty, 
%  \sup_{\xb \in S_{\Xb}, \tb \in U(\thetab_{j})} \|s'_j(\xb;\tb)\| < \infty, 
%  \quad 
% \]
% % 
% or (B) 
and there exists a function $M_{j}: S_{\Xb} \to \RR$ such that  
for each $\xb \in S_{\Xb}$: 
\[
 \sup_{\tb \in U(\thetab_{j})} \big\|\tfrac{m'_j(\xb;\tb)}{s_j(\xb;\tb)}\big\| \leq M_{j}(\xb), 
\qquad 
 \sup_{\tb \in U(\thetab_{j})} \big\|\tfrac{s'_j(\xb;\tb)}{s_j(\xb;\tb)}\big\| \leq M_{j}(\xb), 
\]
and  $ \E \big[M_{j}(\Xb)\big]^{r} < \infty$ for some $r \geq 2$. 
% \[
%   \E \big[M_{j}(\Xb_{1})\big]^2 < \infty, 
%  \quad \E \big[S_{j}(\Xb)\big]^2 < \infty.    
% \] 
% 
Finally, for each $K > 0$ the derivatives $m'_j(\xb;\tb)$ 
and $s'_j(\xb;\tb)$ viewed as functions of $\tb$ 
are continuous at $\thetab_{j}$ uniformly in 
$\xb \in \{\tilde{\xb} \in S_{\Xb}: \|\tilde{\xb}\| \leq K\}$.

\subsection{\texorpdfstring{Regularity assumptions about the copula family~$\mathcal{C}$}{Regularity assumptions about the copula family}}

To formulate the main regularity assumptions about the copula family 
it is useful to introduce the following 
set of functions. 
% Now I give an alternative set of assumptions that seem to be easier to work with. 
% These assumptions are distilled from \cite{chan2009statistical}.  

\begin{definition*}[Class of $\mathcal{J}$- and 
% $\widetilde{\mathcal{J}}^{\beta_1, \beta_2}$-functions]
$\widetilde{\mathcal{J}}^{\beta_1, \beta_2}$-functions] 
%  \label{def J1 and J2 fctions} 
A~function $\varphi:(0,1)^{d} \to \RR$ is called a~\emph{$\mathcal{J}$-function} 
if $\varphi$ is continuous on $(0,1)^d$ and there exist $\eta \in [0,1)$ and a finite 
constant $M_{1}$ such that for all $\ub \in (0,1)^{d}$
\begin{equation*} %\label{eq: bound  on J}
% \left|\varphi(u_1,\dotsc,u_d)\right| \leq \frac{M_1}{\big[\min_{j=1,\dotsc,d}{\min\{u_j, 1-u_j\}}\big]^{\eta}}\,. 
\left|\varphi(u_1,\dotsc,u_d)\right| \leq \sum_{j=1}^{d}\frac{M_1}{\big[{\min\{u_j, 1-u_j\}}\big]^{\eta}}\,. 
\end{equation*}

Let $\beta_1 \in [0, 1/2)$ and $\beta_{2} \geq 0$ be fixed. We say that 
a~function $\varphi:(0,1)^{d} \to \RR$ is a~\emph{$\widetilde{\mathcal{J}}^{\beta_1, \beta_2}$-function} if it is continuous on $(0,1)^{d}$ and there exists a finite constant $M_{2}$ such that for all $\ub \in (0,1)^{d}$
\begin{equation*}
\left|\varphi(u_1,\dotsc,u_d)\right| \leq 
% \frac{M_2}{\big[\min_{j=1,\dotsc,d}{\min\{u_{j},1-u_{j}\}}\big]^{\beta_{1}}}\, .
\sum_{j=1}^{d}\frac{M_2}{\big[{\min\{u_{j},1-u_{j}\}}\big]^{\beta_{1}}}\, .    
\end{equation*}
Further $\left|\varphi^{(j)}(u_1,\dotsc,u_d)\right|\,u_{j}^{\beta_{2}}(1-u_{j})^{\beta_{2}}$ is a $\mathcal{J}$-function 
for all $j \in \{1,\dotsc,d\}$, where 
\[
 \varphi^{(j)}(u_1,\dotsc,u_d) = \frac{\partial \varphi(u_1,\dotsc,u_d)}{\partial u_j}.  
\]
\end{definition*}

Now we are ready to formulate the needed regularity assumptions about the copula family. 
Recall that $\Theta \subset \RR^{p}$, $\alfab$ is the true value of the parameter, and  $c(\ub;\ab)$ is a density corresponding to the copula 
function $C(\ub;\ab)$. 

\bigskip

\noindent \textbf{Assumptions C:} 
\nopagebreak 
%Recall that $c(\ub;\ab)$ is a density corresponding to the copula 
%function $C(\ub;\ab)$. 
\vspace{-2mm}
\begin{assumpC}\label{assump:identifiability}
$c(\ub;\ab_1) = c(\ub;\ab_2)$ for almost all $\ub \in (0,1)^{d}$ 
only if $\ab_1 = \ab_2$.  
%  The function $R(\ab) = \Es \{\rho(\Ub;\ab)\}$
%  has a unique maximizer~$\widetilde{\alfab}$.
\end{assumpC}

\begin{assumpC}\label{assump:smoothness rho}
 The function %$\rho(\ub;\ab ) = 
$\log\{c(\ub; \ab)\}$ 
 is continuously differentiable with respect to $\ab$ 
 for all $\ub \in (0,1)^{d}$.
\end{assumpC}

\noindent 
%To formulate the next assumption 
Denote the $k$th element 
% ($k \in \{1,\dotsc,p\}$) 
of the vector function $\skorb(\ub;\ab) = {\partial \log \{ c(\ub; \ab)\}}/{\partial \ab} $ by $\psi_{k}(\ub;\ab)$.

\begin{assumpC}\label{assump:snice}
 For each $k \in \{ 1,\dotsc,p \}$, the function $\psi_{k}(\cdot;\alfab) \in \widetilde{\mathcal{J}}^{\beta_1,\beta_2}$, where $\beta > \max\{\beta_1 + \tfrac{1}{r-1}, \beta_{2}\}$, 
for $\beta$ introduced in  
assumption~$(\mathbf{F}_{j\eps})$ and $r$ in assumption~$\boldsymbol{(ms)}$.  
% satisfy the assumptions of Lemma~\ref{lemma: as normality of functions of ranks}. 
\end{assumpC}

% \noindent \textbf{Assumptions I:}

% \vspace*{-2mm}

\begin{assumpC}\label{assump:differentiability}
The function $\skorb(\ub; \ab)$ is assumed to be continuously differentiable
with respect to~$\ab$ for all $\ub \in (0,1)^{d}$.  
Further there exist an  open neighborhood~$\mathcal{U} \subset \Theta$ 
of $\alfab$ and  a dominating function~$h(\ub) \in \mathcal{J}$ such that
 $\partial\skorb(\ub; \ab)/\partial\ab\tr$ is continuous
 in $(0,1)^d \times \mathcal{U}$  and 
\[
 \max_{k,\ell \in \{1,\dotsc,p\}}\, \sup_{\ab \in \mathcal{U}}
  \big|\tfrac{\partial \psi_{k}(\ub; \ab)}{\partial a_{\ell}}\big|
 \leq h(\ub).
\]
\end{assumpC}

\begin{assumpC}\label{assump:invertibility}
The $p\times p$ (Fisher information) matrix 
$I(\alfab) 
 = -\Es \big\{ \partial\skorb(\Ub;\ab)/{\partial\ab \tr}\big|_{\ab = \alfab}\big\}$, where 
 \begin{equation*} %\label{eq: Ub}
 \Ub = \big(U_{1},\dotsc,U_{d}\big) \tr 
 = \big(F_{1\eps}(\eps_{1}),\dotsc,F_{d\eps}(\eps_{d})) \tr,   
%  \Ub_{i} = \big(U_{1i},\dotsc,U_{di}\big) \tr 
%  = \big(F_{1\eps}(\eps_{1i}),\dotsc,F_{d\eps}(\eps_{di})) \tr,   
 \end{equation*}
 is finite and nonsingular.
\end{assumpC}

\begin{remark} \label{remark about beta1 and beta2}
Note that the score functions of the commonly used one-parameter 
bivariate copula families with unbounded densities (e.g.\ Clayton, Gumbel, Normal, Student, \ldots) 
can be bounded by 
\[
 |\psi(u_1,u_2; a)| \leq  M_{3} \sum_{j=1}^{2} \big|\log(u_{j})+\log(1-u_{j})\big| 
\]
and its derivative as 
\[
 |\psi^{(j)}(u_1,u_2; a)| \leq  \frac{M_3}{\big[{\min\{u_j, 1-u_j\}}\big]} 
  + M_{3}\sum_{j'=1}^{2} \big|\log(u_{j'})+\log(1-u_{j'})\big|, \quad j=1,2 
\]
for a sufficiently large but finite constant~$M_{3}$ \citep[see also][]{chen2006estimation_ts}. 
Thus in Assumption~\ref{assump:snice} one can consider  $\beta_1$ and $\beta_2$ arbitrarily 
close to zero but positive. 
% 
 
% \boxit{The question is if one should work with the }

% 
% It might be interesting to explore the interplay between the copula family 
% and marginal distributions in more detail. Some specific questions are as follows: 
% % 
% \begin{itemize}
%  \item Although the normal copula is unbounded it 
% seems to be that it is sufficient to assume that $f_{j\eps}$ 
% is bounded (e.g. exponential or uniform distribution seems to be fine). This is 
% probably thanks to the fact that the tail dependence for this copula family is zero 
% and \textbf{it requires a different proof}. 
% % 
% \item Suppose that for a given copula family the score function $\skorb$ and 
% all its derivatives used to formulate regularity assumptions are well behaved 
% (e.g. bounded). Can this allow for non-continuous (or even unbounded) marginal densities $f_{j\eps}$? \textbf{Unfortunately, again a different proof is needed} At least, 
% as far as I see at this moment. 
% % 
% \end{itemize}
% 
\end{remark}

% Assumptions \ref{assump:snice} and \eqref{assump:differentiability} are formulated 
% in order to cover the copula families with unbounded densities. 
Assumption~\ref{assump:snice} is inspired by  \cite{chan2009statistical}. 
Note that generally speaking this assumption is more strict 
than the corresponding assumptions of \cite{tsukahara_2005} 
that are based on $U$-shaped functions. The advantage 
of assumption~\ref{assump:snice} is that 
it enables to derive bounds that depend only on the marginal 
distributions. The price that we pay for this advantage does 
not seem to be big because we are not aware of a standard 
copula family that does not meet \ref{assump:snice} with 
$\beta_{1}$ and $\beta_{2}$ arbitrarily small positive constants. 

\medskip 

Note that assumption~\ref{assump:snice} implies that $\beta > 0$, which does not allow 
for marginal densities~$f_{j\eps}$ that are bounded but possibly 
discontinuous at a border point (e.g.\ exponential 
or uniform distributions). As shown in simulations in Section~\ref{sec: sim study} 
the aimed result~\eqref{eq: equiv of estim of parameters} indeed does not hold 
in general when the marginal densities~$f_{j\eps}$ are not continuous. 

Nevertheless a  closer inspection of the proof shows that %the assumption 
$\beta >0$ is 
needed to get a control over a possibly unbounded 
score function $\skorb(\ub;\ab)$.  But there are commonly used copula 
families (e.g.\ Frank, Ali-Mikhail-Haq, Plackett) for which  
the score function $\skorb(\ub;\ab)$ and its derivatives are bounded. It is of interest  
to formulate an alternative to assumptions~\ref{assump:snice} and \ref{assump:differentiability}  separately as it allows for 
$\beta = 0$ in assumption~$(\mathbf{F}_{j\eps})$,  

\begin{assumpC}\label{assump:very nice}  
The function $\skorb(\ub; \ab)$ is bounded and continuously differentiable
with respect to~$\ab$ for all $\ub \in (0,1)^{d}$.  
Further there exists an open neighborhood~$\mathcal{U}$
of $\alfab$ such that
 $\partial\skorb(\ub; \ab)/\partial\ab\tr$ is continuous
 in $(0,1)^d \times \mathcal{U}$  and 
\[
 \max_{k,\ell \in \{1,\dotsc,p\}}\, \sup_{\ab \in \mathcal{U}}
  \sup_{\ub \in (0,1)^d}\big|\tfrac{\partial \psi_{k}(\ub; \ab)}{\partial a_{\ell}}\big|
 < \infty 
\quad \text{and} \quad 
 \max_{j \in \{1,\dotsc,d\}} \max_{k \in \{1,\dotsc,p\}}\,  \sup_{\ub \in (0,1)^d}\big|\tfrac{\partial \psi_{k}(\ub; \alfab)}{\partial u_{j}}\big|
 < \infty.  
\]
\end{assumpC}

% \noindent \textbf{Assumptions I:}

% \vspace*{-2mm}

% \boxit{We should discuss in  a more detail 
% how the assumptions and results are related with the 
% existing literature.}

\smallskip

\subsection{Main results}

Now we are ready to formulate the main results of the paper. 

% \boxit{Unfortunately, as it is common in the likelihood theory 
%  one needs to speak about 
% appropriately chosen roots of the estimating equations 
% to be rigorous.}

\begin{theorem} \label{thm equiv of MPL estim}
 Suppose that assumptions  $\boldsymbol{(ms)}$,  
\ref{assump:identifiability}-\ref{assump:invertibility} 
and $(\mathbf{F}_{j\eps})$ with $\beta > 0$ are satisfied. 
Then with probability going to one there exist consistent roots 
(say $\widehat{\alfab}_{n}$ and $\widetilde{\alfab}_{n}$)
of the estimating equations~\eqref{eq: MPL based on residuals} 
and~\eqref{eq: MPL based on errors}. 
% Further denote 
%  the consistent 
% roots of \eqref{eq: MPL based on residuals} and~\eqref{eq: MPL based on errors}
% respectively. 
Further $\widehat{\alfab}_{n}$ and $\widetilde{\alfab}_{n}$ 
satisfy \eqref{eq: equiv of estim of parameters}. % 
\end{theorem}

The next theorem say that if assumption \ref{assump:very nice} is satisfied then 
one can also include the case $\beta = 0$ in assumption $(\mathbf{F}_{j\eps})$.   
Thus for instance if one (rightly) assumes that $C$ is a Frank copula 
then the marginal distributions of the errors are allowed to be also uniform or 
exponential. 

% it is sufficient to assume $(\mathbf{\tilde{F}}_{j\eps})$ instead of $(\mathbf{F}_{j\eps})$. 

\begin{theorem} \label{thm equiv of MPL estim for nice copulas}
 Suppose that assumptions  
% $(\mathbf{\tilde{F}}_{j\eps})$, 
$\boldsymbol{(ms)}$,   
\ref{assump:identifiability}, \ref{assump:smoothness rho}, 
\ref{assump:invertibility}, \ref{assump:very nice} and 
$(\mathbf{F}_{j\eps})$ are satisfied. 
Then the statement of Theorem~\ref{thm equiv of MPL estim} holds. 
\end{theorem}

The above theorems imply that when fitting the copula~$C$ one 
can (under the stated assumptions) ignore the fact that he/she 
is working with estimated residuals ($\widehat{\eps}_{ij}$) instead of unobserved 
errors ($\eps_{ij}$). As it is known (and it also follows from the 
proof of Theorem~\ref{thm equiv of MPL estim}) the asymptotic 
distribution of $\widetilde{\alfab}_{n}$ is normal. Thus thanks 
to \eqref{eq: equiv of estim of parameters} one can conclude that 
also $\widehat{\alfab}_{n}$ is asymptotically normal.

\begin{corollary} \label{cor as normality of mpl}
 Suppose that the assumptions either of Theorem~\ref{thm equiv of MPL estim} 
or~\ref{thm equiv of MPL estim for nice copulas} hold. Then with 
probability going to one there exists a consistent 
root~$\widehat{\alfab}_{n}$ 
of~\eqref{eq: MPL based on residuals}. This root satisfies 
\[
 \sqrt{n} \big(\widehat{\alfab}_{n} - \alfab\big) 
 \inDist \mathsf{N}_{p}(\boldsymbol{0}, \Sigma),
 \quad \Sigma =  I^{-1}(\alfab)\,
  \var\big(\widetilde{\boldsymbol{\psi}}\big(\Ub) \big)
 \, I^{-1}(\alfab), 
\]
where $\widetilde{\boldsymbol{\psi}}\big(\ub) = \big(\widetilde{\psi}_{1}(\ub),\dotsc,\widetilde{\psi}_{p}(\ub)\big)\tr$ with 
\begin{equation} \label{eq: elements of psi}
 \widetilde{\psi}_{k}(\ub) =  \psi_{k}(\ub; \alfab) 
  + \sum_{j=1}^{d} \int_{[0,1]^{d}} \big[\ind\{u_{j} \leq v_j\}-v_{j}\big] \psi_{k}^{(j)}(\vb; \alfab)\,\mathrm{d} C(\vb), \quad k=1,\dotsc,p. 
\end{equation}

\end{corollary}

% \boxit{As usual, the assumptions in the theorems are sufficient but not necessary. 
% Thus it might happen that in fact less strict assumptions are necessary. 
% For instance my limited simulation experience suggests that although the 
% normal copula is unbounded it is sufficient that 
% assumptions $(\mathbf{F}_{j\eps})$ holds with $\beta=0$. 
% % The simulation study suggest that Although the normal copula is unbounded it 
% % seems to be that it is sufficient to assume that $f_{j\eps}$ 
% % is bounded (e.g. exponential or uniform distribution seems to be fine). This is 
% % probably thanks to the fact that the tail dependence for this copula family is zero 
% % and \textbf{it requires a different proof}. 
% % % 
% }

\section{Simulation study} \label{sec: sim study} 
A Monte Carlo study was conducted in order to illustrate  the theoretical conclusions and to show 
%We were in particular interested 
how the finite sample  
performance of the maximum pseudo-likelihood estimator depends on the 
level of violation of the regularity assumptions. 
% in situations when 
% either the regularity assumptions are satisfied or are not satisfied. 
%  the theoretical conclusions. The presented results show that 
% if the regularity assumptions are violated, then the maximum pseudo-likelihood estimation may  fail. 

\subsection{Settings}
To keep the presentation as clear as possible we concentrate on a bivariate 
response variable  (some results for a three-dimensional case can be found in the Supplementary material) 
following the model
\begin{equation}
\label{mod1}
Y_{1i}=\theta_{10} + \theta_{11} X_i +\eps_{1i}, \qquad 
Y_{2i}=\theta_{20}+\theta_{21} X_i+\eps_{2i}, \qquad i=1,\dotsc,n.
\end{equation}
% 
%  following simple model was considered
% \begin{align}
% Y_{1i}&=\theta_{10} + \theta_{11} X_i +\eps_{1i}, \label{mod1}\\
% Y_{2i}&=\theta_{20}+\theta_{21} X_i+\eps_{2i}, \notag
% \end{align}
% $i=1,\dots,n$. 
The joint cumulative distribution function $H(y_1,y_2)$ of the random vector $(\eps_{1i},\eps_{2i})^\top$ is $C\big(F_{1\eps}(y_1),F_{2\eps}(y_2)\big)$, where $C$ is a copula
and $F_{1\eps}$, $F_{2\eps}$ are marginal distribution functions. The following five copula families were considered for $C$: Clayton, Frank, Gumbel, Gaussian, and Student with 5 degrees of freedom. The  copula parameter $\alpha$ is chosen such that the corresponding Kendall's tau is $\tau=0.5$ or $\tau=0.75$. The marginal distributions were chosen one of the following: 
\begin{itemize}
\item[$-$] $F_{1\eps}$ is standard normal and $F_{2\eps}$ exponential with mean 1 (denoted as N+E), 
\item[$-$]  $F_{1\eps}$ is standard normal and $F_{2\eps}$ uniform on $[-1,1]$ (denoted as N+U), 
\item[$-$]  $F_{1\eps}$ and  $F_{2\eps}$ are both Student $t$ with 5 degrees of freedeom (denoted as t).
\end{itemize} 
The first two situations satisfy the  assumption $(\mathbf{F}_{j\eps})$ only with $\beta=0$. Hence, the result of Theorem~\ref{thm equiv of MPL estim for nice copulas}  applies only if 
\eqref{assump:very nice} holds. 
% , i.e.  if the score function $\m$ and its derivative are bounded. 
From the five considered copula families, this is the case only for the Frank copula. On the other hand, the $t$ marginals satisfy $(\mathbf{F}_{j\eps})$ with $\beta>0$ and the assumptions of Theorem~1 hold. Hence, these marginals provide a useful regular benchmark  for a comparison with the first two situations. 

The covariate $X_i$ is generated  from the standard normal distribution 
(Poisson distribution with mean 5 was considered as well, but the results are 
almost identical and are not reported).  
The presented results correspond to the particular choice $\theta_{10}=1$, $\theta_{20}=-1$, $\theta_{11}=1$, and $\theta_{21}=2$. 
The unobserved errors $\eps_{ji}$ are estimated as the residuals after fitting the regression lines (marginally) where the parameters are estimated with the help of the least squares method assuming $s_j\equiv 1$, $j=1,2$, cf. Remark~\ref{remark1}. 

%
%
%
%For comparison,  results for all the five copula families combined with two normal marginals or two $t$ marginals with $5$ degrees of freedom  (both satisfying the  assumption $(\mathbf{F}_{j\eps})$ with $\beta>0$)  are presented  in Tables  \ref{tab.2dim.clayton.regular}, \ref{tab.2dim.frank.regular}, \ref{tab.2dim.gumbel.regular}, \ref{tab.2dim.normal.regular}, \ref{tab.2dim.student.regular}. Even though the estimator based on Kendall's tau inversion can be preferable for small sample sizes even for these marginals, the maximum pseudo-likelihood estimator performs slightly better as $n$ growths.  
% 

The following estimators of the parameter $\alpha$ are compared:
\begin{enumerate}
\renewcommand{\labelenumi}{(\roman{enumi})}
\item (oracle) inversion of Kendall's tau based on the unobserved errors $\widetilde{\alpha}^{(ik)}$;
\item inversion of Kendall's tau based on the residuals $\widehat{\alpha}^{(ik)}$;
\item (oracle) maximum pseudo-likelihood estimator  based on the unobserved errors   $\widetilde{\alpha}^{(pl)}$;
\item maximum pseudo-likelihood method estimator on the residuals  $\widehat{\alpha}^{(pl)}$;
\item modified maximum pseudo-likelihood estimator based on the residuals  $\widehat{\alpha}^{(pl*)}$. 
\end{enumerate} 
% 
% \vspace*{-2mm}
The latter estimator $\widehat{\alpha}^{(pl*)}$ is inspired by the estimator introduced in the context of single index conditional copulas by \cite{fermanian2018single}. In our situation 
this estimator coincides with the maximum pseudo-likelihood estimator  computed only from $\widehat{\mathbf{U}}_i$ which lie in $[\delta_n,1-\delta_n]^2$, where $\delta_n=D n^{-1/\lambda}$. 
Note that this choice corresponds to the choice $\delta_{n}$ in the proof 
of Theorem~\ref{thm equiv of MPL estim}. 
In the presented simulations we choose $D=1/4$ and $\lambda=1.9$, thus 
in view of Remark~\ref{remark about beta1 and beta2} the statement of 
Theorem~\ref{thm equiv of MPL estim} (or~\ref{thm equiv of MPL estim for nice copulas})  
holds also for $\widehat{\alpha}^{(pl*)}$ provided that the corresponding regularity 
assumptions hold.

%$\widetilde{\alpha}^{(ml*)}$

In order to have more comparable results for the various copula families, the estimates of the parameters are presented on the Kendall's tau scale. The performance of the estimators is measured by the bias, the standard error (SD), and the root mean square error (RMSE), which are estimated from $1\,000$ random samples of sample sizes   $n=100,\ 1\,000,\ 10\,000$ and whose  100 multiplies are reported,  because 
the obtained quantities are typically of order $10^{-2}$. 
The obtained results for Clayton, Frank and Gaussian copulas are listed in Tables \ref{tab.2dim.clayton}, \ref{tab.2dim.frank}, and \ref{tab.2dim.normal}, while tables for Gumbel and Student copula can be found in the Supplementary material. 
%Remark that the  100 multiplies of bias, SD and RMSE are reported, because 
%the obtained quantities are typically of order $10^{$-$2}$. 
The Monte Carlo simulations were run in R statistical computing environment
\citep{R_2018}. The same starting seed was always used so that the estimates based on the true 
(but unobserved) errors~$\eps_{ij}$ are the same regardless the choice of the marginals $F_{1\eps}$ and $F_{2\eps}$.  These `oracle' estimates are denoted as ``inov'' in the tables and provide 
benchmarks for the estimators calculated from the estimated residuals.

% latex table generated in R 3.4.4 by xtable 1.8$-$2 package
% Fri Feb 22 20:18:53 2019
\begin{table}[ht]
\centering
\begin{footnotesize}
\begin{tabular}{ccc|rrr|rrr|rrr}
  \hline
\hline
   $\tau$& margins & estim& \multicolumn{3}{c|}{$n=100$}&\multicolumn{3}{c|}{$n=1\,000$}& \multicolumn{3}{c}{$n=10\,000$}\\
   && & bias & SD &RMSE& bias & SD &RMSE& bias & SD &RMSE \\
 \hline
0.50 & inov & $\widetilde{\alpha}^{(ik)}$ & $-$0.03 & 5.54 & 5.54 & 0.00 & 1.64 & 1.64 & $-$0.01 & 0.53 & 0.53 \\ 
   &  & $\widetilde{\alpha}^{(pl)}$ & 0.33 & 4.90 & 4.91 & 0.01 & 1.49 & 1.49 & 0.00 & 0.48 & 0.48 \\ 
   \cline{2-12}
 & N+E & $\widehat{\alpha}^{(ik)}$ & $-$1.25 & 5.62 & 5.76 & $-$0.27 & 1.64 & 1.67 & $-$0.05 & 0.53 & 0.53 \\ 
   &  & $\widehat{\alpha}^{(pl)}$ & $-$3.91 & 5.54 & 6.78 & $-$2.26 & 2.08 & 3.08 & $-$0.80 & 0.75 & 1.10 \\ 
   &  & $\widehat{\alpha}^{(pl*)}$ & $-$1.94 & 5.30 & 5.65 & $-$1.23 & 1.81 & 2.19 & $-$0.44 & 0.63 & 0.77 \\ 
%    \hline
\cline{2-12}
 & N+U & $\widehat{\alpha}^{(ik)}$ & $-$0.21 & 5.55 & 5.55 & $-$0.03 & 1.63 & 1.63 & $-$0.02 & 0.53 & 0.53 \\ 
   &  & $\widehat{\alpha}^{(pl)}$ & $-$0.84 & 4.86 & 4.93 & $-$0.61 & 1.53 & 1.65 & $-$0.22 & 0.51 & 0.55 \\ 
   &  & $\widehat{\alpha}^{(pl*)}$ & 0.02 & 5.00 & 5.00 & $-$0.13 & 1.50 & 1.51 & $-$0.05 & 0.49 & 0.49 \\ 
%    \hline
\cline{2-12}
 & t & $\widehat{\alpha}^{(ik)}$ & $-$0.15 & 5.58 & 5.58 & $-$0.01 & 1.64 & 1.64 & $-$0.02 & 0.53 & 0.53 \\ 
   &  & $\widehat{\alpha}^{(pl)}$ & 0.10 & 4.96 & 4.96 & $-$0.02 & 1.50 & 1.50 & $-$0.01 & 0.48 & 0.48 \\ 
   &  & $\widehat{\alpha}^{(pl*)}$ & 0.38 & 5.05 & 5.06 & 0.06 & 1.51 & 1.51 & 0.02 & 0.48 & 0.48 \\ 
   \hline
0.75 & inov & $\widetilde{\alpha}^{(ik)}$ & 0.02 & 3.40 & 3.40 & $-$0.01 & 1.01 & 1.01 & 0.01 & 0.31 & 0.31 \\ 
   &  & $\widetilde{\alpha}^{(pl)}$ & $-$0.77 & 3.12 & 3.21 & $-$0.16 & 0.93 & 0.94 & $-$0.01 & 0.28 & 0.28 \\ 
%    \hline
\cline{2-12}
 & N+E & $\widehat{\alpha}^{(ik)}$ & $-$2.14 & 3.70 & 4.27 & $-$0.48 & 1.08 & 1.18 & $-$0.07 & 0.32 & 0.33 \\ 
   &  & $\widehat{\alpha}^{(pl)}$ & $-$9.19 & 5.85 & 10.89 & $-$4.19 & 2.88 & 5.09 & $-$1.57 & 1.14 & 1.94 \\ 
   &  & $\widehat{\alpha}^{(pl*)}$ & $-$6.26 & 4.95 & 7.98 & $-$2.86 & 2.36 & 3.71 & $-$1.07 & 0.94 & 1.43 \\ 
%    \hline
\cline{2-12}
 & N+U & $\widehat{\alpha}^{(ik)}$ & $-$0.24 & 3.39 & 3.40 & $-$0.06 & 1.01 & 1.01 & 0.00 & 0.31 & 0.31 \\ 
   &  & $\widehat{\alpha}^{(pl)}$ & $-$2.99 & 3.27 & 4.43 & $-$1.22 & 1.18 & 1.70 & $-$0.44 & 0.41 & 0.60 \\ 
   &  & $\widehat{\alpha}^{(pl*)}$ & $-$1.63 & 3.15 & 3.55 & $-$0.60 & 1.01 & 1.17 & $-$0.20 & 0.33 & 0.39 \\ 
%    \hline
\cline{2-12}
 & t & $\widehat{\alpha}^{(ik)}$ & $-$0.22 & 3.45 & 3.45 & $-$0.05 & 1.01 & 1.01 & 0.01 & 0.31 & 0.31 \\ 
   &  & $\widehat{\alpha}^{(pl)}$ & $-$1.21 & 3.21 & 3.43 & $-$0.22 & 0.93 & 0.95 & $-$0.02 & 0.28 & 0.28 \\ 
   &  & $\widehat{\alpha}^{(pl*)}$ & $-$1.04 & 3.24 & 3.40 & $-$0.17 & 0.93 & 0.95 & $-$0.01 & 0.28 & 0.28 \\ 
   \hline
\hline
\end{tabular}
\caption{Model \eqref{mod1} with  Clayton copula, quantities multiplied by 100.} 
\label{tab.2dim.clayton}
\end{footnotesize}
\end{table}
 
 \begin{table}[ht]
\centering
\begin{footnotesize}
\begin{tabular}{ccc|rrr|rrr|rrr}
  \hline
\hline
   $\tau$& margins & estim& \multicolumn{3}{c|}{$n=100$}&\multicolumn{3}{c|}{$n=1\,000$}& \multicolumn{3}{c}{$n=10\,000$}\\
   && & bias & SD &RMSE& bias & SD &RMSE& bias & SD &RMSE \\
%  \hline
\cline{2-12}
0.50 & inov & $\widetilde{\alpha}^{(ik)}$ & $-$0.03 & 4.62 & 4.62 & 0.01 & 1.44 & 1.43 & 0.01 & 0.45 & 0.45 \\ 
   &  & $\widetilde{\alpha}^{(pl)}$ & $-$0.03 & 4.51 & 4.50 & 0.01 & 1.42 & 1.42 & 0.01 & 0.45 & 0.45 \\ 
%    \hline
\cline{2-12}
 & N+E & $\widehat{\alpha}^{(ik)}$ & $-$0.45 & 4.68 & 4.70 & $-$0.05 & 1.44 & 1.44 & 0.00 & 0.45 & 0.45 \\ 
   &  & $\widehat{\alpha}^{(pl)}$ & $-$0.45 & 4.55 & 4.57 & $-$0.05 & 1.43 & 1.43 & 0.00 & 0.45 & 0.45 \\ 
   &  & $\widehat{\alpha}^{(pl*)}$ & $-$0.21 & 4.84 & 4.84 & $-$0.04 & 1.46 & 1.46 & 0.00 & 0.45 & 0.45 \\ 
%    \hline
\cline{2-12}
 & N+U & $\widehat{\alpha}^{(ik)}$ & $-$0.08 & 4.65 & 4.65 & 0.00 & 1.44 & 1.43 & 0.00 & 0.45 & 0.45 \\ 
   &  & $\widehat{\alpha}^{(pl)}$ & $-$0.08 & 4.53 & 4.53 & 0.00 & 1.42 & 1.42 & 0.01 & 0.45 & 0.45 \\ 
   &  & $\widehat{\alpha}^{(pl*)}$ & 0.09 & 4.85 & 4.85 & 0.01 & 1.45 & 1.45 & 0.01 & 0.45 & 0.45 \\ 
   \hline
0.75 & inov & $\widetilde{\alpha}^{(ik)}$ & $-$0.11 & 2.50 & 2.50 & 0.00 & 0.74 & 0.74 & 0.00 & 0.23 & 0.23 \\ 
   &  & $\widetilde{\alpha}^{(pl)}$ & $-$0.53 & 2.45 & 2.50 & $-$0.06 & 0.74 & 0.74 & 0.00 & 0.23 & 0.22 \\ 
%    \hline
\cline{2-12}
 & N+E & $\widehat{\alpha}^{(ik)}$ & $-$1.17 & 2.79 & 3.02 & $-$0.14 & 0.76 & 0.77 & $-$0.01 & 0.23 & 0.23 \\ 
   &  & $\widehat{\alpha}^{(pl)}$ & $-$1.59 & 2.77 & 3.19 & $-$0.19 & 0.76 & 0.78 & $-$0.02 & 0.23 & 0.23 \\ 
   &  & $\widehat{\alpha}^{(pl*)}$ & $-$1.42 & 2.90 & 3.23 & $-$0.17 & 0.77 & 0.79 & $-$0.01 & 0.23 & 0.23 \\ 
%    \hline
\cline{2-12}
 & N+U & $\widehat{\alpha}^{(ik)}$ & $-$0.25 & 2.53 & 2.54 & $-$0.01 & 0.74 & 0.74 & 0.00 & 0.23 & 0.23 \\ 
   &  & $\widehat{\alpha}^{(pl)}$ & $-$0.69 & 2.50 & 2.59 & $-$0.07 & 0.74 & 0.74 & 0.00 & 0.23 & 0.23 \\ 
   &  & $\widehat{\alpha}^{(pl*)}$ & $-$0.57 & 2.62 & 2.68 & $-$0.05 & 0.76 & 0.76 & 0.00 & 0.23 & 0.23 \\ 
   \hline
\hline 
\end{tabular}
\caption{Model \eqref{mod1} with Frank copula, quantities multiplied by 100.} 
\label{tab.2dim.frank}
\end{footnotesize}
\end{table}

\begin{table}[ht]
\centering
\begin{footnotesize}
\begin{tabular}{ccc|rrr|rrr|rrr}
  \hline
\hline
   $\tau$& margins & estim& \multicolumn{3}{c|}{$n=100$}&\multicolumn{3}{c|}{$n=1\,000$}& \multicolumn{3}{c}{$n=10\,000$}\\
   && & bias & SD &RMSE& bias & SD &RMSE& bias & SD &RMSE \\
 \hline
0.50 & inov & $\widetilde{\alpha}^{(ik)}$ & 0.03 & 4.94 & 4.94 & $-$0.07 & 1.53 & 1.53 & 0.00 & 0.48 & 0.48 \\ 
   &  & $\widetilde{\alpha}^{(pl)}$ & 1.07 & 4.51 & 4.63 & 0.10 & 1.39 & 1.40 & 0.03 & 0.44 & 0.44 \\ 
%    \hline
\cline{2-12}
 & N+E & $\widehat{\alpha}^{(ik)}$ & $-$0.43 & 4.97 & 4.99 & $-$0.17 & 1.53 & 1.54 & $-$0.02 & 0.48 & 0.48 \\ 
   &  & $\widehat{\alpha}^{(pl)}$ & 0.32 & 4.54 & 4.55 & $-$0.21 & 1.41 & 1.43 & $-$0.06 & 0.45 & 0.45 \\ 
   &  & $\widehat{\alpha}^{(pl*)}$ & 0.99 & 4.90 & 5.00 & 0.08 & 1.46 & 1.46 & 0.03 & 0.45 & 0.45 \\ 
%    \hline
\cline{2-12}
 & N+U & $\widehat{\alpha}^{(ik)}$ & $-$0.06 & 4.97 & 4.97 & $-$0.08 & 1.53 & 1.53 & 0.00 & 0.48 & 0.48 \\ 
   &  & $\widehat{\alpha}^{(pl)}$ & 0.87 & 4.53 & 4.62 & $-$0.01 & 1.40 & 1.39 & $-$0.01 & 0.44 & 0.44 \\ 
   &  & $\widehat{\alpha}^{(pl*)}$ & 1.36 & 4.86 & 5.05 & 0.22 & 1.46 & 1.47 & 0.07 & 0.45 & 0.45 \\ 
%    \hline
\cline{2-12}
 & t & $\widehat{\alpha}^{(ik)}$ & 0.04 & 4.99 & 4.98 & $-$0.07 & 1.53 & 1.53 & 0.00 & 0.48 & 0.48 \\ 
   &  & $\widehat{\alpha}^{(pl)}$ & 1.08 & 4.55 & 4.67 & 0.09 & 1.40 & 1.40 & 0.02 & 0.44 & 0.44 \\ 
   &  & $\widehat{\alpha}^{(pl*)}$ & 1.42 & 4.90 & 5.10 & 0.21 & 1.46 & 1.48 & 0.06 & 0.45 & 0.45 \\ 
   \hline
0.75 & inov & $\widetilde{\alpha}^{(ik)}$ & 0.16 & 2.79 & 2.80 & $-$0.02 & 0.89 & 0.89 & 0.00 & 0.27 & 0.27 \\ 
   &  & $\widetilde{\alpha}^{(pl)}$ & 0.06 & 2.53 & 2.53 & $-$0.02 & 0.80 & 0.80 & 0.00 & 0.25 & 0.25 \\ 
%    \hline
\cline{2-12}
 & N+E & $\widehat{\alpha}^{(ik)}$ & $-$1.02 & 2.93 & 3.10 & $-$0.24 & 0.90 & 0.93 & $-$0.04 & 0.27 & 0.27 \\ 
   &  & $\widehat{\alpha}^{(pl)}$ & $-$1.81 & 2.81 & 3.34 & $-$0.73 & 0.95 & 1.20 & $-$0.21 & 0.30 & 0.37 \\ 
   &  & $\widehat{\alpha}^{(pl*)}$ & $-$1.01 & 2.77 & 2.95 & $-$0.40 & 0.88 & 0.97 & $-$0.10 & 0.27 & 0.29 \\ 
%    \hline
\cline{2-12}
 & N+U & $\widehat{\alpha}^{(ik)}$ & $-$0.08 & 2.80 & 2.80 & $-$0.05 & 0.89 & 0.89 & $-$0.01 & 0.27 & 0.27 \\ 
   &  & $\widehat{\alpha}^{(pl)}$ & $-$0.48 & 2.52 & 2.56 & $-$0.27 & 0.82 & 0.86 & $-$0.09 & 0.25 & 0.27 \\ 
   &  & $\widehat{\alpha}^{(pl*)}$ & 0.00 & 2.61 & 2.60 & $-$0.05 & 0.81 & 0.81 & $-$0.01 & 0.25 & 0.25 \\ 
%    \hline
\cline{2-12}
 & t & $\widehat{\alpha}^{(ik)}$ & 0.14 & 2.82 & 2.82 & $-$0.02 & 0.89 & 0.89 & $-$0.01 & 0.27 & 0.27 \\ 
   &  & $\widehat{\alpha}^{(pl)}$ & 0.03 & 2.56 & 2.56 & $-$0.02 & 0.79 & 0.79 & 0.00 & 0.25 & 0.25 \\ 
   &  & $\widehat{\alpha}^{(pl*)}$ & 0.20 & 2.62 & 2.62 & 0.02 & 0.80 & 0.80 & 0.02 & 0.25 & 0.25 \\ 
   \hline
\hline
\end{tabular}
\caption{Model \eqref{mod1} with  Gaussian copula, quantities multiplied by 100.} 
\label{tab.2dim.normal}
\end{footnotesize}
\end{table}

\subsection{Findings}

As it is well known \citep{genest_et_al_1995,tsukahara_2005} in case of no covariates 
the maximum pseudo-likelihood is usually more efficient than the moment like estimators. 
This is illustrated by the performance of the estimators $\widetilde{\alpha}^{(ik)}$ and 
$\widetilde{\alpha}^{(pl)}$ that are calculated from the errors~$\eps_{ij}$. The question 
of interest is if this property continues to hold also for estimators that are calculated 
from the residuals (i.e., in the presence of covariates). 

Generally speaking one can conclude that in agreement with our theoretical results  
the maximum pseudo-likelihood estimator $\widehat{\alpha}^{(pl)}$ outperforms 
$\widehat{\alpha}^{(ik)}$ in situations for which our regularity assumptions 
are satisfied (see Table~\ref{tab.2dim.frank} and the rows corresponding to $t$-marginals 
in Tables~\ref{tab.2dim.clayton} and~\ref{tab.2dim.normal}). For these situations the 
modified maximum pseudo-likelihood estimator $\widehat{\alpha}^{(pl*)}$ is of no interest. 

On the other hand the performance of 
% the maximum pseudo-likelihood estimator 
$\widehat{\alpha}^{(pl)}$ may deteriorate significantly 
if the regularity assumptions are not met. The problems are generally worse for larger 
values of Kendall's tau (a stronger dependence). It is also interesting that 
exponential margins (rows denoted as N+E) are much more problematic than uniform margins 
(rows denoted as N+U).  

As illustrated in Table~\ref{tab.2dim.clayton} one should be in particular 
careful when fitting the Clayton copula (and also the Gumbel copula as 
illustrated in the Supplementary material). 
Then $\widehat{\alpha}^{(pl)}$ performs significantly worse than $\widehat{\alpha}^{(ik)}$ 
in cases of non-regular margins combined with a strong dependence ($\tau=0.75$). The problems 
can be to some extent prevented by considering the modified estimator 
$\widehat{\alpha}^{(pl*)}$ in particular in case of uniform margins (N+U). Thus 
while for Frank copula  the modified estimator  $\widehat{\alpha}^{(pl*)}$ is of no 
interest, for the Clayton (and the Gumbel) copula it presents an interesting alternative 
to the `standard' pseudo maximum-likelihood estimator.

The results for the Gaussian copula (see Table~\ref{tab.2dim.normal}) are of independence 
interest. Note that although the density of the copula function is unbounded,  
% the maximum pseudo-likelihood estimator  
the estimator $\widehat{\alpha}^{(pl)}$  
performs better than $\widehat{\alpha}^{(ik)}$ for $\tau = 0.5$ even in case 
of exponential margins (N+E). And this holds true for uniform margins (N+U) even for $\tau=0.75$.  
This raises a question whether a milder assumptions than $(\mathbf{F}_{j\eps})$ would 
be sufficient for the Gaussian copula. 
% Note that similar conclusions can be made also for 
% ($t$-copula)

% seems to perform satisfactory for $\tau=0.5$ even for N+E and N+U marginals, because  it performs even slightly better that the estimator   based on Kendall's tau. However, for $\tau=0.75$ we can observe a bias in Table \ref{tab.2dim.normal} for N+E and N+U marginals. This bias reduces with the sample size, but is still visible for $n=10\,000$, in particular for the exponential marginal $F_{2\eps}$. The modified estimator $\widehat{\alpha}^{(pl*)}$ helps to slightly reduce this bias and also seems to reduce the standard deviation. For the situation N+U, $\widehat{\alpha}^{(pl*)}$  seems to be even preferable compared to $\widehat{\alpha}^{(ik)}$. On the other hand, if  $F_{2\eps}$ is exponential, then $\widehat{\alpha}^{(ik)}$ still performs better. The $t$ marginals again provide a benchmark of a regular situation, where Theorem~1 applies. Here,  $\widehat{\alpha}^{(pl)}$ is 
% preferable even if small samples and the use of the modified estimator  $\widehat{\alpha}^{(pl*)}$ is of no use. 

\medskip 
 
 An analogous simulation study was conducted also for a system of three linear regressions, where the vector of innovations was sampled from  $C\big(F_{1\eps}(y_1),F_{2\eps}(y_2),F_{3\eps}(y_3)\big)$ with
 the marginals $F_{1\eps}$ and $F_{2\eps}$ being standard normal and  $F_{3\eps}$ either exponential (with mean~1)  or uniform on $[-1,1]$. As the obtained results are very similar to the results for model~\eqref{mod1}, they are not presented here, but can be found in the Supplementary material. The common important finding %from the two models  
is that the pseudo-likelihood estimator   $\widehat{\alpha}^{(pl)}$ may 
perform poorly 
(and noticeably worse compared to  $\widehat{\alpha}^{(ik)}$) for copula families with unbounded densities 
even in cases when only one of the marginals does not satisfy the regularity assumption  while the remaining ones are regular. 

% \subsection{Discussion}
% Even though the linear regression model with exponential errors might seem as quite untypical, analogous models are often considered  for time series of non-negative observations, see \cite{andel1,andel2,nielsen}.  Our results for iid setting thus suggest that the method of the pseudo-likelihood estimation can be highly problematic for such models. 
% 
% Even if the theoretical assumptions are fulfilled,  the estimator based on the inversion of Kendall's tau might be sometimes  preferable, in particular for smaller samples. % We have seen this, for instance, for the Frank copula with Kendall's tau  $\tau=0.75$. 

\section{Conclusions and further discussions} 
\label{sec: discussion}

As illustrated in the previous section one should be careful when a copula with an 
unbounded density is fitted with the help of the maximum pseudo-likelihood method. 
Although the assumptions of Theorem~\ref{thm equiv of MPL estim} are not strict one 
should keep in mind that they are not satisfied for distributions with a non-continuous 
error density function~$f_{j\eps}$ (e.g., uniform distribution, exponential distribution, \ldots). 
Although such situations are probably rare in practice, there are applications 
in which for instance uniform errors can naturally appear  \citep[see e.g.,][]{schechtman1986estimating}. 

% \boxit{Shall we mention an obvious application of MPL estimator in GOF statistics? It would 
% obviously fit into this paper but I find the paper already rather long.}

One of the possible next steps would be to generalize the results into the time-series context and to find 
the assumptions so that the results claimed in \cite{chen2006estimation} hold. Based 
on our results for i.i.d. setting and our simulation study 
we conjecture that the method 
of the pseudo-likelihood estimation can be problematic when the marginal models have exponential innovations (or more generally 
positive or bounded innovations with discontinuous density) 
\citep[see e.g.][]{lawrance1985modelling, davis1989estimation, andel1989non, andel1992nonnegative, nielsen2003likelihood}
and one uses $\sqrt{n}$-consistent estimators of the model parameters.
%. This fact is not assumed and exploited to 
%get a faster than the standard $\sqrt{n}$-convergence rate 
%of the estimator of $\theta_{j}$. 

Note that in models where (based on our findings) the use 
of maximum pseudo-likelihood estimation is questionable, one 
can consider 
% to estimate the parameters by 
the method of moments \citep[see e.g., Section~5.5.1 of][]{mcneil2005quantitative,brahimi2012semiparametric}.  
As proved in \cite{cote_genest_omelka_2019} 
many moment estimators based on residuals 
satisfy~\eqref{eq: equiv of estim of parameters} under less restrictive assumptions 
on the marginal error density~$f_{j\eps}$. In particular 
for standard two-dimensional copulas the method of the inversion 
of Kendall's tau can present a `robust' alternative. 
It is usually only slightly less efficient if no covariates are 
present, but in the presence of covariates it can perform significantly better than 
the maximum pseudo-likelihood estimator.  
% (see Section~\ref{sec: sim study}). 
% Further as shown in \cite{cote_genest_omelka_2019} the  

For the sake of brevity %in this paper 
we concentrated only on 
estimation of the copula parameter. We conjecture that also 
other procedures (e.g., procedures for goodness-of-fit testing) 
that make use of the maximum pseudo-likelihood 
estimator $\widehat{\alfab}_{n}$ calculated from the residuals 
will be valid provided that next to our assumptions also 
some standard regularity assumptions for these 
procedures are satisfied. 

\section*{Acknowledgments}
% The authors are grateful to the Editor-in-Chief, \CG{an} Associate Editor and the reviewers for their valuable comments, which led to an improved manuscript. 
 M.\ Omelka gratefully acknowledges support from the grant GACR 19-00015S. The research of \v{S}.\ Hudecov\'{a} was supported by the grant GACR 18-01781Y. 
 N.\ Neumeyer gratefully acknowledges support from the DFG (Research Unit FOR 1735 Structural Inference in Statistics: Adaptation and Efficiency).
 
\appendix

\renewcommand{\theequation}{A\arabic{equation}}
\setcounter{equation}{0}
% % % %

% \section{Proof of Theorem~\ref{thm equiv of MPL estim}}
\section{Proofs of the main results}

Note that the estimated pseudoobservations $\Ubhat_{i}$ given by \eqref{eq: estimated pseudoobservations} can be viewed as estimates of 
`unobserved' pseudoobservations $\Ubtilde_{i}$ (given in \eqref{eq: true pseudoobservations}) 
 which can be further viewed as estimates of $\Ub_{i}$, given by 
 \begin{equation*} %\label{eq: Ui}
 \Ub_{i} = \big(U_{1i},\dotsc,U_{di}\big) \tr 
 = \big(F_{1\eps}(\eps_{1i}),\dotsc,F_{d\eps}(\eps_{di})) \tr.    
 \end{equation*}
To prove 
Theorem~\ref{thm equiv of MPL estim} we need some technical 
results about the `closeness' of $\Uhat_{ji}$ (the $j$-th element of $\Ubhat_{i}$) to 
$\Utilde_{ji}$ and $U_{ji}$. 

As we will show later one does not need to handle $\Uhat_{ji}$ if 
either $U_{ji}$ is close to zero or one or if $M_{j}(\Xb_i)$ is too large. 
This is formalised as follows. Introduce the set of indices 
% % 
\begin{equation} \label{eq: Jjdeltan}
 \mathrm{J}_{jn}^{X} = \big\{i \in \{1,\dotsc,n\}: U_{ji} \in [\delta_n, 1-\delta_{n}], 
  M_{j}(\Xb_i) \leq a_{n} 
 \big\},   
\end{equation}
where 
% 
% % 
% \begin{equation} \label{eq: deltan and an}
%  a_{n} = n^{a}, \text{ where } a \in I_{a} = (\tfrac{\beta}{r\lambda}+\tfrac{1}{2r}, \tfrac{1}{2}-\tfrac{1}{\lambda}).  
% \end{equation}
% %
% 
\begin{equation} \label{eq: deltan and an}
 \delta_{n} = \frac{1}{n^{1/\lambda}},  
% for some $\lambda$. Further introduce 
\quad \text{and} \quad 
 a_{n} = n^{1/(\lambda_{x} r)}, 
  \quad \text{for some }  \lambda \geq 1 
 \text{ and } 0 < \lambda_{x} \leq \lambda.    
%  \text{ where } a = 1/(\lambda_{x} r). 
\end{equation}
%
 
% %  
% In what follows we will assume that
% % 
% \begin{equation} \label{eq: lambda bigger}
%  \lambda > \frac{2(\beta+r)}{r-1}
% \end{equation}
% % 
% % which implies that  $\tfrac{\beta}{r\lambda}+\tfrac{1}{2r} < \tfrac{1}{2}-\tfrac{1}{\lambda}$ 
% % and so the interval~$I_{a}$ is not empty. 

% 
% 
% The derivation of the asymptotic distribution of $\widetilde{\alfab}_{n}$ 
% makes use of Lemma~A3 by \citet{shorack1972functions} 
% and Chibisov--O'Reilly theorem, see, e.g. p.~462 of \cite{shorack_wellner} or Theorem~A of \cite{csorgo1993convergence}. In what follows I am trying to formulate the analogies 
% of those results. It seems to be that \MO{these results do not have to hold uniformly 
% in $i \in \{1,\dotsc,n\}$, but only on a subset of indices~$i$ for which 
% $U_{ji}$ is `not too close' to~0 or~1.} More exactly let $\delta_{n} = \tfrac{M}{n^{1/\lambda}}$ 
% for some finite constants $M$  and $\lambda > 0$. Now 
% for a given $j \in \{1,\dotsc,d\}$ introduce the set of indices for which 
% distance of $U_{ji}$ from~$0$ or~$1$  is at least~$\delta_{n}$, i.e. 

The following lemma gives an upper bound on the number of indices~$i$ for which it holds that 
$U_{ji} \not \in [\delta_n, 1-\delta_n]$ or $M_{j}(\Xb_i) > a_{n}$. 

\begin{lemma} \label{lemma bound on prob of large Mj} 
Let $\delta_{n}$  and $a_{n}$ satisfy \eqref{eq: deltan and an} 
and assumption~$\boldsymbol{(ms)}$ holds. Then 
\begin{equation*}
%  \notag 
\frac{1}{n} \sum_{i=1}^{n} \ind\big\{U_{ji} \not \in [\delta_n, 1-\delta_n] 
 \text{ or } M_{j}(\Xb_i) > a_n\Big\}   
= O_{P}\big(\tfrac{1}{n^{1/\lambda}} \big),  
% = o_{P}\big(\tfrac{1}{n^{\beta/\lambda+1/2}}\big),  
%     = u^{\beta}(1-u)^{\beta}o(n^{-1/2}),      
\end{equation*}
which further implies that 
% it holds 
% 
\begin{equation*} %\label{eq: prob bound on sum of indic with large Mj}
%  \notag 
\Pr\bigg(\sum_{i=1}^{n} \ind\big\{U_{ji} \not \in [\delta_n, 1-\delta_n] 
 \text{ or } M_{j}(\Xb_i) > a_n\Big\}   
 \leq n^{1-1/\lambda}\log n \bigg) \ntoinfty 1. 
%     = u^{\beta}(1-u)^{\beta}o(n^{-1/2}),      
\end{equation*}
\end{lemma}
\begin{proof}
Denote 
\[
 p_{n} = \Pr\big(U_{ji} \not \in [\delta_n, 1-\delta_n] 
\text{ or } M_{j}(\Xb_i) > a_n\big) 
\]
and note that thanks to~\eqref{eq: deltan and an} and Markov's inequality 
(applied to $M_{j}^{r}(\Xb_{i})$)
\begin{align*}
\notag 
 p_{n}   &\leq \Pr\big(U_{ji} \not \in [\delta_n, 1-\delta_n] \big)
 + \Pr\big(M_{j}(\Xb_i) > a_n \big) 
\\
& \leq 2\delta_n + \E\, \tfrac{M_{j}^{r}(\Xb_{i})}{a_n^{r}} 
= O\big(\tfrac{1}{n^{1/\lambda}}\big) 
+ O\big(\tfrac{1}{n^{1/\lambda_{x}}} \big)
= O\big(\tfrac{1}{n^{1/\lambda}}\big). 
% \\
% 
% \label{eq: bound for pn}
%   = o\big(\tfrac{1}{n^{\beta/\lambda+1/2}}\big) 
%   = u^{\beta}(1-u)^{\beta}o(n^{-1/2}),      
\end{align*}
% 
% where we have used that $a > \tfrac{\beta}{r\lambda}+\tfrac{1}{2r}$. 

Now as the random variable $\frac{1}{n} \sum_{i=1}^{n} \ind\big\{U_{ji} \not \in [\delta_n, 1-\delta_n] 
 \text{ or } M_{j}(\Xb_i) > a_n\Big\}$ is non-negative one can 
use once more Markov's inequality to conclude that 
\[
 \frac{1}{n} \sum_{i=1}^{n} \ind\big\{U_{ji} \not \in [\delta_n, 1-\delta_n] 
 \text{ or } M_{j}(\Xb_i) > a_n\Big\} = O_{P}(p_{n}). 
\]

% Now with the help of Chebyshev's inequality one gets 
% % 
% \[
% \Pr\Big( \Big|\frac{1}{n} \sum_{i=1}^{n} \ind\big\{
% U_{ji} \not \in [\delta_n, 1-\delta_n] 
% \text{ or } M_{j}(\Xb_i) > a_n\big\}   
%  - p_{n}\Big| > \eps\Big) \leq \frac{p_{n}(1-p_{n})}{\epsilon^{2}\,n},
% \]
% % 
% which together with \eqref{eq: bound for pn} implies 

% % 
% \begin{align*}
%  \notag 
% \frac{1}{n} \sum_{i=1}^{n} \ind\big\{U_{ji} \not \in [\delta_n, 1-\delta_n] 
% \text{ or } M_{j}(\Xb_i) > a_n\big\}   
%  &= p_{n} + O_{P}\big(\tfrac{\sqrt{p_{n}}}{\sqrt{n}} \big) 
% % \\
% % 
% % \label{eq: Mj bigger than Kn}
% \\
% & = O\big(\tfrac{1}{n^{1/\lambda}} \big) + 
%  O_{P}\big(\tfrac{1}{n^{1/(2\lambda) + 1/2}}\big)  
% = O_{P}\big(\tfrac{1}{n^{1/\lambda}} \big),   
% % = o\big(\tfrac{1}{n^{\beta/\lambda+1/2}}\big) 
% % + o_{P}\big(\tfrac{1}{n^{\beta/2\lambda+3/4}}\big) 
% % = o_{P}\big(\tfrac{1}{n^{\beta/\lambda+1/2}}\big),  
% %     = u^{\beta}(1-u)^{\beta}o(n^{-1/2}),      
% \end{align*}
% % 
% where we have used that $\lambda \geq 1$. 
% % $a > \tfrac{\beta}{r\lambda}+\tfrac{1}{2r}$ 
% % and $\beta/(2\lambda) < 1/4$ (as we assume that $r \geq 2$ which implies 
% % that $\lambda > 4 - 2 \beta > 2 \beta$). 
% 
\end{proof}

\subsection{Some results on statistics with ranks calculated 
from residuals}

% Note that \eqref{eq: generalized shorack bound on empirical cdf of resid} 
% and~\eqref{eq: generalized chibisov reilly for resid} imply 
% % 
% \begin{equation} \label{eq: Uhat close to U}
%  \max_{j=1,\dotsc,d} \max_{i \in \{1,\dotsc,n\}} \big|\Uhat_{ji} - U_{ji}\big| 
%   = O_{P}\big(\tfrac{1}{\sqrt{n}}\big). 
% \end{equation}
% 

% Before we start with the proof of the theorem we introduce two useful auxiliary lemmas. 
% 
\begin{lemma}\label{lemma: convergence in probability of functions of ranks}
Suppose that assumptions $(\mathbf{F}_{j\eps})$ and $\boldsymbol{(ms)}$ 
% formulated on pages \pageref{assumption Fjeps} and \pageref{assumption mu sigma}
hold and that $\varphi$ is a~$\mathcal{J}$-function.
Then 
\[
%  \frac{1}{n} \suman w_{n}(\Ubhat_{i})\,\varphi(\Ubhat_{i}) 
 \frac{1}{n} \suman \varphi(\Ubhat_{i})  \inPr \Es \varphi(\Ub).
\]
\end{lemma}
\begin{proof}
As $\varphi$ is a $\mathcal{J}$-function, it is easy to show that 
the expectation $\Es \varphi(\Ub)$  
exists and is finite. Thus thanks to the law of large numbers it is sufficient to show 
\begin{equation} \label{eq: diff of J}
 D_{n} = \left|\frac{1}{n} \suman \varphi(\Ubhat_{i}) 
 - \frac{1}{n} \suman \varphi(\Ub_{i}) 
  \right| \inPr 0.
\end{equation}

Let $\mathrm{J}_{jn}^{X}$ and $\delta_{n}$  be as in~\eqref{eq: Jjdeltan} and 
\eqref{eq: deltan and an}, where $\lambda$ and $\lambda_{x}$ are chosen 
so that they satisfy the assumptions of 
Lemma~\ref{lemma Ujhat minus Ujtilde}. Then this lemma together 
with the standard Glivenko-Cantelli theorem for the empirical distribution 
function~$\Fhat_{j\eps}$ implies that 
\begin{align}
 \notag 
 \max_{j \in \{1,\dotsc,d\}} &\max_{i \in \mathrm{J}_{jn}^{X}} 
 \big|\Uhat_{ji} - U_{ji}\big|  
\\
\label{eq: Uhat close to U in pr}
& \leq 
 \max_{j \in \{1,\dotsc,d\}} \max_{i \in \mathrm{J}_{jn}^{X}} 
 \big|\Uhat_{ji} - \widetilde{U}_{ji}\big| + 
\max_{j \in \{1,\dotsc,d\}} \max_{i \in \mathrm{J}_{jn}^{X}} 
 \big|\widetilde{U}_{ji} - U_{ji}\big|
  = o_{P}\big(1\big).    
\end{align}
Now introduce
\begin{equation} \label{eq: JdeltanX and KdeltanX}
\mathrm{J}_{n}^{X} = \cap_{j=1}^{d} \mathrm{J}_{jn}^{X}, 
\qquad \text{and} \qquad 
\mathrm{K}_{n}^{X} = \{1,\dotsc,n\}\setminus \mathrm{J}_{n}^{X}   
\end{equation}
and note that with the help of \eqref{eq: Uhat close to U in pr} 
\begin{equation} \label{eq: Ubhat close to Ub in pr}
 \max_{i \in \mathrm{J}_{n}^{X}} 
 \big\|\Ubhat_{i} - \Ub_{i}\big\| = o_{P}(1).    
\end{equation}
As the above equation is not guaranteed for $i \in \mathrm{K}_{n}^{X}$, 
we need to take care about the sets of indices 
$\mathrm{J}_{n}^{X}$ and $\mathrm{K}_{n}^{X}$ separately. That is why  
we bound $D_{n}$ given by \eqref{eq: diff of J} as 
\begin{equation}
\label{eq: bound on diff J 1}
 D_{n} \leq 
\frac{1}{n} \sum_{i  \in \mathrm{K}_{n}^{X}} \big|\varphi(\Ubhat_{i})\big| 
+ 
 \frac{1}{n} \sum_{i \in \mathrm{K}_{n}^{X}} \big|\varphi(\Ub_{i})\big| 
+ \bigg|
\frac{1}{n} \sum_{i \in \mathrm{J}_{n}^{X}} \varphi(\Ubhat_{i}) - \frac{1}{n} \sum_{i \in \mathrm{J}_{n}^{X}} \varphi(\Ub_{i})
  \bigg|. 
% \\
%
% \label{eq: bound on diff J 2}
\end{equation}
In what follows we show that each term on the right-hand side of 
\eqref{eq: bound on diff J 1} is asymptotically negligible. 
\smallskip 

\noindent \textit{Dealing with the first term in \eqref{eq: bound on diff J 1}}

% Let introduce 
% % 
% \begin{align*}
%  \mathbf{I}_{j}^{(0)} &= \big\{ \ub \in (0,1)^{d}: u_{j} = \min_{k \in \{1,\dotsc,d\}} \min\{u_{k}, 1-u_{k}\} \big\}, 
% \\
% % 
%  \mathbf{I}_{j}^{(1)} &= \big\{ \ub \in (0,1)^{d}: 1-u_{j} = \min_{k \in \{1,\dotsc,d\}} \min\{u_{k}, 1-u_{k}\} \big\}  
% \end{align*}
% % 
% and note that $(0,1)^{d} \subset \bigcup_{j=1}^{d} (\mathbf{I}_{j}^{(0)} \cup \mathbf{I}_{j}^{(1)})$. 
% Further for $j \in \{1,\dotsc,d\}$ and $l \in \{0,1\}$  put 
% % 
% \begin{equation} \label{eq: hatKjn}
%  \widehat{\mathrm{K}}_{jn}^{(l)} = \{i: i \in \mathrm{K}_{n}^{X} \, \&  \, \Ubhat_{i} 
%   \in \mathbf{I}_{j}^{(l)}\}.  
% %  \subset \{i:  \Ub_{i} \not \in I_{\delta} \, \&  \, \Ub_{i}  \in \mathbf{I}_{j}^{(l)}\}.   
% \end{equation}
% % 

As $\varphi$ is a $\mathcal{J}$-function one can bound 
\begin{equation*}
  \frac{1}{n} \sum_{i \in \mathrm{K}_{n}^{X}} \left| \varphi(\Ubhat_{i})\right|
  \leq   \sum_{j=1}^{d} %\Bigg[
 \frac{M_{1}}{n} \sum_{i  \in \mathrm{K}_{n}^{X}}
   \frac{1}{\big[\min\{\widehat{U}_{ji},1-\widehat{U}_{ji}\}\big]^{\eta}}\,. 
%   \left|\varphi(\Ubhat_{i})\right| 
% + \frac{1}{n} \sum_{i  \in \mathrm{K}_{j\delta}^{(1)}} 
%   \left|\varphi(\Ub_{i})\right|
 %\Bigg] 
\end{equation*}
Now by  Lemma~\ref{lemma bound on prob of large Mj} 
(with probability going to one) 
there are at most $d n^{1-1/\lambda}\log n$ indices~$i$ 
for which there exists $j \in \{1,\dotsc,d\}$ such that 
$U_{ji} \not \in [\delta_n,1-\delta_n]$  or $M_{j}(\Xb_i) > a_n$.  
Thus one can choose the indices~$i$ for which 
% the bound on the function 
% $|\varphi(\Ubhat_{i})|$ 
$\frac{1}{[\min\{\widehat{U}_{ji},1-\widehat{U}_{ji}\}]^{\eta}}$ 
takes the biggest values and gets that 
(with probability going to one) 
\begin{align} 
\notag 
  \frac{1}{n} \sum_{i \in \mathrm{K}_{n}^{X}} \left|\varphi(\Ubhat_{i})\right| 
 &\leq \sum_{j=1}^{d} \frac{M_{1}}{n} \Bigg[\sum_{i=1}^{\lceil \frac{d}{2}n^{1-1/\lambda}\log n \rceil} 
  \frac{1}{\big(\tfrac{i}{n+1}\big)^{\eta}}
 + \sum_{i=\lfloor n - \frac{d}{2} n^{1 - 1/\lambda}\log n \rfloor}^{n} 
  \frac{1}{\big(1-\tfrac{i}{n+1}\big)^{\eta}}
 \Bigg] 
\\
\label{eq: sum KnX of phi at Uhat}
 & \leq 2\,d^2\,M_{1}\,n^{-(1-\eta)/\lambda} (\log n)^{1-\eta}\,(1+o(1)) = o(1). 
\end{align}

\medskip 

\noindent \textit{Dealing with the second term in \eqref{eq: bound on diff J 1}}
\nopagebreak 

Note that $\E \big|\varphi(\Ub_{i})\big| < \infty$ implies that 
\begin{align*}
 \E\,&\bigg[ \frac{1}{n} \sum_{i \in \mathrm{K}_{n}^{X}} \big|\varphi(\Ub_{i})\big|  
 \bigg] 
 = \E \Big[ \big|\varphi(\Ub_{i})\big|\,
 \ind\big\{\Ub_{i} \not \in [\delta_n,1-\delta_n]^d 
\ \text{or} \ \max_{1 \leq j \leq d} M_{j}(\Xb_i) > a_{n}\big\}  
\Big] 
\\
 &\leq \E \Big[ \big|\varphi(\Ub_{i})\big|\,
 \ind\big\{\Ub_{i} \not \in [\delta_n,1-\delta_n]^d\big\}
 \Big] 
+ \Es \big|\varphi(\Ub_{i})\big| \,  
 \Pr\Big(\max_{1 \leq j \leq d} M_{j}(\Xb_i) > a_{n}\Big)   
 \ntoinfty 0. 
\end{align*}
Thus 
$
 \frac{1}{n} \sum_{i \in \mathrm{K}_{n}^{X}} \big|\varphi(\Ub_{i})\big|  = o_{P}(1)
$
follows from Markov's inequality.

\medskip 

\noindent \textit{Dealing with the third term in \eqref{eq: bound on diff J 1}}

We use the continuity of the function~$\varphi$. 
To be able to do that we need to stay in the interior of $[0,1]^{d}$. 
Thus for a~given $\delta \in (0, 1/2)$ (that will be specified later on), consider the set
\begin{equation} \label{eq: Idelta}
 \mathbf{I}_{\delta} = \{\ub: \ub \in [\delta, 1-\delta]^{d}\}.
\end{equation}
and introduce the corresponding sets of indices
\begin{equation} \label{eq: Jdelta}
 \mathrm{J}_{\delta} = \big\{i \in \{1,\dotsc,n\}: \Ub_{i} \in \mathbf{I}_{\delta} \big\}, 
 \quad \mathrm{K}_{\delta} = \{1,\dotsc,n\}\setminus \mathrm{J}_{\delta}, 
\end{equation}
where for simplicity of notation we do not stress that both  $\mathrm{J}_{\delta}$ 
and $\mathrm{K}_{\delta}$ depends on~$n$. Now one can bound 
\begin{align}
\notag 
 \bigg|\frac{1}{n} &\sum_{i \in \mathrm{J}_{n}^{X}} \varphi(\Ubhat_{i}) 
   -  \frac{1}{n} \sum_{i \in \mathrm{J}_{n}^{X}} \varphi(\Ub_{i})\bigg|
\\ 
\label{eq: bound on diff J 2}  
& \leq 
 \frac{1}{n} \sum_{i \in \mathrm{J}_{n}^{X} \cap \mathrm{J}_{\delta} } 
\big|\varphi(\Ubhat_{i})  
% \frac{1}{n} \sum_{i \in \mathrm{J}_{n}^{X} \cap \mathrm{J}_{\delta}}
  - \varphi(\Ub_{i})\big|
% \\
%
% \label{eq: bound on diff J 2}
 + \frac{1}{n} \sum_{i  \in \mathrm{J}_{n}^{X} \cap \mathrm{K}_{\delta}} 
 \big|\varphi(\Ubhat_{i})\big| 
   +  \frac{1}{n} \sum_{i \in \mathrm{J}_{n}^{X} \cap \mathrm{K}_{\delta}} 
 \big|\varphi(\Ub_{i})\big|. 
\end{align}
Note that by the uniform continuity of the function~$\varphi(\cdot)$ on $[\delta/2,1-\delta/2]^d$ 
and \eqref{eq: Ubhat close to Ub in pr} 
one gets that the first term on the right-hand side of \eqref{eq: bound on diff J 2} 
converges to zero in probability. 

\smallskip 

To deal with the second term on the right-hand side of 
\eqref{eq: bound on diff J 2} note that thanks to~\eqref{eq: Ubhat close to Ub in pr}  
 with probability going to one 
\[
 \mathrm{J}_{n}^{X} \cap \mathrm{K}_{\delta} \subseteq 
 \big\{i \in \{1,\dotsc,n\}: \Ubhat_{i} \not \in [2\delta,1-2\delta]^{d}\big\} 
\]

%  
% \begin{equation} \label{eq: bound on sum J at Uhat}
% %  \notag 
% \left| \frac{1}{n} \sum_{i \in \mathrm{K}_{\delta}} \varphi(\Ubhat_{i})\right|
%  \leq 
%  \left| \frac{1}{n} \sum_{i \in \mathrm{K}_{\delta}\setminus \mathrm{K}_{n}^{X}} \varphi(\Ubhat_{i})\right|
% + \left| \frac{1}{n} \sum_{i \in \mathrm{K}_{n}^{X}} \varphi(\Ubhat_{i})\right|
% %   \leq   \sum_{j=1}^{d} \sum_{l=0}^{1} \Bigg[\frac{1}{n} \sum_{i  \in \mathrm{K}_{j\delta}^{(l)} \cap  \mathrm{J}_{jn}^{X}} 
% %   \left|\varphi(\Ubhat_{i})\right| 
% %  + \frac{1}{n} \sum_{i  \in \mathrm{K}_{j\delta}^{(l)} \cap  \mathrm{J}_{jn}^{c}}
% %   \left|\varphi(\Ubhat_{i})\right| \Bigg],
% % % + \frac{1}{n} \sum_{i  \in \mathrm{K}_{j\delta}^{(1)}} 
% % %   \left|\varphi(\Ub_{i})\right|
% \end{equation}
% 
% where $\mathrm{J}_{jn}^{c} = \{1,\dotsc,n\}\setminus \mathrm{J}_{jn}$.  
% Now fix $j \in \{1,\dotsc,d\}$ and $l \in (0,1)$. 
% for simplicity of notation consider $l=0$ 
% (the case $l=1$ can be handled completely analogously). 

Thus one can bound 
\begin{align*}
 \frac{1}{n} \sum_{i \in \mathrm{J}_{n}^{X} \cap \mathrm{K}_{\delta}}  \big|\varphi(\Ubhat_{i})\big|  
 &\leq \sum_{j=1}^{d} \frac{M_{1}}{n} \Bigg[\sum_{i=1}^{\lceil (n+1)2\delta \rceil} 
  \frac{1}{\big(\tfrac{i}{n+1}\big)^{\eta}}
 + \sum_{i=\lfloor n - (n+1)2\delta \rfloor}^{n} 
  \frac{1}{\big(1-\tfrac{i}{n+1}\big)^{\eta}}
 \Bigg] 
\\
& \leq 2\,d\,M_{1}\,\tfrac{(2\delta)^{1-\eta}}{1-\eta}\big(1 + o(1)\big), 
\end{align*}
which can be made arbitrarily small by taking $\delta$ small enough. 

\smallskip 

Finally with the help of law of large numbers the third term on the right-hand side of 
\eqref{eq: bound on diff J 2} can be  bounded by 
\begin{align} 
\notag 
  \frac{1}{n} \sum_{i \in \mathrm{J}_{n}^{X} \cap \mathrm{K}_{\delta}} \left|\varphi(\Ub_{i})\right| 
 & 
\leq \frac{1}{n} \sum_{i \in \mathrm{K}_{\delta}} \left|\varphi(\Ub_{i})\right| 
\\
\notag 
 & \leq     
\sum_{j=1}^{d} \frac{M_{1}}{n}\Bigg[
 \suman  
   \frac{1}{U_{ji}^{\eta}}\ind\{U_{ji} \leq \delta\} 
+  
\suman  
   \frac{1}{(1-U_{ji})^{\eta}}\ind\{U_{ji} \geq 1-\delta\} 
\Bigg]
\\
% % 
% & = \frac{M_{1}}{n} \sum_{i \in \widehat{\mathrm{K}}_{j\delta}^{(0)}}  \frac{1}{\Uhat_{ij}^{\eta}}\,\ind\{\Uhat_{ji} \leq 2\delta\} 
% % + \frac{1}{n} \sum_{i \in \widehat{\mathrm{K}}_{j\delta}^{(1)}} 
% %   \frac{1}{(1-\Uhat_{ij})^{\eta}}
% %  \Bigg] 
% \\
% % % 
% %  & \leq   M_{1} \sum_{j=1}^{d} \Bigg[\frac{1}{n} \sum_{i=1}^{n} 
% %    \frac{1}{\Uhat_{ij}^{\eta}}\,\ind\{\Uhat_{ji} \leq 2\delta\} 
% %    + \frac{1}{n} \sum_{i=1}^{n} \frac{1}{(1-\Uhat_{ij})^{\eta}} 
% %    \,\ind\{1-\Uhat_{ji} \leq 2\delta\} 
% %  \Bigg] 
% % 
% % \\
% 
\notag 
% \label{eq: sum J for in hatKjd}
 & = 2\,d\,M_{1}\, \big(\tfrac{\delta^{1-\eta}}{1-\eta} + o_{P}(1)\big), 
% 
%    + \frac{1}{n} \sum_{i=\lfloor 2(n+1)\delta \rfloor}^{n}  
%    \frac{1}{\big(1-\tfrac{i}{n+1}\big)^{\eta}}
%  \Bigg] 
% \\
% 
% & = 2\,d\,M_{1}\, \tfrac{(2\,\delta)^{1-\eta}}{1-\eta}\big(1 + o(1)\big), 
% 
\end{align}
which can be also made arbitrarily small by taking $\delta$ sufficiently 
small and $n$ sufficiently large. 
% Thus combining \eqref{eq: bound on sum J in Kjdelta intersec Jjdelta} and \eqref{eq: sum J for in hatKjd} handles the first summand on 
% the right-hand side of \eqref{eq: bound on sum J at Uhat}. 
% 
\end{proof}

\begin{lemma} \label{lemma: as normality of functions of ranks} 
Suppose that assumptions $(\mathbf{F}_{j\eps})$ and $\boldsymbol{(ms)}$ hold. 
Let  $\varphi$ be a $\widetilde{\mathcal{J}}^{\beta_1, \beta_2}$-function such that 
$\E\, \{\varphi(\Ub)\} = 0$ and  
$\beta > \max\{\beta_{1} + \frac{1}{r-1},\beta_{2}\}$.  
% where $\beta$ and $r$ are introduced in assumptions~$(\mathbf{F}_{j\eps})$ 
% and $$. 
Then  
\begin{equation}  \label{eq: asympt equiv score statistics}
%  \frac{1}{\sqrt{n}} \suman w_{n}(\Ubhat_{i})\,\varphi(\Ubhat_{i})
 \frac{1}{\sqrt{n}} \suman \varphi(\Ubhat_{i})  
  = \frac{1}{\sqrt{n}} \suman \varphi(\Ubtilde_{i}) + o_{P}(1). 
\end{equation}
% 
% where
% % 
% \begin{align*}
%  \widetilde{\varphi}(\ub) = \varphi(\ub)
%   + \sum_{j=1}^{d} \int_{[0,1]^{d}} \big[\ind\{u_{j} \leq v_j\}-v_{j}\big] \varphi^{(j)}(\vb)\,\mathrm{d} C(\vb). 
% \end{align*}
% % 
\end{lemma}

\begin{proof}
Let $\mathrm{J}_{n}^{X}$ and $\mathrm{K}_{n}^{X}$ 
be defined as in \eqref{eq: JdeltanX and KdeltanX}. 
Then similarly as in~\eqref{eq: sum KnX of phi at Uhat} 
of the proof of Lemma~\ref{lemma: convergence in probability of functions of ranks} 
% (see \eqref{eq: sum KnX of phi at Uhat}) 
one can bound  
% (with probability going to one) 
% 
\begin{equation} \label{eq: sum of phi over Knx}
 \frac{1}{\sqrt{n}} \sum_{i \in \mathrm{K}_{n}^{X}} \varphi(\Ubhat_{i})   
  = O_{P}(n^{\frac{1}{2}-\frac{1-\beta_{1}}{\lambda}}\,\log n), 
\quad  \frac{1}{\sqrt{n}} \sum_{i \in \mathrm{K}_{n}^{X}} \varphi(\Ubtilde_{i}) 
  = O_{P}(\,n^{\frac{1}{2}-\frac{1-\beta_{1}}{\lambda}}\,\log n),    
\end{equation}
where the role of $\eta$ is now taken by $\beta_{1}$. 

In what follows we take $\lambda$ so that 
\[
 2( 1 - \beta + \tfrac{1}{r-1})< \lambda < 2(1-\beta_1)
\]
and $\lambda_{x}$ satisfies~\eqref{eq: lambdax lower bound}. 
Such choices of $\lambda$ and $\lambda_{x}$ guarantee that  
the right-hand sides of \eqref{eq: sum of phi over Knx}
are of order $o_{P}(1)$ and at the same time 
the assumptions of Lemma~\ref{lemma about hatFjhateps} are satisfied 
%so that in what follows 
and one can make use of 
Lemmas~\ref{lemma Ujhat minus Ujtilde} and~\ref{lemma linear bound for hatU}. 
% 
% \begin{equation} \label{eq: order of deltan}
%  \delta_{n} = o(n^{-1/(2(1-\beta_{1}))}). 
% \end{equation}
% 

It is sufficient to show that 
\begin{equation*} 
%  \frac{1}{\sqrt{n}} \suman w_{n}(\Ubhat_{i})\,\varphi(\Ubhat_{i})
 \frac{1}{\sqrt{n}} \sum_{i \in \mathrm{J}_{n}^{X}} \varphi(\Ubhat_{i})  
  = \frac{1}{\sqrt{n}} \sum_{i \in \mathrm{J}_{n}^{X}} \varphi(\Ubtilde_{i}) + o_{P}(1). 
\end{equation*}
Note that 
\[
 \mathrm{J}_{n}^{X}
% = \big\{i \in \{1,\dotsc,n\}: \Ub_{i} \in I_{\delta_{n}} \big\} 
 = \big\{i \in \{1,\dotsc,n\}: \Ub_{i} \in [\delta_n, 1-\delta_{n}]^{d},  
   \max_{1 \leq j \leq d} M_{j}(\Xb_{i}) \leq a_n\big\} ,   
\]
where $\delta_{n}$ and $a_n$ are given in \eqref{eq: deltan and an}. 

\smallskip 

Now by the mean value theorem
\begin{equation}
\label{eq: expansion of Jhat} 
 \frac{1}{\sqrt{n}} \sum_{i \in \mathrm{J}_{n}^{X}} \varphi(\Ubhat_{i})
 =  \frac{1}{\sqrt{n}} \sum_{i \in \mathrm{J}_{n}^{X}} \varphi(\Ubtilde_{i})
 + \sum_{j=1}^{d} \frac{1}{\sqrt{n}} \sum_{i \in \mathrm{J}_{n}^{X}}  \varphi^{(j)}(\Ub_{i}^{*})
    \big(\widehat{U}_{ji} - \Utilde_{ji}\big),
\end{equation}
where $U_{ji}^{*}$ lies between $\widehat{U}_{ji}$ and $\widetilde{U}_{ji}$. Thus to prove the lemma 
it is sufficient to show that 
the second term on the right-hand side of \eqref{eq: expansion of Jhat} diminishes in 
probability.

With the help of Lemma~\ref{lemma Ujhat minus Ujtilde} 
for a fixed $j \in \{1,\dotsc,d\}$ one gets 
\begin{equation*} %\label{eq: decomposition in An Bn Cn}
 \frac{1}{\sqrt{n}} \sum_{i \in \mathrm{J}_{n}^{X}} \varphi^{(j)}(\Ub_{i}^{*})
% \sum_{i \in \mathrm{J}_{n}^{X}} w_{n}(\Ubhat_i)\,\varphi^{(j)}(\Ub_{i}^{*})
    \big(\Uhat_{ji} - \Utilde_{ji}\big) 
 = A_{n} + B_{n} + C_{n}, 
\end{equation*}
where
\begin{align}
\label{eq: An}
% \notag 
 A_{n} &=  \frac{1}{\sqrt{n}} \sum_{i \in \mathrm{J}_{n}^{X}} % w_{n}(\Ubhat_i)\,
    \varphi^{(j)}(\Ub_{i}^{*})  f_{j\eps}(\eps_{ji})\Big\{\E_{\Xb}\big[\tfrac{m'_{j}(\Xb, \thetab_{j})}{s_{j}(\Xb; \thetab_{j})}\big] + \eps_{ji}\,\E_{\Xb}\big[\tfrac{s'_{j}(\Xb; \thetab_{j})}{s_{j}(\Xb; \thetab_{j})}\big] \Big\}\tr (\widehat{\thetab}_{j}-\thetab_{j}),
\\
\label{eq: Bn}
B_{n} &=  \frac{1}{\sqrt{n}} \sum_{i \in \mathrm{J}_{n}^{X}} % w_{n}(\Ubhat_i)\,
 \varphi^{(j)}(\Ub_{i}^{*}) f_{j\eps}(\eps_{ji})\big(\hateps_{ji} - \eps_{ji}\big),
\\
\label{eq: Cn}
C_{n} &=  \frac{o_{P}(1)}{n} \sum_{i \in \mathrm{J}_{n}^{X}} % w_{n}(\Ubhat_i)\,
\varphi^{(j)}(\Ub_{i}^{*})U_{ji}^{\beta-\gamma}(1-U_{ji})^{\beta-\gamma} \big(1 + M_{j}(\Xb_i)\big),     
\end{align}
and $\gamma > 0$ is taken sufficiently small so that $\beta - \gamma > \beta_2$. 
In what follows we show that $C_{n}$ 
and $A_{n} + B_{n}$ are asymptotically negligible. 

\medskip 

\noindent \textit{Dealing with $C_{n}$}. 
With the help of
Lemma~A3 of \citet{shorack1972functions} and Lemma~\ref{lemma linear bound for hatU} for each $\eps > 0$ 
there exists a positive constant $L$ 
such that the quantity~$C_{n}$ given by \eqref{eq: Cn} can be with 
probability at least $1-\eps$ bounded  by 
\begin{align}
\notag 
% \label{eq: Cn asymptotically}
\big|C_{n}\big| &\leq \frac{o_{P}(1)}{n} 
  \sum_{i \in \mathrm{J}_{n}^{X}} 
 \big|\varphi^{(j)}(\Ub_{i}^{*})(U_{ji}^{*})^{\beta_2}(1-U_{ji}^{*})^{\beta_2}\big|\,   
\frac{U_{ji}^{\beta-\gamma}(1-U_{ji})^{\beta-\gamma}}{(U_{ji}^{*})^{\beta_2}(1-U_{ji}^{*})^{\beta_2}}  
\big(1 + M_{j}(\Xb_{i})\big)
\\
\notag 
&\leq \frac{o_{P}(1)}{n} 
  \sum_{i \in \mathrm{J}_{n}^{X}} 
  \frac{M_1}{\big[\min_{j=1,\dotsc,d}{\min\{U_{ji}^{*}, 1-U_{ji}^{*}\}}\big]^{\eta}}
 \,\frac{1}{L^{\beta_{2}}} \big(1 + M_{j}(\Xb_{i})\big)
\\ 
\notag 
&= \frac{o_{P}(1)}{n} 
  \sum_{i \in \mathrm{J}_{n}^{X}} 
  \frac{M_1}{\big[\min_{j=1,\dotsc,d}{\min\{U_{ji}, 1-U_{ji}\}}\big]^{\eta}} 
 \,\frac{1}{L^{\beta_{2}+\eta}} \big(1 + M_{j}(\Xb_{i})\big)
\\
\notag 
&  = o_{P}(1) \, O_{P}(1) = o_{P}(1), 
\end{align}
where the law of large numbers  is used on the last line. 

Thus one can concentrate on the quantities $A_{n}$ and $B_{n}$. 

\medskip 

\noindent \textit{Dealing with $A_{n}$}. Note that $A_{n}$ given by \eqref{eq: An} 
can be rewritten as 
\begin{align}
 \notag 
 A_{n} &=  \sqrt{n}\,(\widehat{\thetab}_{j}-\thetab_{j})\tr\,\E_{\Xb}\big[\tfrac{m'_{j}(\Xb, \thetab_{j})}{s_{j}(\Xb; \thetab_{j})}\big]\, \frac{1}{n} \sum_{i \in \mathrm{J}_{n}^{X}} % w_{n}(\Ubhat_i)\,
    \varphi^{(j)}(\Ub_{i}^{*})  f_{j\eps}(\eps_{ji})
\\
\label{eq: An rewritten}
& \quad + \sqrt{n}\,(\widehat{\thetab}_{j}-\thetab_{j})\tr\,\E_{\Xb}\big[\tfrac{s'_{j}(\Xb; \thetab_{j})}{s_{j}(\Xb; \thetab_{j})}\big]\, \frac{1}{n} \sum_{i \in \mathrm{J}_{n}^{X}} % w_{n}(\Ubhat_i)\,
    \varphi^{(j)}(\Ub_{i}^{*})  f_{j\eps}(\eps_{ji}) \eps_{ji}. 
\end{align}
Now analogously as in the proof of 
Lemma~\ref{lemma: convergence in probability of functions of ranks} one 
can show that 
% Using \eqref{eq: generalized shorack bound on empirical cdf of resid} 
% and \eqref{eq: Uhat close to U in pr} one can show by the same argument as in 
% Lemma~1 of \cite{ogvp_testing_rao_2017} that 
% 
\begin{align} 
\notag 
 \frac{1}{n} \sum_{i \in \mathrm{J}_{n}^{X}} %w_{n}(\Ubhat_i)\,
    \varphi^{(j)}(\Ub_{i}^{*}) f_{j\eps}(\eps_{ji}) 
 &= \frac{1}{n} \sum_{i \in \mathrm{J}_{n}^{X}} %w_{n}(\Ubhat_i)\,
    \varphi^{(j)}(\Ub_{i}^{*}) f_{j\eps}\big(F_{j\eps}^{-1}(U_{ji}) \big) 
\\
% % 
% \notag 
%  &=\frac{1}{n} \sum_{i \in \mathrm{J}_{n}^{X}} % w_{n}(\Ub_i)\,
%     \varphi^{(j)}(\Ub_{i}) f_{j\eps}\big(F_{j\eps}^{-1}(U_{ji}) \big) + o_{P}(1) 
% \\
% % 
\label{eq: sum of deriv of J location}
&= 
 \Es \big[\varphi^{(j)}(\Ub) f_{j\eps}\big(F_{j\eps}^{-1}(U_{j}) \big) \big] + o_{P}(1) 
\end{align}
and also 
\begin{equation}
\label{eq: sum of deriv of J scale}
 \frac{1}{n} \sum_{i \in \mathrm{J}_{n}^{X}} % w_{n}(\Ubhat_i)\,
  \varphi^{(j)}(\Ub_{i}^{*}) f_{j\eps}(\eps_{ji})\,\eps_{ji} = 
 \Es \big[ \varphi^{(j)}(\Ub) f_{j\eps}\big(F_{j\eps}^{-1}(U_{j}) \big)\,F_{j\eps}^{-1}(U_{j}) 
 \big] 
 + o_{P}(1).  
\end{equation}
Combining \eqref{eq: An rewritten}, \eqref{eq: sum of deriv of J location},  
\eqref{eq: sum of deriv of J scale} and the fact that the estimator $\hatthetab_{j}$ 
is $\sqrt{n}$-consistent  yields 
\begin{align} 
\notag 
 A_{n} &=  
\sqrt{n}\,(\widehat{\thetab}_{j}-\thetab_{j})\tr \,\Es \big[\varphi^{(j)}(\Ub) f_{j\eps}\big(F_{j\eps}^{-1}(U_{j}) \big)\big]  
\E_{\Xb}\big[\tfrac{m'_{j}(\Xb, \thetab_{j})}{s_{j}(\Xb; \thetab_{j})}\big]
\\
\label{eq: An asymptotically} 
& \quad +
\sqrt{n}\,(\widehat{\thetab}_{j}-\thetab_{j})\tr\, \Es \big[\varphi^{(j)}(\Ub) f_{j\eps}\big(F_{j\eps}^{-1}(U_{j}) \big) 
 F_{j\eps}^{-1}(U_j)\big]\E_{\Xb}\big[\tfrac{s'_{j}(\Xb; \thetab_{j})}{s_{j}(\Xb; \thetab_{j})}\big]   + o_{P}(1).    
\end{align}

\noindent \textit{Dealing with $B_{n}$}. Now have a look at the term~$B_{n}$ defined in \eqref{eq: Bn}. One can proceed analogously as above and show that 
\begin{equation} \label{eq: Bn simplified}
 B_{n} = \frac{1}{\sqrt{n}} \sum_{i \in \mathrm{J}_{n}^{X}} \varphi^{(j)}(\Ub_{i})
    f_{j\eps}(\eps_{ji})\big(\hateps_{ji} - \eps_{ji}\big) + o_{P}(1)
 = B_{n1} + B_{n2} + o_{P}(1),  
\end{equation}
where 
\begin{align*}
%  \label{eq: Bn1}
 B_{n1} &= \frac{1}{\sqrt{n}} \sum_{i \in \mathrm{J}_{n}^{X}} \varphi^{(j)}(\Ub_{i})
    f_{j\eps}\big(F_{j\eps}^{-1}(U_{ji})\big)\big[\tfrac{m_{j}(\Xb_{i};\thetab_{j}) 
     - m_{j}(\Xb_{i}; \widehat{\thetab}_{j})}{s_{j}(\Xb_{i}; \widehat{\thetab}_{j})}\big],  \\ 
% 
% \label{eq: Bn2}
 B_{n2} &= \frac{1}{\sqrt{n}} \sum_{i \in \mathrm{J}_{n}^{X}} \varphi^{(j)}(\Ub_{i})
%     f_{j\eps}(\eps_{ji})\eps_{ji}
   f_{j\eps}\big(F_{j\eps}^{-1}(U_{ji})\big)F_{j\eps}^{-1}(U_{ji})
 \,\big[\tfrac{
s_{j}(\Xb_{i}; \thetab_{j}) - s_{j}(\Xb_{i}; \widehat{\thetab}_{j})}{s_{j}(\Xb_{i}; \widehat{\thetab}_{j})}\big]. 
\end{align*}
Now similarly as in the proof of Lemma~\ref{lemma about hatFjhateps} 
one can show that 
% with the help of assumption~~$\boldsymbol{(m s)}$   
%
\begin{align}
\notag 
 B_{n1} &= \sqrt{n}(\thetab_{j} -\widehat{\thetab}_{j})\tr\,\frac{1}{n} \sum_{i \in \mathrm{J}_{n}^{X}} \varphi^{(j)}(\Ub_{i})
    f_{j\eps}\big(F_{j\eps}^{-1}(U_{ji})\big)\,\tfrac{m'_{j}(\Xb_{i};\thetab_{j}) 
     }{s_{j}(\Xb_{i}; \widehat{\thetab}_{j})} + o_{P}(1)  \\ 
 \label{eq: Bn1} 
&= \sqrt{n}\,\big(\thetab_{j} - \widehat{\thetab}_{j}\big)\tr\,\E \big[\varphi^{(j)}(\Ub)
    f_{j\eps}\big(F_{j\eps}^{-1}(U_{j})\big)\big]\,\E\big[\tfrac{m'_{j}(\Xb;\thetab_{j}) 
     }{s_{j}(\Xb; \thetab_{j})}\big] + o_{P}(1)
\end{align}
and analogously also 
\begin{equation}
 \label{eq: Bn2}
 B_{n2} = \sqrt{n}\,\big(\thetab_{j} - \widehat{\thetab}_{j}\big)\tr\,\E \big[\varphi^{(j)}(\Ub)
    f_{j\eps}\big(F_{j\eps}^{-1}(U_{j})\big)\,F_{j\eps}^{-1}(U_{j})\big]\,\E\Big[\tfrac{s'_{j}(\Xb;\thetab_{j})}{s_{j}(\Xb;\thetab_{j})}\Big] + o_{P}(1).  
\end{equation}
% 
% where $m'_{j}(\xb; \tb)$ and  $s'_{j}(\xb; \tb)$ are the 
% first order derivatives of $m_{j}(\xb; \tb)$ and  $s_{j}(\xb; \tb)$ 
% with respect to $\tb$. 

Now \eqref{eq: An asymptotically}, \eqref{eq: Bn simplified}, \eqref{eq: Bn1} and \eqref{eq: Bn2} 
yields that $ B_{n} = - A_{n} + o_{P}(1)$,   which was to be proved. 
\end{proof}

The following lemma will be useful for copula families with `nicely bounded' 
score functions. 

\begin{lemma} \label{lemma: as normality of functions of ranks second} 
Suppose that assumptions $(\mathbf{F}_{j\eps})$ and $\boldsymbol{(ms)}$ hold. 
Let $\varphi$ be a $\widetilde{\mathcal{J}}^{0, 0}$-function such that 
$\E\, \{\varphi(\Ub)\} = 0$ and $\varphi^{(j)}$ is bounded for each $j \in \{1,\dotsc,p\}$. 
Then the statement of Lemma~\ref{lemma: as normality of functions of ranks} holds. 
\end{lemma}
\begin{proof}
By the mean value theorem
\begin{align*}
% \notag
 \frac{1}{\sqrt{n}} \suman \varphi(\Ubhat_{i})
%  &
=  \frac{1}{\sqrt{n}} \suman \varphi(\Ubtilde_{i})
% \\
% \label{eq: expansion of Jhat}
%  & 
% \qquad   
+ \sum_{j=1}^{d} \frac{1}{\sqrt{n}} \suman   \varphi^{(j)}(\Ub_{i}^{*})
    \big(\widehat{U}_{ji} - \Utilde_{ji}\big). 
\end{align*}
Now take $\lambda > 2(1 + \frac{1}{r})$ and recall the sets of indices 
$\mathrm{J}_{n}^{X}$ of $\mathrm{K}_{n}^{X}$ introduced in~\eqref{eq: JdeltanX and KdeltanX}. 
Then 
\begin{align}
 \notag 
\frac{1}{\sqrt{n}}& \suman   \varphi^{(j)}(\Ub_{i}^{*})
    \big(\widehat{U}_{ji} - \Utilde_{ji}\big) 
\\
% 
% \notag 
\label{eq: bounded score expansion}
&= 
 \frac{1}{\sqrt{n}} \sum_{i \in \mathrm{J}_{n}^{X}}   \varphi^{(j)}(\Ub_{i}^{*})
    \big(\widehat{U}_{ji} - \Utilde_{ji}\big)
 + \frac{1}{\sqrt{n}} \sum_{i \in \mathrm{K}_{n}^{X}}   \varphi^{(j)}(\Ub_{i}^{*})
    \big(\widehat{U}_{ji} - \Utilde_{ji}\big). 
\end{align}
Now with the help of Lemma~\ref{lemma Ujhat minus Ujtilde second} 
one can show that the second term on the right-hand side of 
\eqref{eq: bounded score expansion} can be bounded as 
the preceding equation is $o_{P}(1)$
% 
% \begin{align*}
\[
\frac{1}{\sqrt{n}} \sum_{i \in \mathrm{K}_{n}^{X}}   \big|\varphi^{(j)}(\Ub_{i}^{*})
    \big(\widehat{U}_{ji} - \Utilde_{ji}\big)\big| 
 \leq \frac{O_{P}(1)}{n} \sum_{i \in \mathrm{K}_{n}^{X}} \big(1+M_{j}(\Xb_{i})\big) = o_{P}(1), 
% \\
% 
% &\leq \frac{O_{P}(1)}{n} \suman \Big(\tfrac{2\,M_{j}^{r}(\Xb_{i})}{a_{n}^{r-1}}\Big)  = o_{P}(1). 
\]
% \end{align*}
where the last equation follows from Markov's inequality and 
\begin{align*}
 \E\bigg[ \frac{1}{n} \sum_{i \in \mathrm{K}_{n}^{X}} \big(1+M_{j}(\Xb_{i})\big) \bigg] 
 &= \E\big[\big(1+M_{j}(\Xb)\big)
\ind\big\{\Ub \not \in [\delta_n,1-\delta_n]^d 
\; \text{or} \; \max_{1 \leq j \leq d}M_{j}(\Xb) > a_{n}
\big\}\big] 
\\
&= o(1). 
\end{align*}
% \begin{align*}
%  \E\bigg[ \frac{1}{n} \sum_{i \in \mathrm{K}_{n}^{X}} \big(1+M_{j}(\Xb_{i})\big) \bigg] 
% &= \E\big[\big(1+M_{j}(\Xb_{i}\big)
% \ind\big\{\Ub_{i} \in [\delta_n,1-\delta_n]^d 
% \ \text{or} \ \max_{1 \leq j \leq d}M_{j}(\Xb_i) \geq a_{n}
% \big\}\big]
% \\
% % 
% &\leq \big(1+\E M_{j}(\Xb_{i}\big)\,\Pr\big(\Ub_{i} \in [\delta_n,1-\delta_n]^d\big)   
% \ind\big\{\Ub_{i} \in [\delta_n,1-\delta_n]^d 
% \ \text{or} \ \max_{1 \leq j \leq d}M_{j}(\Xb_i) \geq a_{n}
% \big\}\big]
% \end{align*}

Finally the first term on the right-hand side of 
\eqref{eq: bounded score expansion}
can be handled analogously as in the proof of Lemma~\ref{lemma: as normality of functions of ranks}. 
\end{proof}

\begin{corollary} \label{corollary: as normality of functions of ranks} 
Suppose that assumptions of Lemma~\ref{lemma: as normality of functions of ranks} 
or Lemma~\ref{lemma: as normality of functions of ranks second}  
are satisfied. Then 
\begin{equation*} 
%  \frac{1}{\sqrt{n}} \suman w_{n}(\Ubhat_{i})\,\varphi(\Ubhat_{i})
 \frac{1}{\sqrt{n}} \suman \varphi(\Ubhat_{i})  
  = \frac{1}{\sqrt{n}} \suman \widetilde{\varphi}(\Ub_{i}) + o_{P}(1),
\end{equation*}
where
\begin{align*}
 \widetilde{\varphi}(\ub) = \varphi(\ub)
  + \sum_{j=1}^{d} \int_{[0,1]^{d}} \big[\ind\{u_{j} \leq v_j\}-v_{j}\big] \varphi^{(j)}(\vb)\,\mathrm{d} C(\vb). 
\end{align*}
\end{corollary}
\begin{proof}
 With the help of~\eqref{eq: asympt equiv score statistics} it is sufficient to show that 
\begin{equation*} 
%  \frac{1}{\sqrt{n}} \suman w_{n}(\Ubhat_{i})\,\varphi(\Ubhat_{i})
 \frac{1}{\sqrt{n}} \suman \varphi(\Ubtilde_{i})  
  = \frac{1}{\sqrt{n}} \suman \widetilde{\varphi}(\Ub_{i}) + o_{P}(1). 
\end{equation*}
But this can be proved component-wise by mimicking 
the proof of Lemma~2 of \cite{ogvp_testing_rao_2017}, where 
the situation with $d=2$ but a more general $\varphi$ depending possibly 
also on $\Xb_{i}$ is considered. 
%  for general $d$ but no $\Xb_{i}$. 
\end{proof}

\subsection{Proofs of Theorems~\ref{thm equiv of MPL estim} and~\ref{thm equiv of MPL estim for nice copulas}} 
\label{subsec: proof of the main theorem}

\begin{proof}[Proof of Theorem~\ref{thm equiv of MPL estim}]
With the help of Lemmas~\ref{lemma: convergence in probability of functions of ranks} 
and~\ref{lemma: as normality of functions of ranks} the proof can closely follow the 
proof of Lemma~3 in \cite{ogvp_testing_rao_2017}. 
In order to do that define  
\begin{equation} \label{eq: Wn}
\boldsymbol{W}_n(\ab) =  \frac{1}{n}\suman
  \skorb\big(\Ubhat_{i}; \ab\big) 
\qquad \text{and} \qquad 
\boldsymbol{W}(\ab) =  \Es \skorb\big(\Ub; \ab\big).  
\end{equation}
In what follows we show that assumptions of Theorem~A.10.2 of~\citet{BKRW1993} 
are satisfied for $\boldsymbol{W}_n$ and $\boldsymbol{W}$ given by 
\eqref{eq: Wn}.

It follows  from the standard maximum likelihood theory 
that Assumption~(GM0) is satisfied thanks to Assumption~\ref{assump:identifiability}. 
Moreover, Assumptions~\ref{assump:differentiability} and~\ref{assump:invertibility} imply  Assumption~(GM3). Assumption~(GM2) is also satisfied as
thanks to assumption~\ref{assump:snice}  
one can for each $k \in \{1,\dotsc,p\}$ apply Corollary~\ref{corollary: as normality of functions of ranks} 
to $\varphi(\ub) = \psi_{k}(\ub; \alfab)$ 
% $\frac{1}{\sqrt{n}}\suman \skorb_{k}\big(\Ubhat_{i}; \ab\big)$ 
and get 
\[
\frac{1}{n}\suman \skorb
\big(\Ubhat_{i}; \alfab\big)
= \frac{1}{n} \suman \widetilde{\boldsymbol{\psi}}(\Ub_{i}) + o_{P}(n^{-1/2}),
\]
where $\widetilde{\boldsymbol{\psi}}(\ub)$ was introduced in 
Corollary~\ref{cor as normality of mpl}.
% = \big(\widetilde{\psi}_{1}(\ub),\dotsc,\widetilde{\psi}_{p}(\ub)\big)\tr$ 
% % with  
% % 
% \begin{equation} \label{eq: elements of psi}
%  \widetilde{\psi}_{k}(\ub) =  \psi_{k}(\ub; \alfab) 
%   + \sum_{j=1}^{d} \int_{[0,1]^{d}} \big[\ind\{u_{j} \leq v_j\}-v_{j}\big] \psi_{k}^{(j)}(\vb; \alfab)\,\mathrm{d} C(\vb), \quad k=1,\dotsc,p. 
% \end{equation}

Thus, it remains to check Assumption~(U) from Theorem~A.10.2. Therefore for each $\eps > 0$ 
and for each $k, \ell \in \{1,\dotsc,p\}$, it is sufficient to find a~neighborhood 
$\mathcal{U}_{\eps} =\{\ab\in \mathcal{U}:\,\|\ab-\alfab\|<\eps\}$ such that
\[
 \sup_{\ab \in \mathcal{U}_{\eps}} \left|\frac{1}{n} \suman \frac{\partial \psi_{k}(\Ubhat_{i};\ab)}{\partial a_\ell}\,  - I^{(j, \ell)}(\ab)\right| \leq \epsilon + o_{P}(1),
\]
where $I^{(j, \ell)}(\ab)$ stands for the $(j, \ell)$~element of $I(\ab)$.

For simplicity of notation, let us put $g_{k, \ell}(\ub; \ab) =  \partial \psi_{k}(\ub;\ab)/\partial a_\ell$. Assumption~\ref{assump:differentiability} allows to adapt Lemma~\ref{lemma: convergence in probability of functions of ranks}, which gives
$$
 \frac{1}{n} \suman g_{k,\ell}(\Ubhat_{i}; \alfab) - I^{(k, \ell)}(\alfab)
 = o_{P}(1).
$$
Hence, it remains to show
\begin{equation} \label{eq: Dn}
 D_{n} = \sup_{\ab \in \mathcal{U}_{\eps}} \left|\frac{1}{n} \suman g_{k, \ell}(\Ubhat_{i};\ab) - \frac{1}{n} \suman g_{k, \ell}(\Ubhat_{i};\alfab) \right| \leq \epsilon + o_{P}(1).
\end{equation}
For a~given $\delta\in(0,1/4)$ (that will be specified later on), let us introduce the sets $\mathbf{I}_{\delta}$ and $\mathrm{J}_{\delta}$ as in
\eqref{eq: Idelta} and~\eqref{eq: Jdelta}.
Then the left-hand side of~\eqref{eq: Dn} can be bounded by
% \begin{align}
\begin{equation}
\label{eq: bound on Dn 1}
 D_{n} \leq \sup_{\ab \in \mathcal{U}_{\eps}} 
  \left|\frac{1}{n} \sum_{i \in  \mathrm{J}_{\delta} \cap \mathrm{J}_{n}^{X}} g_{k, \ell}(\Ubhat_{i};\ab) 
 - \frac{1}{n} \sum_{i \in  \mathrm{J}_{\delta}\cap \mathrm{J}_{n}^{X}} g_{k, \ell}(\Ubhat_{i}; \alfab)\right| 
+ \frac{2}{n} \sum_{i \not \in  \mathrm{J}_{\delta}\cap \mathrm{J}_{n}^{X}} h(\Ubhat_{i}), 
\end{equation}
where $\mathrm{J}_{n}^{X}$ was introduce in~\eqref{eq: JdeltanX and KdeltanX} and $h$ in Assumption \ref{assump:differentiability}.   
Now with probability going to one for each sufficiently large~$n$,
if $\Ub_{i} \in \mathbf{I}_{\delta}$,
then $\Ubhat_{i} \in \mathbf{I}_{\delta/2}$.
Thus for each $\delta\in(0,1/4)$ the term on the right-hand side of \eqref{eq: bound on Dn 1} can be made arbitrarily small (Assumption~\ref{assump:differentiability})
up to $o_{P}(1)$ term by considering a~sufficiently small neighbourhood~$\mathcal{U}_{\eps}$.

Finally, analogously as in the proof of Lemma~\ref{lemma: convergence in probability of functions of ranks}, one can show that
\begin{equation*}
% \label{eq: bound on h on a small set}
 \frac{1}{n} \sum_{i \not \in  \mathrm{J}_{\delta}\cap \mathrm{J}_{n}^{X}} h(\Ubhat_{i})
  \leq r(\delta), 
\end{equation*}
where $r(\delta) \to 0$ as $\delta \to 0_{+}$.  

Thus we have verified the assumptions of Theorem~A.10.2 of~\citet{BKRW1993} 
which yields that there exists a consistent root 
(say $\widehat{\alfab}_{n}$)
of the estimating equation~\eqref{eq: MPL based on residuals} which has 
the following  asymptotic representation 
\[
  \sqrt{n}\,\big(\widehat{\alfab}_{n} - \alfab \big) = \{ I(\alfab)\}^{-1}
 \frac{1}{\sqrt{n}}\suman
  \widetilde{\boldsymbol{\psi}}\big(\Ub_{i})
   + o_{P}(1),   
\]
where the elements of the vector function $\widetilde{\boldsymbol{\psi}}$  are given in 
\eqref{eq: elements of psi}. 
Note that completely analogously one can show that there exists 
a consistent root (say $\widetilde{\alfab}_{n}$)
of the estimating equation~\eqref{eq: MPL based on errors} 
which has the  same asymptotic representation. This finally 
implies the statement of the theorem. 
\end{proof}

\begin{proof}[Proof of Theorem~\ref{thm equiv of MPL estim for nice copulas}] 
The proof is completely analogous to the proof of Theorem~\ref{thm equiv of MPL estim for nice copulas}. The only difference is that one uses Lemma~\ref{lemma: as normality of functions of ranks second} instead of Lemma~\ref{lemma: as normality of functions of ranks}. In fact the proof is 
even simpler as  thanks to assumption~\ref{assump:very nice} one can take a finite 
constant instead of the function~$h$. 
\end{proof}

% % 
\renewcommand{\theequation}{B\arabic{equation}}
\setcounter{equation}{0}
% % % %

\section{\texorpdfstring{Some results on $\hatF_{j\hateps}$ and $\Uhat_{ji}$}{Some results on ecdf of residuals and estimated pseudo-observations}}

In what follows let $x_{+} = \max\{x,0\}$. 

\begin{lemma} \label{lemma about hatFjhateps}
Suppose that assumptions $(\mathbf{F}_{j\eps})$ and $\boldsymbol{(ms)}$ 
hold. Then for $\delta_n = n^{-1/\lambda}$ where 
% $\lambda > \max\{\tfrac{2r - 2\beta(r-1)}{r-1}, 2+2\beta\}$
$\lambda > 2(1 - \beta + \tfrac{1}{r-1})$   
it holds uniformly in $u \in [\delta_n/2,1-\delta_n/2]$ 
\begin{align}
 \notag 
 \hatF_{j\hateps}\big(F_{j\eps}^{-1}(u)\big) 
 &= \hatF_{j\eps}\big(F_{j\eps}^{-1}(u)\big)
% \\  
%  + f_{j\eps}(F_{j\eps}^{-1}(u))\Big\{\E_{\Xb}\big[\tfrac{m_{j}(\Xb, \hatthetab_{j}) - m_{j}(\Xb, \hatthetab_{j})}{s_{j}(\Xb; \thetab_{j})}\big] + F_{j\eps}^{-1}(u)\,\E_{\Xb}\big[\tfrac{s_{j}(\Xb, \hatthetab_{j}) - s_{j}(\Xb; \thetab_{j})}{s_{j}(\Xb; \thetab_{j})}\big] \Big\}
%   + u^{\beta}(1-u)^{\beta} o_{P}\big(\tfrac{1}{\sqrt{n}}\big)
% \\
%  = 
% \hatF_{j\eps}\big(F_{j\eps}^{-1}(u)\big) 
  + f_{j\eps}(F_{j\eps}^{-1}(u))\,\E_{\Xb}\Big[\tfrac{m'_{j}(\Xb;\thetab_{j})}{s_{j}(\Xb;\thetab_{j})}
+ F_{j\eps}^{-1}(u)\,\tfrac{s'_{j}(\Xb;\thetab_{j})}{s_{j}(\Xb;\thetab_{j})}\Big]\tr
 (\widehat{\thetab}_{j}-\thetab_{j})
\\
\label{eq: approximation of hatFjhateps}
  & \qquad  + u^{(\beta - \gamma)_{+}}(1-u)^{(\beta - \gamma)_{+}} o_{P}\big(\tfrac{1}{\sqrt{n}}\big)  
\end{align}
for each $\gamma > 0$ and $j \in \{1,\dotsc,d\}$. 
\end{lemma}
\begin{proof}
We will show the statement for $u \in [\frac{\delta_n}{2},\tfrac{1}{2}]$. The proof would 
be completely analogous for $u \in [\tfrac{1}{2},1-\frac{\delta_n}{2}]$. 

Note that 
\[
 \hatF_{j\hateps}\big(F_{j\eps}^{-1}(u)\big)  = 
  \frac{1}{n} \sum_{i=1}^{n} \ind\Big\{\eps_{ji} \leq 
  \tfrac{m_{j}(\Xb_{i}; \hatthetab_j) - m_{j}(\Xb_{i}; \thetab_j)}{s_{j}(\Xb_{i}, \thetab_{j})}
 + \tfrac{F_{j\eps}^{-1}(u)s_{j}(\Xb_{i}; \hatthetab_j)}{s_{j}(\Xb_{i}; \thetab_j)}\Big\}. 
\]
In what follows we need to take care of the fact that the majorant $M_{j}(\xb)$ from 
assumption~$\boldsymbol{(ms)}$ can be unbounded. 
Let $a_{n} = n^{1/(\lambda_{x}r)}$, where $\lambda_{x}$ will be specified 
later. Then similarly as in the proof of Lemma~\ref{lemma bound on prob of large Mj} one can use Markov's inequality to bound 
\begin{align}
\notag 
 \Big|&\hatF_{j\hateps}\big(F_{j\eps}^{-1}(u)\big) - 
\frac{1}{n} \sum_{i=1}^{n} \ind\Big\{\eps_{ji} \leq 
  \tfrac{m_{j}(\Xb_{i}; \hatthetab_j) - m_{j}(\Xb_{i}; \thetab_j)}{s_{j}(\Xb_{i}, \thetab_{j})}
 + \tfrac{F_{j\eps}^{-1}(u)s_{j}(\Xb_{i}; \hatthetab_j)}{s_{j}(\Xb_{i}; \thetab_j)}, M_{j}(\Xb_i) \leq a_n\Big\} \Big|
\\
\notag 
% \label{eq: Fjhateps aprox}
& \qquad  \leq 
\frac{1}{n} \sum_{i=1}^{n} \ind\big\{M_{j}(\Xb_i) > a_n\big\} 
 \leq  \frac{1}{n} \sum_{i=1}^{n} \tfrac{M_{j}^{r}(\Xb_i)}{a_{n}^{r}} \ind\big\{M_{j}(\Xb_i) > a_n\big\} 
= o_{P}\big(\tfrac{1}{n^{1/\lambda_{x}} }\big).   
\end{align}
Note that thanks to the assumption $\lambda > 2(1 - \beta + \tfrac{1}{r-1})$ it is straightforward to verify that 
$
 \tfrac{1}{2} + \tfrac{\beta}{\lambda} 
 < r\big(\tfrac{1}{2} - \tfrac{1-\beta}{\lambda}\big).  
$
In the following we will take $\lambda_{x}$ such that 
\begin{equation} \label{eq: lambdax}
 \tfrac{1}{2} + \tfrac{\beta}{\lambda} 
 < \tfrac{1}{\lambda_{x}} 
 < r\big(\tfrac{1}{2} - \tfrac{1-\beta}{\lambda}\big).  
\end{equation}
Now with the help of \eqref{eq: lambdax} 
one can conclude that 
\begin{align}
 \notag 
 \hatF_{j\hateps}\big(F_{j\eps}^{-1}(u)\big) &= 
\frac{1}{n} \sum_{i=1}^{n} \ind\Big\{U_{ji} \leq 
  F_{j\eps}\big(\tfrac{m_{j}(\Xb_{i}; \hatthetab_j) - m_{j}(\Xb_{i}; \thetab_j)}{s_{j}(\Xb_{i}; \thetab_{j})}
 + \tfrac{F_{j\eps}^{-1}(u)s_{j}(\Xb_{i}; \hatthetab_j)}{s_{j}(\Xb_{i}; \thetab_j)}
\big), M_{j}(\Xb_i) \leq a_n\Big\} 
\\
\label{eq: approx of hatFjhateps} 
%  + o\big(\tfrac{1}{n^{1/(2(1-\beta))}} \big). 
& \qquad + u^{(\beta - \gamma)_{+}}(1-u)^{(\beta - \gamma)_{+}}o_{P}(n^{-1/2}),      
\end{align}
for $u \in [\delta_n/2, 1/2]$.

Now for simplicity of notation introduce 
\begin{equation} \label{eq: yjx}
 y_{j\xb}(\tb,u) = \tfrac{m_{j}(\xb;\tb) - m_{j}(\xb;\thetab_j)}{s_{j}(\xb;\thetab_j)}   + F_{j\eps}^{-1}(u) \tfrac{s_{j}(\xb;\tb)}{s_{j}(\xb;\thetab_j)}.  
\end{equation}
Further for $u \in (0,1]$ and $\tb \in \RR^{p_{j}}$ put 
\begin{equation*} %\label{eq: wu}
 w(u) = \min\{u, 1-u\}^{(\beta - \gamma)_{+}}, 
% = \min\{u^{(\beta - \gamma)_{+}}, (1-u)^{(\beta - \gamma)_{+}}\}, 
 \quad \text{and} \quad \tbn = \thetab_{j} + \tb/n^{1/2-\eta}, 
\end{equation*}
where $\eta > 0$ is sufficiently small. Note that the 
function $w$ is increasing on $(0,\frac{1}{2})$ and 
decreasing on $(\frac{1}{2},1)$ for $\beta - \gamma > 0$. 
Finally let
\[
 u^{(n)} = \max\{u,\delta_n/2\}
\]
and for $i \in \{1,\dotsc,n\}$ introduce 
the processes 
\begin{equation*}
  Z_{ni}(\tb,u) = \tfrac{1}{w(u^{(n)})\sqrt{n}}
% \, \ind\big\{u \geq \delta_n/2\big\}\,
 \ind\big\{U_{ji} \leq F_{j\eps}\big(y_{j\Xb_{i}}(\tbn,u^{(n)})\big), |M_{j}(\Xb_{i})| \leq a_n \big\}
\end{equation*}
that are indexed by the set $\mathcal{F} = T_{1} \times (0,1/2]$, 
where $T_{1} = \big\{\tb \in \RR^{p_j}: \|\tb\| \leq 1\big\}$. 

% the set of functions 
% % 
% \begin{align}
% \notag 
%  \mathcal{F}_{n} &= \Big\{(\xb,v) \mapsto 
% %  w^{-1}(u)\ind\big\{v \leq F_{j\eps}\big(\tfrac{m_{j}(\xb;\tbn) - m_{j}(\xb;\thetab_j)}{s_{j}(\xb;\thetab_j)} + F_{j\eps}^{-1}(u) \tfrac{s_{j}(\xb;\tbn)}{s_{j}(\xb;\thetab_j)} \big), |M_{j}(\xb)| \leq a_n \big\}\, \ind\big\{u \in [\delta_n/2, 1-\delta_n/2]\big\},   
%  w^{-1}(u)\ind\big\{v \leq F_{j\eps}\big(y_{j\xb}(\tbn,u)\big), |M_{j}(\xb)| \leq a_n \big\}\, \ind\big\{u \geq \delta_n/2\big\},   
% \\
%   & \qquad \qquad \|\tb\| \leq 1, u \in (0,1/2],   
%  \Big\}  
% \end{align}
% 
Note that assumption $\boldsymbol{(ms)}$ guarantees that 
$n^{1/2-\eta}(\widehat{\thetab}_{j} - \thetab_{j}) \inPr 0$ 
for each $\eta \in (0,\frac{1}{2})$, which further implies that  
$\Pr(\|n^{1/2-\eta}(\widehat{\thetab}_{j} - \thetab_{j})\| \leq  1) \ntoinfty 1$. 
Put 
\[
 \widehat{\boldsymbol{\vartheta}}_{n} = n^{1/2-\eta}(\widehat{\thetab}_{j} - \thetab_{j}).
\]
Then with the help of \eqref{eq: approx of hatFjhateps} one can (with probability 
going to one) write that for $u \in [\delta_n/2,1/2]$
\begin{equation} \label{eq: hatFhateps via empir process}
 \hatF_{j\hateps}\big(F_{j\eps}^{-1}(u)\big) = \frac{w(u)}{\sqrt{n}}\suman Z_{ni}(\widehat{\boldsymbol{\vartheta}}_{n},u) 
+ w(u)\,o_{P}(n^{-1/2}).   
% + u^{(\beta - \gamma)_{+}}(1-u)^{(\beta - \gamma)_{+}}o_{P}(n^{-1/2}),  
\end{equation}
% 
% 
% % 
% \begin{equation} \label{eq: hatFhateps via empir process}
%  \hatF_{j\hateps}\big(F_{j\eps}^{-1}(u)\big) = \Pr_{n}(f_{\widehat{\boldsymbol{\vartheta}}_{n},u}) + u^{(\beta - \gamma)_{+}}(1-u)^{(\beta - \gamma)_{+}}o_{P}(n^{-1/2}),  
% \end{equation}
% % 
% where \citep[in agreement with the notation used e.g.][]{vaart_wellner_2007} 
% $\Pr_{n}$ stands for the empirical measure. 

Now equip the space $\mathcal{F}$ with the semimetric $\rho$ given by 
\begin{equation} \label{eq: rho}
 \rho\big((\tb_1,u_1), (\tb_2,u_2)\big) = K\,\sqrt{\|\tb_1 - \tb_2\| 
 + \tfrac{u_{2} - u_{1}}{w^2(u_2)} + \big(\tfrac{1}{w(u_1)}-\tfrac{1}{w(u_2)}\big)^2 
  u_1}, \quad \text{for} \quad u_1 \leq u_2,  
%  + \frac{|u_1 - u_{2}|}{w^2(u_1 \vee u_2)} + \Big(\frac{1}{w(u_1)}-\frac{1}{w(u_2)}\Big)^2 
%   u_1 \wedge u_2, 
\end{equation}
where $K$ is a finite constant that will be specified afterwards. 

Later we show that the assumptions of Theorem 2.11.11 of \cite{vaart_wellner} are satisfied 
for the empirical process indexed by $\mathcal{F}$, which implies that the 
process is asymptotically tight. Further as 
$\sup_{u \in (0,\frac{1}{2}]} 
 \rho\big((\widehat{\boldsymbol{\vartheta}}_{n},u), (\boldsymbol{0},u)\big) = o_{P}(1)$, 
one gets that uniformly in $u \in (0,1/2]$ 
% $u \in [\delta_{n}/2,1-\delta_{n}/2]$ 
% 
\begin{equation} \label{eq: using asympt tightness}
  \suman Z_{ni}(\widehat{\boldsymbol{\vartheta}}_{n},u) - \suman Z_{ni}(\boldsymbol{0},u) 
 - \suman \E_{U,\Xb}\,\big[Z_{ni}(\widehat{\boldsymbol{\vartheta}}_{n},u)  - Z_{ni}(\boldsymbol{0},u)\big]  
 = o_{P}(1),   
\end{equation}
where $\E_{U,\Xb}$ stands for the expectation with respect 
to $U_{ji}$'s and $\Xb_{i}$'s (while 
considering $\widehat{\boldsymbol{\vartheta}}_{n}$ being fixed).  

In what follows we concentrate on $u \in [\delta_n/2,1/2]$. If not stated 
otherwise all the 
following results hold uniformly for $u$ from this interval. 
% $u \in [\delta_n/2,1/2]$. 

Note that similarly as in~\eqref{eq: hatFhateps via empir process} one can argue that 
% for $u \in [\delta_n/2,1-\delta_n/2]$ 
% 
% 
\begin{equation*} %\label{eq: hatFhateps via empir process}
 \hatF_{j\eps}\big(F_{j\eps}^{-1}(u)\big) = \frac{w(u)}{\sqrt{n}}\suman Z_{ni}(\boldsymbol{0},u) 
+ w(u)\,o_{P}\big(\tfrac{1}{\sqrt{n}}\big).   
% + u^{(\beta - \gamma)_{+}}(1-u)^{(\beta - \gamma)_{+}}o_{P}(n^{-1/2}),  
\end{equation*}
This together with \eqref{eq: approx of hatFjhateps} and \eqref{eq: using asympt tightness} 
implies 
\begin{equation} \label{eq: Pn of f}
\hatF_{j\hateps}\big(F_{j\eps}^{-1}(u)\big) 
=  \hatF_{j\eps}\big(F_{j\eps}^{-1}(u)\big)
 + w(u)\,\sqrt{n}\,\E_{U,\Xb}\big[Z_{n1}(\widehat{\boldsymbol{\vartheta}}_{n},u)  - Z_{n1}(\boldsymbol{0},u)\big]   
% \\
% 
 + w(u)\,o_{P}\big(\tfrac{1}{\sqrt{n}}\big).   
%   + u^{(\beta-\gamma)_{+}}(1-u)^{(\beta-\gamma)_{+}} o_{P}\big(\tfrac{1}{\sqrt{n}}\big).   
\end{equation}
Thus to finish the proof it remains to deal with the second term on the 
right-hand side of \eqref{eq: Pn of f}. As $\sqrt{n}\,(\hatthetab_j - \thetab_j) = O_{P}(1)$
one can use the mean value theorem which guarantees  
that (with probability going to one)  there exists $\tb_{*} \in T_{1}$ such that  
\begin{align}
\notag 
%   \Pr&\big(f_{\widehat{\boldsymbol{\vartheta}}_{n},u} - f_{\boldsymbol{0},u}\big)  
w(u)&\sqrt{n}\,\E_{U,\Xb}\big[Z_{n1}(\widehat{\boldsymbol{\vartheta}}_{n},u)  - Z_{n1}(\boldsymbol{0},u)\big]   
% = 
%  \EX\Big[ \big[ F_{j\eps}\big(y_{j\Xb}(\widehat{\boldsymbol{\vartheta}}_{n},u)\big) 
%  - u\big]\,\ind\big\{M_{j}(\Xb) 
%  \leq a_{n} \big\}\Big]
\\
\notag 
&=  \EX\Big[ \big[ F_{j\eps}\big(y_{j\Xb}(\hatthetab_j,u)\big) - u\big]\,\ind\big\{M_{j}(\Xb) 
 \leq a_{n} \big\}\Big]
% \]
\\
% 
% \notag 
& = \E_{\Xb}\Big[ f_{j\eps}(y_{j\Xb}(\tbn_{*},u)) \Big(\tfrac{m'_{j}(\Xb;\tbn_{*})}{s_{j}(\Xb;\thetab_{j})} 
  + F_{j\eps}^{-1}(u) \tfrac{s'_{j}(\Xb;\tbn_{*})}{s_{j}(\Xb;\thetab_{j})}\Big)^{\!\!\mathsf{T}}
 \ind\big\{M_{j}(\Xb) 
 \leq a_{n} \big\} \Big] 
\big(\hatthetab_j - \thetab_{j}\big).  
% \\
% 
\label{eq: P of f}
% & \qquad + u^{\beta}(1-u)^{\beta}o_{P}(n^{-1/2}),
% \\
% % 
% & =  f_{j\eps}\big(F_{j\eps}^{-1}(u)\big) 
% \E_{\Xb}\,\Big[\tfrac{m'_{j}(\Xb;\thetab_{j})}{s_{j}(\Xb;\thetab_{j})} 
%   + F_{j\eps}^{-1}(u) \tfrac{s'_{j}(\Xb;\thetab_{j})}{s'_{j}(\Xb;\thetab_{j})} 
% \Big] \big(\hatthetab_j - \thetab_{j}\big)     
% \\
% % 
% & \qquad + u^{\beta}(1-u)^{\beta}o_{P}(n^{-1/2}),
\end{align} 
Note that for $\xb$ such that $M_{j}(\xb) \leq a_{n}$ one 
has 
\begin{equation} \label{eq: mj at tb minus mj at theta}
 \big|\tfrac{m_{j}(\xb;\tbn) - m_{j}(\xb;\thetab_j)}{s_{j}(\xb;\thetab_j)}\big| 
 \leq M_{j}(\xb)\|\tbn - \thetab_{j}\| \leq a_{n}\, n^{-1/2+\eta} 
  = n^{1/(\lambda_{x} r) - 1/2+\eta} 
\end{equation}
and also 
\begin{equation} \label{eq: sj at tb dividided by sj at theta}
 \big|\tfrac{s_{j}(\xb;\tbn)}{s_{j}(\xb;\thetab_j)} - 1\big| 
\leq  M_{j}(\xb)\|\tbn - \thetab_{j}\| \leq a_{n}\, n^{-1/2+\eta} 
= n^{1/(\lambda_{x} r) - 1/2+\eta}, 
\end{equation}
where both inequalities hold uniformly in $\tb \in T_{1}$ 
and $\xb \in \{\tilde{\xb}: M_{j}(\tilde{\xb})\leq a_{n}\}$. Thus with the help 
of Lemma~\ref{lemma about cont of fjeps} 
\begin{equation} \label{eq: fjeps et yjx divided by u to beta}
 \sup_{\tb_{*} \in T_{1}} 
 \sup_{\xb \in \{\tilde{\xb}: M_{j}(\tilde{\xb})\leq a_{n}\}} 
% \sup_{u \in [\delta_n/2,1-\delta_n/2]}
\sup_{u \in [\delta_n/2,1/2]}  
\frac{\big|f_{j\eps}(y_{j\xb}(\tbn_{*},u)) - f_{j\eps}(F_{j\eps}^{-1}(u))\big|}
 {u^{(\beta - \gamma)_{+}}(1-u)^{(\beta - \gamma)_{+}}} = o_{P}(1) 
\end{equation}
and also 
\begin{equation} \label{eq: fjeps et yjx times Fjinv divided by u to beta}
 \sup_{\tb_{*} \in T_{1}} 
 \sup_{\xb \in \{\tilde{\xb}: M_{j}(\tilde{\xb})\leq a_{n}\}} 
% \sup_{u \in [\delta_n/2,1-\delta_n/2]} 
\sup_{u \in [\delta_n/2,1/2]}  
\frac{\big|f_{j\eps}(y_{j\xb}(\tbn_{*},u))F_{j\eps}^{-1}(u) - f_{j\eps}(F_{j\eps}^{-1}(u))F_{j\eps}^{-1}(u)\big|}
 {u^{(\beta - \gamma)_{+}}(1-u)^{(\beta - \gamma)_{+}}} = o_{P}(1).
\end{equation}
Now combining the above findings with assumption $\boldsymbol{(ms)}$ yields that 
\eqref{eq: P of f} can be simplified to 
\begin{align*} 
%   \Pr\big(f_{\hatthetab_j,u}\big) 
w(u)\sqrt{n}&\,\E_{U,\Xb}\big[Z_{n1}(\widehat{\boldsymbol{\vartheta}}_{n},u)  - Z_{n1}(\boldsymbol{0},u)\big]    
\\
& =  f_{j\eps}\big(F_{j\eps}^{-1}(u)\big) 
\E_{\Xb}\,\Big[\tfrac{m'_{j}(\Xb;\thetab_{j})}{s_{j}(\Xb;\thetab_{j})} 
  + F_{j\eps}^{-1}(u) \tfrac{s'_{j}(\Xb;\thetab_{j})}{s_{j}(\Xb;\thetab_{j})} 
\Big]\tr \big(\hatthetab_j - \thetab_{j}\big)     
% \\
% 
% & \qquad 
% + u^{(\beta - \gamma)_{+}}(1-u)^{(\beta - \gamma)_{+}}\,o_{P}(n^{-1/2}),
+ w(u)\,o_{P}(n^{-1/2}),  
\end{align*}
which together with \eqref{eq: Pn of f} implies \eqref{eq: approximation of hatFjhateps}. 

\bigskip 

\noindent \textbf{Verifying assumptions of Theorem 2.11.11 of \cite{vaart_wellner}}

\smallskip  
First of all we need to show that the semimetric~$\rho$ defined in \eqref{eq: rho} is 
\textit{Gaussian-dominated}. To prove that it is sufficient to show that 
\citep[see p.~212 of][]{vaart_wellner}
\begin{equation} \label{eq: gauss dominancy}
 \int_{0}^{\infty} \sqrt{\log N(\epsilon, \mathcal{F}, \rho)}\,d\epsilon < \infty, 
\end{equation}
where $N(\epsilon, \mathcal{F}, \rho)$ is the covering number of~$\mathcal{F}$. 

It is known \citep[see Example~2.11.15 of][]{vaart_wellner} 
that \eqref{eq: gauss dominancy} holds true if 
$\mathcal{F}$ is replaced with $(0,1/2]$ and $\rho$ with 
\begin{equation} \label{eq: rho0}
 \rho_{0}(u_1, u_2) = 
 \sqrt{\tfrac{u_{2} - u_{1}}{w^2(u_2)} + \big(\tfrac{1}{w(u_1)}-\tfrac{1}{w(u_2)}\big)^2 
  u_1}, \quad \text{for} \quad u_1 \leq u_2,  
%  + \frac{|u_1 - u_{2}|}{w^2(u_1 \vee u_2)} + \Big(\frac{1}{w(u_1)}-\frac{1}{w(u_2)}\Big)^2 
%   u_1 \wedge u_2, 
\end{equation}
as $\rho_{0}$ is Gaussian. But from the definition 
of~$\rho$ in \eqref{eq: rho} it follows that one can bound 
\begin{align*}
 N\big(\epsilon, \mathcal{F}, \rho\big) &\leq N\big(\epsilon^2/(4K^2), T_{1}, \|\cdot\|\big)\, 
 N\big(\epsilon/(2K), (0,1/2], \rho_{0}\big) 
\\
&= O(\epsilon^{-2p_j})\, N\big(\epsilon/(2K), (0,1/2], \rho_{0}\big),  
\end{align*}
thus also $(\mathcal{F},\rho)$ satisfies \eqref{eq: gauss dominancy}.  

\medskip 

Next we need to check the three assumptions of  Theorem 2.11.11 of \cite{vaart_wellner}. 
As in our situations the processes $Z_{n1},\dotsc,Z_{nn}$ are identically distributed, the 
assumptions can be rewritten as follows. 
 
% \begin{enumerate}
\noindent \textit{(I) For each $\zeta > 0$
\begin{equation} \label{eq: FL type assump}
%  \suman \E \Big[\|Z_{ni}\|_{\mathcal{F}}\, \ind\big\{ \|Z_{ni}\|_{\mathcal{F}} > \zeta\big\} \Big] 
 n\,\E \Big[\|Z_{n1}\|_{\mathcal{F}}\, \ind\big\{ \|Z_{n1}\|_{\mathcal{F}} > \zeta\big\} \Big] 
\ntoinfty 0.  
\end{equation}
}

\noindent \textit{(II) For each $(\tb_1,u_1), (\tb_2,u_2) \in \mathcal{F}$  
\begin{equation} \label{eq: L2 continuity in rho assump}
%  \suman \E \big(Z_{ni}(\tb_2,u_2) -  Z_{ni}(\tb_1,u_1)\big)^2
 n\,\E \big(Z_{n1}(\tb_2,u_2) -  Z_{n1}(\tb_1,u_1)\big)^2    
 \leq \rho^{2}\big((\tb_2,u_2), (\tb_1,u_1)\big). 
\end{equation}
}

\noindent \textit{(III) For every $\rho$-ball $B(\epsilon) \subset \mathcal{F}$ of radius less than $\epsilon$ 
\begin{equation} \label{eq: weak second moment}
%  \sup_{v > 0} \suman v^{2}\, \Pr\bigg( \sup_{(\tb_1,u_1), (\tb_2,u_2) \in B(\epsilon)} 
%  \big|Z_{ni}(\tb_2,u_2) -  Z_{ni}(\tb_1,u_1)\big|  
 n\,\sup_{v > 0} v^{2}\, \Pr\bigg( \sup_{(\tb_1,u_1), (\tb_2,u_2) \in B(\epsilon)} 
 \big|Z_{n1}(\tb_2,u_2) -  Z_{n1}(\tb_1,u_1)\big|  
 >  v \bigg) \leq \epsilon^{2}. 
\end{equation}
}
%
 
% \end{enumerate}
 
Note that \textbf{the first assumption} \eqref{eq: FL type assump} is easy to check as 
\[
 \|Z_{n1}\|_{\mathcal{F}} \leq 
 \sup_{(\tb,u) \in \mathcal{F}} |Z_{n1}(\tb,u)| 
 \leq \sup_{u \in (0, \frac{1}{2}]} \tfrac{1}{\sqrt{n}\,w(u^{(n)})} 
 \leq \tfrac{1}{\sqrt{n}\,w(\delta_{n}/2)} \ntoinfty 0. 
\]
% 
%  for each $i \in \{1,\dotsc,n\}$. 

\medskip 

To verify \textbf{the second assumption} \eqref{eq: L2 continuity in rho assump} fix $\tb_1,\tb_2$ and $u_1,u_2$ (so that $u_1 \leq u_2$) and calculate 
\allowdisplaybreaks{ 
\begin{align}
\notag  \nobreak 
%   \suman & \E \big(Z_{ni}(\tb_2,u_2) -  Z_{ni}(\tb_1,u_1)\big)^2
  n\, & \Es \big(Z_{n1}(\tb_2,u_2) -  Z_{n1}(\tb_1,u_1)\big)^2    
\\
\notag \nobreak 
 &=  \E\Big[\Big( \tfrac{\ind\{U_{j} \leq F_{j\eps}(y_{j\Xb}(\tbn_{2},u_{2}^{(n)}))\}}{w(u_2^{(n)})} 
- \tfrac{\ind\{U_{j} \leq F_{j\eps}(y_{j\Xb}(\tbn_{1},u_{1}^{(n)}))\}}{w(u_1^{(n)})} 
\Big)^{2} \ind\{M_{j}(\Xb) \leq a_n\}
 \Big]
\\
\notag 
 &\leq  2\,\E\Big[\Big( \tfrac{\ind\{U_{j} \leq F_{j\eps}(y_{j\Xb}(\tbn_{2},u_{2}^{(n)}))\}}{w(u_2^{(n)})} 
- \tfrac{\ind\{U_{j} \leq F_{j\eps}(y_{j\Xb}(\tbn_{1},u_{1}^{(n)}))\}}{w(u_{2}^{(n)})} 
\Big)^{2} \ind\{M_{j}(\Xb) \leq a_n\}
 \Big]
\\*
\notag 
&\quad  + 2\,\E\Big[\Big( \tfrac{\ind\{U_{j} \leq F_{j\eps}(y_{j\Xb}(\tbn_{1},u_{1}^{(n)}))\}}{w(u_{2}^{(n)})} 
- \tfrac{\ind\{U_{j} \leq F_{j\eps}(y_{j\Xb}(\tbn_{1},u_{1}^{(n)}))\}}{w(u_{1}^{(n)})} 
\Big)^{2} \ind\{M_{j}(\Xb) \leq a_n\}
 \Big]
\\
% % 
\label{eq: sum of second moments of Zni} 
& = \tfrac{2}{w^{2}(u_2^{(n)})}
 \,\E\big[\big|F_{j\eps}\big(y_{j\Xb}(\tbn_{2},u_{2}^{(n)})\big) 
 -  F_{j\eps}\big(y_{j\Xb}(\tbn_{1},u_{1}^{(n)})\big) \big| 
  \ind\{M_{j}(\Xb) \leq a_n\} \big]
\\*
\notag  
& \quad + 2\,\Big(\tfrac{1}{w(u_2^{(n)})} - \tfrac{1}{w(u_1^{(n)})}\Big)^{2} 
 \,\E\big[F_{j\eps}(y_{j\Xb}\big(\tbn_{1},u_{1}^{(n)})\big)  
  \ind\{M_{j}(\Xb) \leq a_n\} \big]. 
% % 
\end{align}
}

Now we will have a look at \textit{the first term} on the right-hand side 
of \eqref{eq: sum of second moments of Zni}.  
For a given $u \in [\delta_n/2,1/2]$ 
% $u \in [\delta_n/2,1-\delta_n/2]$ 
by the mean value theorem there exists $\tb_{*}$ between $\tb_{1}$  
and $\tb_{2}$ such that 
\begin{align}
\notag   
 \E&\big[\big|F_{j\eps}(y_{j\Xb}\big(\tbn_{2},u)\big) 
 -  F_{j\eps}(y_{j\Xb}\big(\tbn_{1},u)\big) \big| 
  \ind\{M_{j}(\Xb) \leq a_n\} \big]
% \]
\\
\label{eq: diff of Fjeps at different tb} 
& \leq \E\,\Big[ f_{j\eps}(y_{j\Xb}(\tb_{*}^{(n)},u_{})) 
\Big(M_{j}(\Xb)  
% \Big(\tfrac{M_{j}(\Xb_{1})}{s_{j}(\Xb;\thetab_{j})} 
  + \big|F_{j\eps}^{-1}(u)
% \big|\tfrac{M_{j}(\Xb_{1})}{s_{j}(\Xb;\thetab_{j})}\Big)\Big]
 \big|M_{j}(\Xb)\Big)\Big]
\|\tbn_{1} - \tbn_{2}\| 
\end{align}
Now with the help of \eqref{eq: yjx}, \eqref{eq: mj at tb minus mj at theta}, 
\eqref{eq: sj at tb dividided by sj at theta}  and 
Lemma~\ref{lemma about yjx} one can conclude that 
with probability going to one
\begin{equation} \label{eq: bound on yjx}
  \sup_{\tb \in T_{1}} 
 \sup_{\xb \in \{\tilde{\xb}: M_{j}(\tilde{\xb})\leq a_{n}\}} y_{j\xb}(\tbn,u) 
 \leq F_{j\eps}^{-1}(2\,u), 
\end{equation}
which together with \eqref{eq: diff of Fjeps at different tb} 
implies that 
\begin{align}
% %
\notag 
% \tfrac{2}{w^{2}(u_2^{(n)})}&
\E&\big[\big|F_{j\eps}(y_{j\Xb}\big(\tbn_{2},u^{(n)})\big)  
 -  F_{j\eps}(y_{j\Xb}\big(\tbn_{1},u^{(n)})\big) \big| 
  \ind\{M_{j}(\Xb) \leq a_n\} \big] 
\\
% %
% \notag   
% & \leq \E \Big[f_{j\eps}(y_{j\Xb}(\tb^{*},u))(1 + |y_{j\Xb}(\tb^{*},u)|) 
% \tfrac{M_{j}(\Xb) 
% (1 + |F_{j\eps}^{-1}(u)|)
% }{1 + |y_{j\Xb}(\tb^{*},u)|}\Big] 
% \tfrac{\|\tb_{1} - \tb_{2}\|}{n^{1/2-\eta}}  
% \\
% %  
\label{eq: diff of Fjeps at different tb final}   
& \leq 
% \tfrac{2}{w^{2}(u_2^{(n)})} 
O(n^{-1/2+\eta})\,\E M_{j}(\Xb)\, \|\tb_{1} - \tb_{2}\|\, 
%  O\big((u^{(n)})^{\beta}(1-u^{(n)})^{\beta}\big)
 O\big((u^{(n)})^{\beta}\big)
 \leq O(n^{-1/2+\eta})\,\|\tb_{1} - \tb_{2}\|\, w(u^{(n)})
\end{align}
uniformly in $u$. 

Now fix $\tb$ and $\xb$.
%  and suppose that $u_{2} > \delta_n/2$.  
Then by the mean value theorem there exists $\tilde{u}$  
between $u_1^{(n)}$ and $u_2^{(n)}$ such that 
\begin{multline} 
\label{eq: diff of Fjeps at different u work}
 \big|F_{j\eps}\big(y_{j\xb}(\tbn,u_1^{(n)})\big) 
 - F_{j\eps}\big(y_{j\xb}(\tbn,u_2^{(n)})\big) \big| 
\\
\leq f_{j\eps}(y_{j\xb}(\tbn,\tilde{u})) \tfrac{s_{j}(\xb;\tbn)}{s_{j}(\xb;\thetab_{j})}
\tfrac{1}{f_{j\eps}(F_{j\eps}^{-1}(\tilde{u}))}|u_{1}^{(n)} - u_{2}^{(n)}|, 
\end{multline}

which together with 
\[
 \Big|\frac{s_{j}(\xb;\tbn)}{s_{j}(\xb;\thetab_{j})} \Big|
 = \Big|1+\frac{s_{j}(\xb;\tbn)-s_{j}(\xb;\thetab_j)}{s_{j}(\xb;\thetab_{j})} \Big| 
 \leq 1 + M_{j}(\xb)\|\tbn - \thetab_j\| 
 \leq 1 + a_{n}\, n^{-1/2+\eta},  
\]
assumption $(\mathbf{F}_{j\eps})$ and \eqref{eq: diff of Fjeps at different u work} implies that 
\begin{equation} \label{eq: diff of Fjeps at different u}
 \big|F_{j\eps} \big(y_{j\xb}(\tbn,u_1^{(n)})\big) 
 - F_{j\eps}\big(y_{j\xb}(\tbn,u_2^{(n)})\big) \big| 
 \leq O(1)\,|u_1^{(n)} - u_{2}^{(n)}|  
\end{equation}
uniformly in $\tb$ and $\xb$. 

Now combining the inequalities 
\eqref{eq: bound on yjx}, \eqref{eq: diff of Fjeps at different tb final}  
and~\eqref{eq: diff of Fjeps at different u} implies that 
% for all $\tb_1,\tb_2$, $u_1,u_2 \in (0,1/2]$  (and for all sufficiently large~$n$) 
% 
\begin{multline} \label{eq: bound on the first term in second moments of Zni}
\tfrac{1}{w^{2}(u_2^{(n)})}
 \,\E\big[\big|F_{j\eps}(y_{j\Xb}\big(\tbn_{2},u_{2}^{(n)})\big) 
 -  F_{j\eps}(y_{j\Xb}\big(\tbn_{1},u_{1}^{(n)})\big) \big| 
  \ind\{M_{j}(\Xb) \leq a_n\} \big]
\\
\leq 
\tfrac{ O(n^{-1/2+\eta}) \|\tb_{1} - \tb_{2}\|\,w(u_{2}^{(n)})}{w^{2}(u_2^{(n)})}
+ 
\tfrac{O(1)(u_{2}^{(n)} - u_{1}^{(n)})}{w^{2}(u_2^{(n)})}
= O(1)\Big(\|\tb_{1} - \tb_{2}\| + \tfrac{u_{2}^{(n)} - u_{1}^{(n)}}{w^{2}(u_2^{(n)})} \Big). 
\end{multline}

\smallskip 

Now turn our attention to \textit{the second term} on the right-hand side 
of \eqref{eq: sum of second moments of Zni}. Analogously as above one can bound 
\begin{align}
\notag 
\E&\,\big[F_{j\eps}(y_{j\Xb}\big(\tbn_{1},u_{1}^{(n)})\big)  
  \ind\{M_{j}(\Xb) \leq a_n\} \big]
\\ 
\notag 
& \leq \E\big[\big|F_{j\eps}(y_{j\Xb}\big(\tbn_{1},u_{1}^{(n)})\big)  
 - F_{j\eps}(y_{j\Xb}\big(\thetab_j,u_{1}^{(n)})\big) \big|
  \ind\{M_{j}(\Xb) \leq a_n\} \big]
\\
\notag 
 & \quad + \E\big[F_{j\eps}(y_{j\Xb}\big(\thetab_j,u_{1}^{(n)})\big)  
  \ind\{M_{j}(\Xb) \leq a_n\} \big] 
\\
\notag 
& \leq O(1) \|\tbn_{1}-\thetab_j\| + u_{1}^{(n)} = O(n^{-1/2+\eta})  + u_{1}^{(n)} 
 \leq 2\,u_{1}^{(n)}. 
\end{align}
Combining this with \eqref{eq: sum of second moments of Zni} and 
\eqref{eq: bound on the first term in second moments of Zni} 
one gets 
\begin{align*}
\notag  \nobreak 
%   \suman & \E \big(Z_{ni}(\tb_2,u_2) -  Z_{ni}(\tb_1,u_1)\big)^2
  n\, & \Es \big(Z_{n1}(\tb_2,u_2) -  Z_{n1}(\tb_1,u_1)\big)^2    
\leq O(1)\Big[\|\tb_{1} - \tb_{2}\| + \tfrac{u_{2}^{(n)} - u_{1}^{(n)}}{w^{2}(u_2^{(n)})}  + 
\Big(\tfrac{1}{w(u_2^{(n)})} - \tfrac{1}{w(u_1^{(n)})}\Big)^{2}
u_{1}^{(n)} 
\Big]
\\
&\leq O(1)\Big[\|\tb_{1} - \tb_{2}\| 
 + \rho_{0}^{2}\big(u_1^{(n)}, u_2^{(n)}\big)\Big]
\leq O(1)\Big[\|\tb_{1} - \tb_{2}\| 
 + 2\rho_{0}^{2}\big(u_1, u_2\big)\Big],  
\end{align*}
where the last inequality follows by Lemma~\ref{lemma about rho0}(iii) 
in Appendix~D.

\medskip 

Finally we show that also \textbf{the third assumption} \eqref{eq: weak second moment} is 
satisfied. Let $B(\epsilon)$ be a fixed $\epsilon$-ball. 
Then from the properties of the Euclidean norm and 
the function~$\rho_{0}$ (see Lemma~\ref{lemma about rho0}(iv) in Appendix~D),  
there exist $\tb_{0} \in T_{1}$ and $u_{L},u_{U} \in (0, \tfrac{1}{2}]$ such 
that 
\[
 \mathcal{B}(\epsilon) \subset T_{\epsilon} \times [u_{L},u_{U}], 
\quad \text{where} \quad T_{\epsilon} = \big\{\tb: \sqrt{\|\tb - \tb_{0}\|} \leq \tfrac{\epsilon}{K} \big\} 
\quad \text{and} \quad \rho_{0}(u_{L},u_{U}) < \tfrac{2\epsilon}{K}.  
\]
% % 
% and $u_{L}$ and $u_{U}$ are the smallest and the largest $u  \in [0,1/2]$ such that 
% % 
% \[
%  \rho_{0}(u_{L}, u_{U}) \leq \epsilon \quad \text{and} \quad 
%  \rho_{0}(u_{U}, u_{U}) \leq \epsilon,  
% \]
% % 
% with the semimetric $\rho_{0}$ given by \eqref{eq: rho0}. 

% Then with the help of Markov's inequality 
Then one can bound 
\begin{align} 
% 
% \notag 
%  \sup_{v > 0} & \suman v^{2}\, \Pr\bigg( \sup_{(\tb_1,u_1), (\tb_2,u_2) \in B(\epsilon)} 
%  \big|Z_{ni}(\tb_2,u_2) -  Z_{ni}(\tb_1,u_1)\big|  
%  \geq  v \bigg)
% \\
% 
\notag 
% & = 
n\, &\sup_{v > 0} v^{2}\, \Pr\bigg( \sup_{(\tb_1,u_1), (\tb_2,u_2) \in B(\epsilon)} 
 \big|Z_{ni}(\tb_2,u_2) -  Z_{ni}(\tb_1,u_1)\big|  
 >  v \bigg)
\\
%
% \notag  
 &\leq  2 \sup_{v > 0} v^{2}\, \Pr\bigg( 
% \sup_{(\tb,u) \in B(\epsilon)}
\sup_{(\tb,u) \in T_{\epsilon} \times [u_{L}, u_{U}]}  
 \sqrt{n}\,\big|Z_{ni}(\tb,u) -  Z_{ni}(\tb_{0},u_{U})\big|  
 >  v/2 \bigg)
% \\
% 
\label{eq: bound on weak second moment} 
%  &\leq 
% 2 \sup_{v > 0} v^{2}\, \Pr\bigg( \sup_{(\tb,u) \in T_{\epsilon} \times [u_{L},u_{U}]} 
%  \sqrt{n}\,\big|Z_{ni}(\tb,u) -  Z_{ni}(\tb_{0},u_{U})\big|  
%  >  v/4 \bigg)
% \\
% % 
% \notag 
% & \quad + 2 \sup_{v > 0} v^{2}\, \Pr\bigg( \sup_{(\tb,u) \in T_{\epsilon} \times [u_{U},u_{U}]} 
%  \sqrt{n}\,\big|Z_{ni}(\tb,u) -  Z_{ni}(\tb_{0},u_{U})\big|  
%  >  v/4 \bigg)
\end{align} 
% 

% Now we concentrate on the first term in \eqref{eq: bound on weak second moment}.  
To deal with the last probability introduce 
% in \eqref{eq: bound on weak second moment}.  
% For this reason introduce 
% 
\[
 G_{j\xb}^{(L)} = \inf_{(\tb, u) \in T_{\epsilon} \times [u_{L},u_{U}]} F_{j\eps}\big(y_{j\xb}(\tbn,u^{(n)})\big), 
 \quad 
 G_{j\xb}^{(U)} = \sup_{(\tb, u) \in T_{\epsilon} \times [u_{L},u_{U}]} F_{j\eps}\big(y_{j\xb}(\tbn,u^{(n)})\big).   
\]
Then one can bound 
% % 
% \begin{align}
% \notag 
% &\sup_{(\tb_1,u_1), (\tb_2,u_2) \in B(\epsilon)} 
%  \sqrt{n}\,\big|Z_{n1}(\tb_2,u_2) -  Z_{n1}(\tb_1,u_1)\big|  
% \\
% % 
% \notag 
% & = \sup_{(\tb_1,u_1), (\tb_2,u_2) \in B(\epsilon)} 
%  \Big| 
% % \, \ind\big\{u \geq \delta_n/2\big\}\,
% \tfrac{\ind\{U_{j1} \leq F_{j\eps}(y_{j\Xb_{1}}(\tbn_{2},u_{2}^{(n)}))\}}{w(u_{2}^{(n)})}
% - \tfrac{\ind\{U_{j1} \leq F_{j\eps}(y_{j\Xb_{1}}(\tbn_{1},u_{1}^{(n)}))\}}{w(u_{1}^{(n)})}
%  \Big| 
% \ind\big\{M_{j}(\Xb_{1}) \leq a_n\big\} 
% \\
% % 
% \notag 
% & \leq \sup_{(\tb_1,u_1), (\tb_2,u_2) \in B(\epsilon)} 
%  \Big| 
% \tfrac{\ind\{U_{j1} \leq F_{j\eps}(y_{j\Xb_{1}}(\tbn_{2},u_{2}^{(n)}))\} - \ind\{U_{j1} \leq F_{j\eps}(y_{j\Xb_{1}}(\tbn_{1},u_{1}^{(n)}))\}}{w(u_{2}^{(n)})}
% \Big|\,\ind\big\{M_{j}(\Xb_{1}) \leq a_n\big\}  
% \\
% % 
% \notag 
% & 
% \quad +  \sup_{(\tb_1,u_1), (\tb_2,u_2) \in B(\epsilon)} 
%  \Big| \tfrac{\ind\{U_{j1} \leq F_{j\eps}(y_{j\Xb_{1}}(\tbn_{1},u_{1}^{(n)}))\}}{w(u_{2}^{(n)})}
%  - \tfrac{\ind\{U_{j1} \leq F_{j\eps}(y_{j\Xb_{1}}(\tbn_{1},u_{1}^{(n)}))\}}{w(u_{1}^{(n)})}
%  \Big| \,\ind\big\{M_{j}(\Xb_{1}) \leq a_n\big\}  
% \\
% % 
% \notag 
% & \leq 
% \Big[
% \tfrac{\ind\{G_{j\Xb_{1}}^{(L)} \leq U_{j1} \leq G_{j\Xb_{1}}^{(U)}\}}{w(u_{L}^{(n)})}
%  + \ind\{U_{j1} \leq G_{j\Xb_{1}}^{(U)}\}\Big(\tfrac{1}{w(u_{L}^{(n)})}- 
% \tfrac{1}{w(u_{U}^{(n)})}\Big)
%  \Big] 
% \ind\big\{M_{j}(\Xb_{1}) \leq a_n\big\}.  
% % \\
% \end{align}
% %
% 
\begin{align}
\notag 
&\sup_{(\tb, u) \in T_{\epsilon} \times [u_{L},u_{U}]} 
 \sqrt{n}\,\big|Z_{ni}(\tb,u) -  Z_{ni}(\tb_0,u_{U})\big|  
\\*
\notag 
& = \sup_{(\tb, u) \in T_{\epsilon} \times [u_{L},u_{U}]} 
 \Big| 
% \, \ind\big\{u \geq \delta_n/2\big\}\,
\tfrac{\ind\{U_{ji} \leq F_{j\eps}(y_{j\Xb_{i}}(\tbn,u^{(n)}))\}}{w(u^{(n)})}
- \tfrac{\ind\{U_{ji} \leq F_{j\eps}(y_{j\Xb_{i}}(\tbn_{0},u_{U}^{(n)}))\}}{w(u_{U}^{(n)})}
 \Big|\,  
\ind\big\{M_{j}(\Xb_{i}) \leq a_n\big\} 
\\
\notag 
& \leq \sup_{(\tb, u) \in T_{\epsilon} \times [u_{L},u_{U}]} 
 \Big| 
\tfrac{\ind\{U_{ji} \leq F_{j\eps}(y_{j\Xb_{i}}(\tbn,u^{(n)}))\} 
- \ind\{U_{ji} \leq F_{j\eps}(y_{j\Xb_{i}}(\tbn_{0},u_{U}^{(n)}))\}}{w(u_{U}^{(n)})}
\Big|\,\ind\big\{M_{j}(\Xb_{i}) \leq a_n\big\}  
\\
\notag 
& 
\quad +  \sup_{(\tb, u) \in T_{\epsilon} \times [u_{L},u_{U}]} 
 \Big| 
% \tfrac{\ind\{U_{ji} \leq F_{j\eps}(y_{j\Xb_{i}}(\tbn_{0},u_{U}^{(n)}))\}}{w(u^{(n)})}
%  - \tfrac{\ind\{U_{ji} \leq F_{j\eps}(y_{j\Xb_{i}}(\tbn_{0},u_{U}^{(n)}))\}}{w(u_{U}^{(n)})}
\ind\{U_{ji} \leq F_{j\eps}(y_{j\Xb_{i}}(\tbn,u^{(n)}))\}
\Big(\tfrac{1}{w(u^{(n)})}
 - \tfrac{1}{w(u_{U}^{(n)})}\Big)
 \Big| \,\ind\big\{M_{j}(\Xb_{i}) \leq a_n\big\}  
\\
\label{eq: bound on diff of Zn1}
& \leq 
\Big[
\tfrac{\ind\{G_{j\Xb_{i}}^{(L)} \leq U_{ji} \leq G_{j\Xb_{i}}^{(U)}\}}{w(u_{U}^{(n)})}
 + \ind\{U_{ji} \leq u_{L}\} \Big(\tfrac{1}{w(u_{L}^{(n)})}- 
\tfrac{1}{w(u_{U}^{(n)})}\Big) \Big]
\ind\big\{M_{j}(\Xb_{i}) \leq a_n\big\}.  
\\
\notag 
& + \sup_{(\tb, u) \in T_{\epsilon} \times [u_{L},u_{U}]} 
 \ind\{u_{L} \leq U_{ji} \leq F_{j\eps}(y_{j\Xb_{i}}(\tbn,u^{(n)}))\}
\Big(\tfrac{1}{w(u^{(n)})}
 - \tfrac{1}{w(u_{U}^{(n)})}\Big)
%  \Big]
\ind\big\{M_{j}(\Xb_{i}) \leq a_n\big\}   
% \\
\\
\notag  
& = V_{n1} + V_{n2}, 
\end{align}
where $V_{n1}$, $V_{n2}$ stand for the first and second term on the right-hand side of 
\eqref{eq: bound on diff of Zn1} respectively.

Now similarly as 
% \eqref{eq: diff of Fjeps at different tb final} 
in~\eqref{eq: bound on the first term in second moments of Zni} 
one can bound the second moment of~$V_{n1}$ as  
% the sum of the first two terms on the right-hand side of 
% \eqref{eq: bound on diff of Zn1} as 
%
\begin{align}
\notag 
\E\, V_{n1}^{2} & \leq \E\,\Big[ 
\tfrac{2(G_{j\Xb_{i}}^{(U)} - G_{j\Xb_{i}}^{(L)})}{w^{2}(u_{U}^{(n)})}
 \ind\big\{M_{j}(\Xb_{i}) \leq a_n\big\} \Big] 
 + 2\,u_{L} \Big(\tfrac{1}{w(u_{L}^{(n)})}- 
\tfrac{1}{w(u_{U}^{(n)})}\Big)^{2}
\\
\notag 
& \leq \tfrac{O(1)}{w^{2}(u_{U}^{(n)}) } \sup_{(\tb, u) \in T_{\epsilon} \times [u_{L},u_{U}]} \big[\|\tb - \tb_{0}\|\, 
w(u_{U}^{(n)})\, O(n^{-1/2+\eta}) +|u_{U}-u_{L}|\big]
\\
& \notag \qquad +
2\,u_{L}\Big(\tfrac{1}{w(u_{L}^{(n)})} - \tfrac{1}{w(u_{U}^{(n)})}\Big)^2
\\
\notag 
% \label{eq: weak second moment of Vn1} 
& = % O\big(\tfrac{\epsilon^{2}}{K}\big) 
O\big(1\big)\big[\tfrac{\epsilon^{2}}{K}  + \tfrac{u_{U}-u_{L}}{w^2(u_{U})}\big] 
 + 2\,u_{L}\Big(\tfrac{1}{w(u_{L})} - \tfrac{1}{w(u_{U})}\Big)^2 
 = O\big(\tfrac{\epsilon^{2}}{K}\big)  + O\big(\tfrac{\rho_{0}^{2}(u_{L}, u_{U})}{K}\big) 
 \leq  \tfrac{\epsilon^{2}}{64},  
\end{align}
provided that $K$ in the definition of the semimetric~\eqref{eq: rho} is taken sufficiently large. 

Thus also by Markov's inequality 
\begin{equation} \label{eq: weak second moment of Vn1} 
 \sup_{v >0} v^{2}\, \Pr(V_{n1} > \tfrac{v}{4}) \leq \tfrac{\epsilon^{2}}{4}.   
%  O(\epsilon^{2}). 
\end{equation}

Now we can concentrate on the second term in~\eqref{eq: bound on diff of Zn1}. 
To do so note that from the definition of the semimetric~$\rho_{0}$ 
in~\eqref{eq: rho0} it follows that for each $u \in [u_{L},u_{U}]$
\[
\big(\tfrac{1}{w(u)}-\tfrac{1}{w(u_{U})}\big)^2  
 \leq \tfrac{\rho_{0}^{2}(u, u_{U})}{u} \leq \tfrac{4\epsilon^{2}}{K^{2}\,u}, 
\]
which further implies that 
% for each $u \in [u_{L},u_{U}]$
% 
\[
 \Big(\tfrac{1}{w(u^{(n)})}
 - \tfrac{1}{w(u_{U}^{(n)})}\Big) \leq \tfrac{2\epsilon}{K\sqrt{u}}\,. 
\]
Using the above inequality 
one can bound (with probability going to one)
\begin{align}
\notag 
% & \sup_{(\tb, u) \in T_{\epsilon} \times [u_{L},u_{U}]} 
%  \ind\{u_{L} \leq U_{ji} \leq F_{j\eps}(y_{j\Xb_{i}}(\tbn,u^{(n)}))\}
% \Big(\tfrac{1}{w(u^{(n)})}
%  - \tfrac{1}{w(u_{U}^{(n)})}\Big)
% \ind\big\{M_{j}(\Xb_{i}) \leq a_n\big\}  
V_{n2} & \leq 
\tfrac{2\epsilon\,\sup_{(\tb, u) \in T_{\epsilon} \times [u_{L},u_{U}]} 
 \ind\{u_{L} \leq U_{ji} \leq F_{j\eps}(y_{j\Xb_{i}}(\tbn,u^{(n)}))\}
}{K \sqrt{u^{(n)}}}
\ind\big\{M_{j}(\Xb_{i}) \leq a_n\big\}  
\\
\notag  
& \leq 
\tfrac{2\epsilon\,\sup_{(\tb, u) \in T_{\epsilon} \times [u_{L},u_{U}]} 
 \ind\{u_{L} \leq U_{ji} \leq F_{j\eps}(y_{j\Xb_{i}}(\tbn,u^{(n)}))\}
}{K \sqrt{F_{j\eps}(y_{j\Xb_{i}}(\tbn,u^{(n)}))}}\, 
\sqrt{\tfrac{F_{j\eps}(y_{j\Xb_{i}}(\tbn,u^{(n)}))}{u^{(n)}}}
\ind\big\{M_{j}(\Xb_{i}) \leq a_n\big\}  
\\
\notag 
& \leq 
\tfrac{2\epsilon}{K\sqrt{U_{ji}}}\, 
\sqrt{2}, 
% \ind\big\{M_{j}(\Xb_{i}) \leq a_n\big\}  
\end{align}
where we have used that thanks to \eqref{eq: bound on yjx} 
\[
 \sqrt{\tfrac{F_{j\eps}(y_{j\Xb_{i}}(\tbn,u^{(n)}))}{u^{(n)}}} 
 \leq \sqrt{\tfrac{2u^{(n)}}{u^{(n)}}} \leq \sqrt{2}
\]
and for each $\tb, u$  
\[
 \tfrac{\ind\{u_{L} \leq U_{ji} \leq F_{j\eps}(y_{j\Xb_{i}}(\tbn,u^{(n)}))\}
}{\sqrt{F_{j\eps}(y_{j\Xb_{i}}(\tbn,u^{(n)}))}} 
\leq 
\tfrac{\ind\{u_{L} \leq U_{ji} \leq F_{j\eps}(y_{j\Xb_{i}}(\tbn,u^{(n)}))\}
}{\sqrt{U_{ji}}} \leq \tfrac{1}{\sqrt{U_{ji}}}\,. 
\]

Thus we can bound 
\begin{equation} \label{eq: second weak moment of Vn2}
 \sup_{v >0} v^{2}\, \Pr\big(V_{n2} > \tfrac{v}{4}\big) 
 \leq \sup_{v >0} v^{2}\, \Pr\Big(\tfrac{2\epsilon\,\sqrt{2}}{K\sqrt{U_{ji}}}  > \tfrac{v}{4} \Big) 
 = \sup_{v >0} v^{2}\, \Pr\Big( U_{ji} < \tfrac{128\,\epsilon^{2}}{K^{2}\,v^2} \Big) 
 \leq \tfrac{\epsilon^{2}}{4}  
\end{equation}
for a sufficiently large~$K$. 
Now combining \eqref{eq: weak second moment of Vn1} and 
\eqref{eq: second weak moment of Vn2} yields that 
\[
2 \sup_{v > 0} v^{2}\, \Pr\bigg( \sup_{(\tb,u) \in T_{\epsilon} \times [u_{L},u_{U}]} 
 \sqrt{n}\,\big|Z_{ni}(\tb,u) -  Z_{ni}(\tb_{0},u_{U})\big|  
 >  v/2 \bigg) \leq \epsilon^2, 
\] 
which together with \eqref{eq: bound on weak second moment} implies 
that also \eqref{eq: weak second moment} is satisfied.  

\end{proof}

Note that while $\lambda_{x}$ is only a cleverly chosen constant 
in Lemma~\ref{lemma about hatFjhateps} that is not 
involved in the statement, in the following lemmas 
we will speak about $\mathrm{J}_{jn}^{X}$ and thus we need 
to be more specific about $\lambda_{x}$. Thus in what follows 
we often assume that 
\begin{equation} \label{eq: lambdax lower bound}
%  \tfrac{1}{2} + \tfrac{\beta}{\lambda} 
%  < 
\tfrac{1}{\lambda_{x} r} 
 < \tfrac{1}{2} - \tfrac{1-\beta}{\lambda}.  
\end{equation}

\begin{lemma} \label{lemma Ujhat minus Ujtilde}
Suppose that the assumptions of Lemma~\ref{lemma about hatFjhateps} are satisfied and $\lambda_{x}$ satisfies~\eqref{eq: lambdax lower bound}.  
Then it holds uniformly in $k  \in \mathrm{J}_{jn}^{X}$
 \begin{align*} 
%   \label{eq: generalized chibisov reilly for resid} 
\notag 
  \widehat{U}_{jk} - \widetilde{U}_{jk}
 &=  f_{j\eps}(\eps_{jk})\,\E_{\Xb}\Big[\tfrac{m'_{j}(\Xb;\thetab_{j})}{s_{j}(\Xb;\thetab_{j})}
+ \eps_{jk}\,\tfrac{s'_{j}(\Xb;\thetab_{j})}{s_{j}(\Xb;\thetab_{j})}\Big]\tr
 (\widehat{\thetab}_{j}-\thetab_{j}) + f_{j\eps}(\eps_{jk}) \big(\widehat{\eps}_{jk} - \eps_{jk} \big) 
\\
 & \quad  + U_{jk}^{(\beta - \gamma)_{+}}(1-U_{jk})^{(\beta - \gamma)_{+}}
   [M_{j}(\Xb_{k})+1]\,o_{P}\big(\tfrac{1}{\sqrt{n}}\big) 
%  + U_{jk}^{(\beta - \gamma)_{+}}(1-U_{jk})^{(\beta - \gamma)_{+}}
%   o_{P}\big(\tfrac{1}{\sqrt{n}}\big)  
% \qquad \text{for each } \beta \in \big(0,\tfrac{1}{2}\big)  
\end{align*}
for each $\gamma > 0$ and $j \in \{1,\dotsc,d\}$. 
% for each $\gamma \in (0, \beta)$. 
%  
\end{lemma}
\begin{proof}
The lemma will be shown by substitution of 
$u = F_{j\eps}(\hateps_{jk})$  into the approximation 
% of $\hatF_{j\hateps}\big(F_{j\eps}^{-1}(u)\big)$ 
\eqref{eq: approximation of hatFjhateps} stated 
in Lemma~\ref{lemma about hatFjhateps}. Note that all the following 
statements holds uniformly in $k  \in \mathrm{J}_{jn}^{X}$.  

The proof will be divided into \textbf{four steps}. First we 
show that with probability going to one 
% uniformly 
% 
\begin{equation} \label{eq: Feps at hateps}
 F_{j\eps}(\hateps_{jk}) \in [\delta_n/2, 1 - \delta_n/2] 
\end{equation}
 to justify the substitution into~\eqref{eq: approximation of hatFjhateps}. 
Second 
\begin{equation} \label{eq: Feps at hateps to beta}
F_{j\eps}(\hateps_{jk})^{(\beta - \gamma)_{+}}\big(1-F_{j\eps}(\hateps_{jk})\big)^{(\beta - \gamma)_{+}}\,o_{P}(1) 
= U_{jk}^{(\beta - \gamma)_{+}}(1-U_{jk})^{(\beta - \gamma)_{+}}o_{P}(1).  
\end{equation}
Next we show that 
\begin{align} \label{eq: feps at hateps}
 f_{j\eps}(\hateps_{jk}) &= f_{j\eps}(\eps_{jk}) +  U_{jk}^{(\beta - \gamma)_{+}}(1-U_{jk})^{(\beta - \gamma)_{+}} o_{P}(1), 
\\
\label{eq: feps at hateps times Fjinv}
f_{j\eps}(\hateps_{jk})\hateps_{jk} &= f_{j\eps}(\eps_{jk})F_{j\eps}^{-1}(U_{jk}) +  U_{jk}^{(\beta - \gamma)_{+}}(1-U_{jk})^{(\beta - \gamma)_{+}} o_{P}(1), 
\end{align}
and finally we derive 
\begin{align} 
\notag 
  \hatF_{j\eps}(\hateps_{jk}) &= \hatF_{j\eps}(\eps_{jk}) 
 + f_{j\eps}(\eps_{jk})(\hateps_{jk} - \eps_{jk}) 
\\
\label{eq: hatFeps at hateps}
& \qquad   
+ U_{jk}^{(\beta - \gamma)_{+}}(1-U_{jk})^{(\beta - \gamma)_{+}}( M_{j}(\Xb_{k})+1) o_{P}\big(\tfrac{1}{\sqrt{n}}\big)
\end{align}
and realise that $\widehat{U}_{jk} = \hatF_{j\hateps}(\hateps_{jk})$ and $\widetilde{U}_{jk} = \hatF_{j\eps}(\eps_{jk})$. 

\medskip 

\noindent \textit{Showing \eqref{eq: Feps at hateps}.} 

% Note that using~\eqref{eq: yjx} together
% \eqref{eq: diff of Fjeps at different tb} implies that with 
Analogously as in \eqref{eq: diff of Fjeps at different tb final}   
for $k \in \mathrm{J}_{jn}^{X}$
% (with probability going to one) for all $\eta > 0$ 
% 
\begin{align*} 
\notag 
 \big|F_{j\eps}(\hateps_{jk}) - U_{jk}\big| 
 & = \big|F_{j\eps}(y_{j\Xb_{k}}\big(\hatthetab_{j}, F_{j\eps}^{-1}(U_{jk}))\big)
 - F_{j\eps}(y_{j\Xb_{k}}\big(\thetab_{j}, F_{j\eps}^{-1}(U_{jk}))\big) \big|
\\
\notag 
& \leq O_{P}(1)\,M_{j}(\Xb_{k}) \|\hatthetab_j -  \thetab_j\|\, U_{jk}^{\beta} 
(1-U_{jk})^{\beta}
\\
% 
% \label{eq: Fjeps at hateps minus U} 
 &\leq O_{P}(n^{1/(\lambda_{x} r) - 1/2})\, U_{jk}^{\beta} (1-U_{jk})^{\beta}, 
%  = O(n^{1/(\lambda_{x} r) -1/2 + \gamma}).  
\end{align*}
% 
% where the function $w(u)$ was introduced in~\eqref{eq: wu}. 
This further implies that 
% Now by \eqref{eq: deltan and an} for a sufficiently small $\gamma > 0$ one gets 
% that $a -1/2 + \gamma < 1/\lambda$. Thus with the help of 
% \eqref{eq: Fjeps at hateps minus U} one conclude that 
% 
\begin{equation} \label{eq: Fjeps at hateps minus Ujk}
  \tfrac{|F_{j\eps}(\hateps_{jk}) - U_{jk}|}{U_{jk}(1-U_{jk})} 
  \leq O_{P}(n^{1/(\lambda_{x} r) -1/2})\, \delta_{n}^{\beta - 1} 
  = O_{P}(n^{1/(\lambda_{x} r) -1/2  + (1-\beta)/\lambda}) = o_{P}(1), 
\end{equation}
 where we have used that $\lambda_{x}$ satisfies~\eqref{eq: lambdax lower bound}.  
% implies that for a sufficiently small $\eta$ one has  
% $1/(\lambda_{x} r) -1/2 + \eta + (1-\beta)/\lambda < 0$. 
Thus for a sufficiently large~$n$  one gets that 
\[
 F_{j\eps}(\hateps_{jk}) \geq \tfrac{U_{jk}}{2} \geq \tfrac{\delta_{n}}{2} 
\]
and analogously also 
\[ 
 F_{j\eps}(\hateps_{jk}) \leq U_{jk} + \tfrac{1}{2}(1-U_{jk}) \leq 1 - \delta_{n}  
 + \tfrac{\delta_{n}}{2} = 1 - \tfrac{\delta_{n}}{2}.   
\]

\medskip 

\noindent \textit{Showing \eqref{eq: Feps at hateps to beta}.} 
 
Note that with the help of \eqref{eq: Fjeps at hateps minus Ujk} one can conclude that 
\[
 F_{j\eps}(\hateps_{jk}) \in \big(\tfrac{1}{2}\,U_{jk}, \tfrac{3}{2}\,U_{jk}\big), 
 \quad \text{and} \quad  
 1- F_{j\eps}(\hateps_{jk}) \in \big(\tfrac{1}{2}(1-U_{jk}), \tfrac{3}{2}(1-U_{jk})\big), 
\]
 which implies~\eqref{eq: Feps at hateps to beta}.

\medskip 

\noindent \textit{Showing \eqref{eq: feps at hateps} and~\eqref{eq: feps at hateps times Fjinv}.}
This follows from \eqref{eq: Feps at hateps}, \eqref{eq: fjeps et yjx divided by u to beta} 
and \eqref{eq: fjeps et yjx times Fjinv divided by u to beta}.

\medskip 

\noindent \textit{Showing \eqref{eq: hatFeps at hateps}.} 
\nopagebreak 

% Note that with the help of \eqref{eq: Fjeps at hateps minus U} 
% for all $\eta > 0$ 
% % such that (with probability going to one 
% for $k \in \mathrm{J}_{jn}^{X}$  
Without loss of generality consider only those $k \in \mathrm{J}_{jn}^{X}$ 
for which  $U_{jk} \leq \tfrac{1}{2}$.  
Now for $\eta \in (0, \tfrac{1}{2}-\tfrac{\beta}{\lambda}-\tfrac{1}{\lambda_{x}r}$) 
introduce 
\[
%  \big|F_{j\eps}(\hateps_{jk}) - F_{j\eps}(\eps_{jk}) \big|
%  \big|F_{j\eps}(\hateps_{jk}) - U_{jk} \big|  
% = o_{P}\big(\tfrac{1}{r_{n}}\big)\,w(U_{jk}),
%  \quad \text{where} \quad 
 r_{n} =  n^{1/2 - 1/(\lambda_{x} r) - \eta}\,. 
\]
%
% In what follows we take $\eta < \tfrac{1-2\beta}{\lambda}$. 
Similarly as in the proof of Lemma~\ref{lemma about hatFjhateps} 
define for $i \in \{1,\dotsc,n\}$ the processes 
\[
  Z_{ni}(\tb,u) = \tfrac{1}{w(u^{(n)})\sqrt{n}}
 \,\ind\Big\{U_{ji} \leq  F_{j\eps}\Big(F_{j\eps}^{-1}(u^{(n)})\big(1+\tfrac{t_{1}}{r_{n}}\big)+\tfrac{t_{2}}{r_{n}}\Big)\Big\}
\]
that are indexed by the set $\mathcal{F} = [-1,1]^2 \times (0,\tfrac{1}{2}]$. Now one 
can write $\hatF_{j\eps}(\hateps_{jk})$ as 
\[
 \hatF_{j\eps}(\hateps_{jk}) = \frac{w(\hat{u}^{(n)})}{\sqrt{n}}\suman Z_{ni}(\hat{\tb}_{n},\hat{u}),  
% + w(\hat{u}_{n}^{(n)})\,o_{P}\big(\tfrac{1}{\sqrt{n}}\big),  
\]
where 
\[
%  \hat{u} = \hatF_{j\eps}(\eps_{jk}) = \Utilde_{jk},
% \quad \text{and} \quad 
 \hat{\tb}_{n} =  
 \Big(
r_{n}\Big[\tfrac{s_{j}(\Xb_{k};\thetab_j)}{s_{j}(\Xb_{k};\hatthetab_j)}-1\Big], 
\tfrac{r_{n}[m_{j}(\Xb_{k};\thetab_j) - m_{j}(\Xb_{k};\hatthetab_j)]}{s_{j}(\Xb_{k};\hatthetab_j)} 
\Big), \quad
\hat{u} = F_{j\eps}(\eps_{jk}).    
\]
Note that for $k \in \mathrm{J}_{jn}^{X}$ %(with probability going to one)
\begin{equation} \label{eq: hatU minus tilde as empir mean}
 \hatF_{j\eps}(\hateps_{jk}) - \hatF_{j\eps}(\eps_{jk}) 
 = \frac{w(\hat{u}_{n})}{\sqrt{n}}\suman \big[Z_{ni}(\hat{\tb}_{n},\hat{u}_{n}) - Z_{ni}(\boldsymbol{0},\hat{u}_{n})
 \big]. 
\end{equation}

Now equip the space $\mathcal{F}$ with the semimetric $\rho$ given by 
\begin{equation*} 
 \rho\big((\tb_1,u_1), (\tb_2,u_2)\big) = K\,\sqrt{\|\tb_1 - \tb_2\| 
 + \tfrac{u_{2} - u_{1}}{w^2(u_2)} + \big(\tfrac{1}{w(u_1)}-\tfrac{1}{w(u_2)}\big)^2 
  u_1}, \quad \text{for} \quad u_1 \leq u_2,  
%  + \frac{|u_1 - u_{2}|}{w^2(u_1 \vee u_2)} + \Big(\frac{1}{w(u_1)}-\frac{1}{w(u_2)}\Big)^2 
%   u_1 \wedge u_2, 
\end{equation*}
where $K$ is a sufficiently large but finite constant. Then completely analogously 
as in the proof of Lemma~\ref{lemma about hatFjhateps} one can verify 
the assumptions of Theorem 2.11.11 of \cite{vaart_wellner}. Thus 
$\sup_{u \in (0,\frac{1}{2}]} 
 \rho\big((\hat{\tb}_{n},u), (\boldsymbol{0},u)\big) = o_{P}(1)$, 
implies that 
% one gets that uniformly in $u \in (0,1/2]$ 
% $u \in [\delta_{n}/2,1-\delta_{n}/2]$ 
% 
\begin{align*} \label{eq: using asympt tightness}
  \suman Z_{ni}(\hat{\tb}_{n},\hat{u}_{n})  - Z_{ni}(\boldsymbol{0},\hat{u}_{n})  
   &= \suman \E_{U,\Xb} \big[Z_{ni}(\hat{\tb}_{n},\hat{u}_{n}) - Z_{ni}(\boldsymbol{0},\hat{u}_{n} )\big] + o_{P}(1)
\\ 
  &= \frac{\sqrt{n}}{w(U_{jk})} \big[ F_{j\eps}(\hateps_{jk}) - F_{j\eps}(\eps_{jk})\big] + o_{P}(1),
\end{align*}
which together with \eqref{eq: hatU minus tilde as empir mean} implies that 
\begin{equation} \label{eq: diff of hatFeps}
 \hatF_{j\eps}(\hateps_{jk}) - \hatF_{j\eps}(\eps_{jk}) 
  = F_{j\eps}(\hateps_{jk}) - F_{j\eps}(\eps_{jk}) + w(U_{jk})\,o_{P}\big(\tfrac{1}{\sqrt{n}}\big).   
\end{equation}
Now the right-hand side of the above equations can be 
with the help of \eqref{eq: fjeps et yjx divided by u to beta} 
and \eqref{eq: fjeps et yjx times Fjinv divided by u to beta} rewritten as 
\[
 F_{j\eps}(\hateps_{jk}) - F_{j\eps}(\eps_{jk})
 = f_{j\eps}(\eps_{jk})\big(\hateps_{jk}- \eps_{jk}\big) +  
%    U_{jk}^{(\beta - \gamma)_{+}}(1-U_{jk})^{(\beta - \gamma)_{+}}
    w(U_{jk}) \,M_{j}(\Xb_{k}) o_{P}\big(\tfrac{1}{\sqrt{n}}\big), 
\]
which combined with \eqref{eq: diff of hatFeps} implies \eqref{eq: hatFeps at hateps}.
\end{proof}

\begin{lemma} \label{lemma linear bound for hatU}
Suppose that the assumptions of Lemma~\ref{lemma Ujhat minus Ujtilde} are satisfied 
and $\beta > 0$. Then for each $\epsilon > 0$ there exists $L_{\epsilon} > 0$ such that 
for each $j \in \{1,\dotsc,d\}$ for all sufficiently large~$n$   
\begin{equation*} %\label{eq: generalized shorack bound on empirical cdf of resid}
 \Pr\big( \forall_{k \in \mathrm{J}_{jn}^{X}} \;  \; L_{\epsilon}\,U_{jk} \leq \widehat{U}_{jk} \leq
     1 - L_{\epsilon}\,(1 - U_{jk}) \big)
  \geq 1 - \epsilon.
\end{equation*}
\end{lemma}
\begin{proof}
We concentrate on the inequality $L_{\epsilon}\,U_{jk} \leq \widehat{U}_{jk}$. 
Showing the upper inequality for $\widehat{U}_{jk}$ would be analogous. 

By Lemma~\ref{lemma Ujhat minus Ujtilde} one gets 
$\widehat{U}_{jk} \geq \widetilde{U}_{jk} - |R_{jk}|$, where 
\begin{align}
\notag 
 R_{jk}&=  f_{j\eps}(\eps_{jk})\,\E_{\Xb}\Big[\tfrac{m'_{j}(\Xb;\thetab_{j})}{s_{j}(\Xb;\thetab_{j})}
+ \eps_{jk}\,\tfrac{s'_{j}(\Xb;\thetab_{j})}{s_{j}(\Xb;\thetab_{j})}\Big]\tr
 (\widehat{\thetab}_{j}-\thetab_{j}) + f_{j\eps}(\eps_{jk}) 
\big(\widehat{\eps}_{jk} - \eps_{jk} \big) 
\\
\label{eq: Rij}
 & \quad 
 + U_{jk}^{\beta-\gamma}(1-U_{jk})^{\beta-\gamma}
   [M_{j}(\Xb_{k})+1] o_{P}\big(\tfrac{1}{\sqrt{n}}\big).  
\end{align}
and $\gamma > 0$ can be taken arbitrarily small. 

Now by Lemma~A3 of \citet{shorack1972functions}  for each $\epsilon > 0$ there 
exists $\widetilde{L} \in (0,1)$ such that 
\[
 \Pr\big(\forall_{k \in \{1,\dotsc,n\}}: 
 \widetilde{U}_{jk} \geq \widetilde{L}\, U_{jk} \big) \geq 1 - \epsilon/2.  
\]
Thus one can take $L_{\epsilon} = \widetilde{L}/2$  provided we show that 
\[
 \Pr\big(\forall_{k \in \mathrm{J}_{jn}^{X}} : 
 |R_{jk}| \leq \tfrac{\widetilde{L}\,U_{jk}}{2} \big) \geq 1 - \epsilon/2.  
\]
% 
% To do that note that for $j \in \mathrm{J}_{jn}$ one has 
% $U_{jk} \geq n^{-1/\lambda}$. On the other hand for instance 
% % 
% \[
%  f_{j\eps}(\eps_{jk}) 
% \big(\widehat{\eps}_{jk} - \eps_{jk} \big) 
% \]
To do that one can consider each of the summands on the right-hand side 
of~\eqref{eq: Rij} separately. Thus for instance one has that uniformly 
in $k \in \mathrm{J}_{jn}^{X}$ 
\begin{align*}
 \Big|\tfrac{f_{j\eps}(\eps_{jk})(\widehat{\eps}_{jk} - \eps_{jk}) }{U_{jk}}\Big|  
 &\leq U_{jk}^{\beta-1} M_{j}(\Xb_{k}) O_{P}\big(\tfrac{1}{\sqrt{n}}\big) 
  \leq n^{(1-\beta)/\lambda}\, a_{n}\,O_{P}\big(\tfrac{1}{\sqrt{n}}\big)  
\\
   &= O_{P}\big(n^{(1-\beta)/\lambda + 1/(\lambda_{x} r) - 1/2} \big)
   = o_{P}(1),
\end{align*}
as $\lambda_{x}$ satisfies~\eqref{eq: lambdax lower bound}. 
% as $\lambda > 2( 1 - \beta + 1/r)$ which implies that 
% $(1-\beta)/\lambda + 1/(\lambda_{x} r) - 1/2 < 0$. 
The other summands on the right-hand side 
of~\eqref{eq: Rij} can be handled analogously. 
\end{proof}

\subsection*{Some results useful when 
\texorpdfstring{$(\mathbf{F}_{j\eps})$ holds 
with $\beta = 0$}{Fjeps holds with beta = 0}}

\begin{lemma} \label{lemma about hatFjhateps second}
Suppose that assumptions $(\mathbf{F}_{j\eps})$ and $\boldsymbol{(ms)}$ 
hold. Then for each $j \in \{1,\dotsc,d\}$ 
\begin{equation} \label{eq: Fjhateps at Fjeps at sqrt rate}
 \sup_{u \in (0,1)} \sqrt{n}\,\big|\hatF_{j\hateps}\big(F_{j\eps}^{-1}(u)\big) 
 - \hatF_{j\eps}\big(F_{j\eps}^{-1}(u)\big)  \big| = O_{P}(1). 
\end{equation}
% 
% Further for each $\delta \in (0,\tfrac{1}{2})$ 
% it holds uniformly in $u \in [\delta,1-\delta]$ 
% % 
% \begin{align}
%  \notag 
%  \hatF_{j\hateps}\big(F_{j\eps}^{-1}(u)\big) 
%  &= \hatF_{j\eps}\big(F_{j\eps}^{-1}(u)\big)
% % \\  
% %  + f_{j\eps}(F_{j\eps}^{-1}(u))\Big\{\E_{\Xb}\big[\tfrac{m_{j}(\Xb, \hatthetab_{j}) - m_{j}(\Xb, \hatthetab_{j})}{s_{j}(\Xb; \thetab_{j})}\big] + F_{j\eps}^{-1}(u)\,\E_{\Xb}\big[\tfrac{s_{j}(\Xb, \hatthetab_{j}) - s_{j}(\Xb; \thetab_{j})}{s_{j}(\Xb; \thetab_{j})}\big] \Big\}
% %   + u^{\beta}(1-u)^{\beta} o_{P}\big(\tfrac{1}{\sqrt{n}}\big)
% % \\
% %  = 
% % \hatF_{j\eps}\big(F_{j\eps}^{-1}(u)\big) 
%   + f_{j\eps}(F_{j\eps}^{-1}(u))\,\E_{\Xb}\Big[\tfrac{m'_{j}(\Xb;\thetab_{j})}{s_{j}(\Xb;\thetab_{j})}
% + F_{j\eps}^{-1}(u)\,\tfrac{s'_{j}(\Xb;\thetab_{j})}{s_{j}(\Xb;\thetab_{j})}\Big]
%  (\widehat{\thetab}_{j}-\thetab_{j})
% \\
% % 
% \label{eq: approximation of hatFjhateps second}
%   &  \quad + o_{P}\big(\tfrac{1}{\sqrt{n}}\big).   
% \end{align}
% % 
\end{lemma}
\begin{proof}
 Let $U(\thetab_j)$ be the neighborhood of $\thetab_{j}$ introduced 
in~$\boldsymbol{(ms)}$. Now consider the set of functions 
\begin{align}
 \notag  
 \mathcal{F} &= \Big\{(\xb,e) \mapsto 
 \ind\big\{e \leq \tfrac{m_{j}(\xb;\tb) - m_{j}(\xb;\thetab_j)}{s_{j}(\xb;\thetab_j)} + F_{j\eps}^{-1}(u) \tfrac{s_{j}(\xb;\tb)}{s_{j}(\xb;\thetab_j)}
% , |M_{j}(\xb)| \leq a_n 
\big\};\,  
% \\
% % 
% \notag 
% & \qquad - \ind\{v - u \leq 0, |M_{j}(\xb)| \leq a_n\big\} \big]; 
 u \in (0,1), 
 \tb \in U(\thetab_j)   
 \Big\}  
\end{align}
and denote its elements as $f_{\tb, u}$. Then one can write 
\[
 \hatF_{j\hateps}\big(F_{j\eps}^{-1}(u)\big) 
 = \frac{1}{n} \suman f_{\hatthetab_{j},u}(\Xb_{i},\eps_{ji}) 
 \stackrel{Say}{=} \mathsf{P}_{n} \big(f_{\hatthetab_{j},u}\big). 
\]
Similarly as in the proof of Theorem~4 of \cite{ogv_sjs_2015} one can 
argue that the set~$\mathcal{F}$ is $P$-Donsker. Further 
similarly as in the proof of Lemma~\ref{lemma about hatFjhateps} 
one can show that  
\[
 \sup_{u \in (0,1)}\var_{U,\Xb}\big(f_{\hatthetab_{j},u} - f_{\thetab_{j},u}\big) 
  \leq \EX \big|F_{j\eps}(y_{j\Xb}(\hatthetab_j, u)) - F_{j\eps}(y_{j\Xb}(\thetab_j, u))\big|
 =  o_{P}(1), 
\]
which further implies that uniformly in $u \in (0,1)$
\begin{equation} \label{eq: hafFjhateps thanks to asympt tightness}
 \sqrt{n}\big[\hatF_{j\hateps}\big(F_{j\eps}^{-1}(u)\big) 
 - \hatF_{j\eps}\big(F_{j\eps}^{-1}(u)\big) \big]
 = \sqrt{n}\big[\mathsf{P} \big(f_{\hatthetab_{j},u}\big) 
 - \mathsf{P} \big(f_{\thetab_{j},u}\big)\big] 
 + o_{P}\big(\tfrac{1}{\sqrt{n}}\big).   
\end{equation}
Now by the mean value theorem there exists $\tb_{*}$ between 
$\hatthetab_{j}$ and $\thetab_{j}$ such that 
\begin{align*}
\notag   
 \sup_{u \in (0,1)}& \Big|\sqrt{n}\big[\mathsf{P} \big(f_{\hatthetab_{j},u}\big) 
 - \mathsf{P} \big(f_{\thetab_{j},u}\big)\big]\Big|
\\ 
 & = \sqrt{n}\,\sup_{u \in (0,1)}\EX\big[\big|F_{j\eps}(y_{j\Xb}\big(\hatthetab_{j},u)\big) 
 -  F_{j\eps}(y_{j\Xb}\big(\thetab_{j},u)\big) \big| 
\\
\notag 
& \leq \sqrt{n}\,\sup_{u \in (0,1)}\EX\,\Big[ f_{j\eps}(y_{j\Xb}(\tb_{*},u)) 
\Big(M_{j}(\Xb)  
% \Big(\tfrac{M_{j}(\Xb_{1})}{s_{j}(\Xb;\thetab_{j})} 
  + \big|F_{j\eps}^{-1}(u)
% \big|\tfrac{M_{j}(\Xb_{1})}{s_{j}(\Xb;\thetab_{j})}\Big)\Big]
 \big|M_{j}(\Xb)\Big)\Big]
\|\hatthetab_{j} - \thetab_{j}\| 
\\
\notag   
& \leq 
\sqrt{n}\,\|\hatthetab_{j} - \thetab_{j}\|\,  
\sup_{u \in (0,1)} 
\EX \Big[f_{j\eps}(y_{j\Xb}(\tb^{*},u))(1 + |y_{j\Xb}(\tb^{*},u)|) 
\tfrac{M_{j}(\Xb) 
(1 + |F_{j\eps}^{-1}(u)|)
}{1 + |y_{j\Xb}(\tb^{*},u)|}\Big] 
\\
\label{eq: diff of Fjeps at different tb} 
& \leq O_{P}(1)\, \E\, M_{j}(\Xb)\, 
 = O_{P}(1)\,O(1) = O_{P}(1),
\end{align*}
% 
% uniformly in $u$. 
% 
% % 
% \[
%  \sup_{u \in (0,1)}\Big|\sqrt{n}\big[\mathsf{P} \big(f_{\hatthetab_{j},u}\big) 
%  - \mathsf{P} \big(f_{\thetab_{j},u}\big)\big]\Big| = O_{P}(1), 
% \]
% % 
which together with \eqref{eq: hafFjhateps thanks to asympt tightness} 
implies \eqref{eq: Fjhateps at Fjeps at sqrt rate}. 
% Further by calculating 
% $\mathsf{P} \big(f_{\hatthetab_{j},u}\big)  - \mathsf{P} \big(f_{\thetab_{j},u}\big)\big]$ 
% for $u \in [\delta, 1-\delta]$ one gets \eqref{eq: approximation of hatFjhateps second}. 
% % 
\end{proof}

\begin{lemma} \label{lemma Ujhat minus Ujtilde second}
Suppose that the assumptions of Lemma~\ref{lemma about hatFjhateps second} are satisfied. 
Then for each $j \in \{1,\dotsc,d\}$ 
\begin{equation*}
\max_{k \in \{1,\dotsc,n\}} 
\Big|\frac{\widehat{U}_{jk} - \widetilde{U}_{jk}}{1+M_{j}(\Xb_{k})}\Big| = 
O_{P}\big(\tfrac{1}{\sqrt{n}}\big).  
\end{equation*}
% 
% Further for each $\delta > 0$
% it holds that uniformly in $k  \in \mathrm{J}_{j\delta}^{X}$
% % 
%  \begin{align*} 
% %   \label{eq: generalized chibisov reilly for resid} 
% \notag 
%   \widehat{U}_{jk} - \widetilde{U}_{jk}
%  &=  f_{j\eps}(\eps_{jk})\E_{\Xb}\Big[\tfrac{m'_{j}(\Xb;\thetab_{j})}{s_{j}(\Xb;\thetab_{j})}
% + \eps_{jk}\,\tfrac{s'_{j}(\Xb;\thetab_{j})}{s_{j}(\Xb;\thetab_{j})}\Big]
%  (\widehat{\thetab}_{j}-\thetab_{j}) + f_{j\eps}(\eps_{jk}) \big(\widehat{\eps}_{jk} - \eps_{jk} \big) 
% \\
%  & \quad  + [M_{j}(\Xb_{k})+1] o_{P}\big(\tfrac{1}{\sqrt{n}}\big).  
% %  + U_{jk}^{\beta-\gamma}(1-U_{jk})^{\beta-\gamma}
% %   o_{P}\big(\tfrac{1}{\sqrt{n}}\big)  
% % \qquad \text{for each } \beta \in \big(0,\tfrac{1}{2}\big)  
% \end{align*}
%
\end{lemma}
\begin{proof}
 The proof follows by substitution of 
$u = F_{j\eps}(\hateps_{jk})$  into \eqref{eq: Fjhateps at Fjeps at sqrt rate} 
and following the proof of Lemma~\ref{lemma Ujhat minus Ujtilde}. 
% and \eqref{eq: approximation of hatFjhateps second} stated 
% in Lemma~\ref{lemma about hatFjhateps second}. 
% 
\end{proof}

% % 
\renewcommand{\theequation}{C\arabic{equation}}
\setcounter{equation}{0}
% % % %

\section{Further auxiliary results}
\begin{lemma} \label{lemma about yjx}
%  Suppose that $y_{j\xb}(\tb,u)$ is given \eqref{eq: yjx} where 
% $\tb \in U_{n}(\thetab_j)$ and $U_{n}(\thetab_j)$ 
% is neighborhood of $\thetab_{j}$ with diameter $n^{-1/2+\eta}$
% for some $\eta > 0$.  
Suppose that assumption $(\mathbf{F}_{j\eps})$ holds.   
Let $\lambda$ satisfy $\lambda > 2(1 - \beta + \tfrac{1}{r-1})$   
and $\lambda_{x}$ satisfies~\eqref{eq: lambdax lower bound}. 
Further for $\eta > 0$ 
introduce $b_{n} = n^{\frac{1}{\lambda_{x} r} - \frac{1}{2} - \eta}$. 
Then there exists $\eta > 0$ such that for all sufficiently large~$n$ 
for all  $u \in [\tfrac{\delta_n}{2}, \tfrac{1}{2}]$ for each $j \in \{1,\dotsc,d\}$ 
\[
  F_{j\eps}^{-1}(\tfrac{u}{2}) \leq  b_n +  \big[1+ b_n\sign\big(F_{j\eps}^{-1}(u)\big)\big]F_{j\eps}^{-1}(u)
 \leq F_{j\eps}^{-1}(2u), 
%  \ \forall u \in [\delta_n/2, 1/2]
\]
% 

% % 
% \[
%   n^{\frac{1}{\lambda r} - \frac{1}{2}} + 
%  \big[1+ n^{\frac{1}{\lambda r} - \frac{1}{2}}\sign\big(F_{j\eps}^{-1}(u)\big)\big]F_{j\eps}^{-1}(u)
%  \leq F_{j\eps}^{-1}(2u), \ \forall u \in [\delta_n/2, 1/2]
% \]
% % 
and
\[
 F_{j\eps}^{-1}(1- \tfrac{u}{2}) \geq  -b_{n} + \big[1 - b_{n}\sign\big(F_{j\eps}^{-1}(1-u)\big)\big] F_{j\eps}^{-1}(1-u) \geq F_{j\eps}^{-1}(1-2u).  
% \ \forall u \in [\delta_n/2, 1/2,]. 
\]
%  
% \[
% %   \sup_{\tb \in U_{n}(\thetab_j)}  \sup_{\xb \in \{\tilde{\xb}: M_{j}(\tilde{\xb})\leq a_{n}\}} 
% %    y_{j\xb}(\tb ,1-u) 
%  -n^{\frac{1}{\lambda r} - \frac{1}{2}} + 
%  \big[1- n^{\frac{1}{\lambda r} - \frac{1}{2}}\sign\big(F_{j\eps}^{-1}(1-u)\big)\big] F_{j\eps}^{-1}(1-u) \geq F_{j\eps}^{-1}(1-2u), \ \forall u \in [\delta_n/2, 1/2,]. 
% \]
% % 
\end{lemma}
\begin{proof}
 We show only that 
\[
b_n +  \big[1+ b_n\sign\big(F_{j\eps}^{-1}(u)\big)\big]F_{j\eps}^{-1}(u)
 \leq F_{j\eps}^{-1}(2u), 
\]
as the remaining inequalities could be proved analogously. Thus we need to show that 
\begin{equation} \label{eq: bound on perturbed quantile}
 b_{n} + 
 b_{n}\,|F_{j\eps}^{-1}(u)|
 \leq F_{j\eps}^{-1}(2u) - F_{j\eps}^{-1}(u). 
\end{equation}
Now by the mean value theorem
\[
 F_{j\eps}^{-1}(2u) - F_{j\eps}^{-1}(u)  
 = \frac{u}{f_{j\eps}\big(F_{j\eps}^{-1}(\tilde{u})\big)},  
%  \geq D\,u^{1-\beta},
\] 
where $\tilde{u}$ is between $u$ and $2u$. Thus with the help of
\eqref{eq: bound on perturbed quantile} it remains to show that 
\[
 \frac{f_{j\eps}\big(F_{j\eps}^{-1}(\tilde{u})\big)
 (1+|F_{j\eps}^{-1}(u)|)}{u} \leq \frac{1}{b_{n}} = n^{\frac{1}{2} - \frac{1}{\lambda_{x} r} - \eta}. 
\] 
Now by assumption $(\mathbf{F}_{j\eps})$ and using the fact that $u \geq \delta_{n}$
\[
 \frac{f_{j\eps}\big(F_{j\eps}^{-1}(\tilde{u})\big)
 (1+|F_{j\eps}^{-1}(u)|)}{u} = O(u^{\beta-1}) 
 \leq O(n^{\frac{1-\beta}{\lambda}}) = o(n^{\frac{1}{2} - \frac{1}{\lambda_{x} r}-\eta}),  
\]
where we have used that $\lambda_{x}$ satisfies~\eqref{eq: lambdax lower bound}, 
which guarantees that one can find  $\eta > 0$ sufficiently small 
so that 
$\frac{1}{2} - \frac{1}{\lambda_{x} r}-\eta > \frac{1-\beta}{\lambda}$ holds. 
% 
% where we have used that $\frac{1-\beta}{\lambda} < \frac{1}{2} - \frac{1}{\lambda r}$ 
% which is satisfied thanks to the assumption that $\lambda > 2(1 - \beta + 1/r)$. 

\end{proof}

\begin{lemma} \label{lemma about cont of fjeps} 
Suppose that the assumptions of Lemma~\ref{lemma about yjx} are satisfied. 
Then there exists  $\eta > 0$ such that for all $\gamma > 0$ 
for each $j \in \{1,\dotsc,d\}$ 
\begin{equation*} 
 \sup_{s_1,s_2 \in \{-1,1\}}\sup_{u \in [\frac{\delta_n}{2},1-\frac{\delta_n}{2}]} 
 \bigg|\frac{f_{j\eps}\big(s_{1}b_{n} + (1+s_{2}b_{n})F_{j\eps}^{-1}(u)\big) 
  - f_{j\eps}\big(F_{j\eps}^{-1}(u)\big)}
 {u^{(\beta-\gamma)_{+}}(1-u)^{(\beta-\gamma)_{+}}}\bigg| = o(1) 
\end{equation*}
and also 
\begin{equation*} 
 \sup_{s_1,s_2 \in \{-1,1\}}\sup_{u \in [\frac{\delta_n}{2},1-\frac{\delta_n}{2}]} 
 \bigg|\frac{\big[f_{j\eps}\big(s_{1}b_{n} + (1+s_{2}b_{n})F_{j\eps}^{-1}(u)\big) 
  - f_{j\eps}\big(F_{j\eps}^{-1}(u)\big)\big]\,F_{j\eps}^{-1}(u)}
 {u^{(\beta-\gamma)_{+}}(1-u)^{(\beta-\gamma)_{+}}} 
 \bigg| = o(1)  
\end{equation*}
as $n \to \infty$. 
\end{lemma}
 
\begin{proof}
We will prove only that 
\begin{equation*} 
 \sup_{s_1,s_2 \in \{-1,1\}}\sup_{u \in [\frac{\delta_n}{2},\frac{1}{2}]} 
 \bigg|\frac{\big[f_{j\eps}\big(s_{1}b_{n} + (1+s_{2}b_{n})F_{j\eps}^{-1}(u)\big) 
  - f_{j\eps}\big(F_{j\eps}^{-1}(u)\big)\big]\,F_{j\eps}^{-1}(u)}
 {u^{(\beta-\gamma)_{+}}(1-u)^{(\beta-\gamma)_{+}}} 
 \bigg| = o(1)  
\end{equation*}
as the remaining cases can be shown analogously. 

\smallskip 

First suppose that $\lim_{u \to 0_{+}} f_{j\eps}\big(F_{j\eps}^{-1}(u)\big) >  0$. 
Then from Remark~\ref{remark about beta equal to zero} one can conclude that  
$\lim_{u \to 0_{+}} F_{j\eps}^{-1}(u) > -\infty$ and $\beta = 0$. 
Thus also $(\beta-\gamma)_{+} = 0$  
and the statement follows from the continuity of $f_{j\eps}$. 

\smallskip 

Now suppose that $\lim_{u \to 0_{+}} f_{j\eps}\big(F_{j\eps}^{-1}(u)\big) =  0$. 
Note that for a given $u_{U} \in (0,\frac{1}{2})$ 
\begin{equation*} 
 \sup_{s_1,s_2 \in \{-1,1\}}\sup_{u \in [\frac{u_{U}}{2},\frac{1}{2}]} 
 \bigg|\frac{\big[f_{j\eps}\big(s_{1}b_{n} + (1+s_{2}b_{n})F_{j\eps}^{-1}(u)\big) 
  - f_{j\eps}\big(F_{j\eps}^{-1}(u)\big)\big]\,F_{j\eps}^{-1}(u)}
 {u^{(\beta-\gamma)_{+}}(1-u)^{(\beta-\gamma)_{+}}} 
 \bigg| = o(1),   
\end{equation*}
which follows from the continuity of the function $f_{j\eps}$. 
% , see assumption $(\mathbf{F}_{j\eps})$. 

Now let $\epsilon > 0$ be given and $\gamma > 0$ fixed. 
Thanks to assumption $(\mathbf{F}_{j\eps})$ 
one can choose $u_{U}$ so that 
\[
 \sup_{u \in (0, 2\,u_{U}]} 
 \bigg| \frac{f_{j\eps}\big(F_{j\eps}^{-1}(u)\big)F_{j\eps}^{-1}(u)}{u^{(\beta - \gamma)_{+}}(1-u)^{(\beta - \gamma)_{+}}}\bigg| < \frac{\epsilon}{M},   
\]
where 
$
 M = \sup_{u \in (0,1/2)}\frac{f_{j\eps}(F_{j\eps}^{-1}(2u))}{f_{j\eps}(F_{j\eps}^{-1}(u))}.  
%   
% = \max\bigg\{ \sup_{u \in (0,1/2)}\frac{f_{j\eps}\big(F_{j\eps}^{-1}(2u)\big)}{f_{j\eps}\big(F_{j\eps}^{-1}(u)\big)}, 
%  \sup_{u \in (1/2,1)}\frac{f_{j\eps}\big(F_{j\eps}^{-1}(1-2u)\big)}{f_{j\eps}\big(F_{j\eps}^{-1}(1-u)\big)}
%   \bigg\}.
$
Now thanks to Lemma~\ref{lemma about yjx} one can conclude that also 
\begin{align*} 
 \sup_{s_1,s_2 \in \{-1,1\}}\sup_{u \in [\frac{\delta_{n}}{2},u_{U}]} 
 &\bigg|\frac{f_{j\eps}\big(s_{1}b_{n} + (1+s_{2}b_{n})F_{j\eps}^{-1}(u)\big) 
  \,F_{j\eps}^{-1}(u)}
 {u^{(\beta-\gamma)_{+}}(1-u)^{(\beta-\gamma)_{+}}} 
 \bigg| 
\\ 
& 
\leq 
\sup_{u \in (0,2u_{U}]} 
 \bigg| \frac{M f_{j\eps}\big(F_{j\eps}^{-1}(u)\big)F_{j\eps}^{-1}(u)}{u^{(\beta - \gamma)_{+}}(1-u)^{(\beta - \gamma)_{+}}}\bigg| < \epsilon, 
\end{align*}
which finishes the proof of the lemma. 
\end{proof}

\begin{lemma} \label{lemma density small at infinity}
 Suppose that the density $f_{j\eps}$ satisfies assumption $(\mathbf{F}_{j\eps})$. Then 
\[
 \lim_{|x| \to \infty} |x|f_{j\eps}(x) = 0. 
\]
\end{lemma}
\begin{proof}
 We will consider only $x \to \infty$. The remaining case would be handled analogously. 

First, note that one can assume that $\lim_{u \to 1_{-}} F_{j\eps}^{-1}(u) = \infty$, 
otherwise the proof is trivial. Now suppose that 
\[
 \lim_{x \to \infty} x f_{j\eps}(x) \neq 0. 
\]
Then one can find a positive constant $a$ and a sequence $\{z_{n}\}_{n=1}^{\infty}$ monotonically going 
to infinity such that 
\[
 z_{n} f(z_{n}) \geq a , \qquad \forall n \in \NN. 
\]
Note that by assumption $(\mathbf{F}_{j\eps})$ the function $f_{j\eps}(x)$  is non-increasing 
for $x > F_{j\eps}^{-1}(u_{2})$. In what follows we will assume that  $z_{1} >  F_{j\eps}^{-1}(u_{2})$ 
and that $z_{n+1} \geq 2 z_{n}$ (otherwise one can take an appropriate subsequence of $\{z_{n}\}$). 
Now one can bound 
\begin{align*}
 \int_{z_{1}}^{\infty} f_{j\eps}(x) dx 
 & = \sum_{n=1}^{\infty} \int_{z_n}^{z_{n+1}} \tfrac{x f_{j\eps}(x)}{x}\, dx 
\geq  \sum_{n=1}^{\infty} \int_{z_n}^{z_{n+1}} \tfrac{a}{z_{n+1}}\, dx    
\\
 & 
 = a \sum_{n=1}^{\infty} \tfrac{z_{n+1}-z_{n}}{z_{n+1}}
 = a \sum_{n=1}^{\infty} \big(1 - \tfrac{z_{n}}{z_{n+1}}\big)
 \geq a \sum_{n=1}^{\infty} \big(1 - \tfrac{1}{2}\big) = \infty, 
\end{align*}
which is in contradiction with the fact, that $f_{j\eps}$ is a density. 
\end{proof}

%
% % 
\renewcommand{\theequation}{D\arabic{equation}}
\setcounter{equation}{0}

\section{Some properties of \texorpdfstring{$\rho_{0}$}{rho0} function}
Recall the definition of~$\rho_{0}$ in~\eqref{eq: rho0} and 
for simplicity of notation put $b = (\beta - \gamma)_{+}$.  
Then for each $u_1,u_2$ satisfying $0 < u_{1} \leq u_{2} \leq \frac{1}{2}$ 
one has $\rho_{0}(u_1,u_2) = r_{0}(u_1,u_2)$, 
where   
\begin{equation*} %\label{eq: rho0}
r_{0}(u_1, u_2) = 
% r_{0}(u_2, u_1) = 
 \sqrt{\tfrac{u_{2} - u_{1}}{u_2^{2b}} + \big(\tfrac{1}{u_1^{b}}-\tfrac{1}{u_2^{b}}\big)^2 u_1}\,. 
% \,, \quad \text{for} \quad u_1 \leq u_2   
%  + \frac{|u_1 - u_{2}|}{w^2(u_1 \vee u_2)} + \Big(\frac{1}{w(u_1)}-\frac{1}{w(u_2)}\Big)^2 
%   u_1 \wedge u_2, 
\end{equation*}

\renewcommand{\labelenumi}{(\roman{enumi}).}

\begin{lemma} \label{lemma about rho0}
 Let $u_{0} \in (0, \frac{1}{2})$ and $b \in [0,\tfrac{1}{2})$ be fixed. Then 
the following statements hold. 
\begin{enumerate}
 \item The function $g_{R}(u) = r_{0}^{2}(u_{0},u)$ is increasing for 
 $u \in (u_{0}, \frac{1}{2})$.
 \item For $b > 0$ the function $g_{L}(u) = r_{0}^{2}(u, u_{0})$ is 
increasing on $(0,u_{*})$  and decreasing on  $(u_{*},u_{0})$, where 
$u_{*} = u_{0}\big(\frac{1-2b}{2(1-b)}\big)^{1/b}$. 
\item For each $0 \leq u_{1} < u_{2} < u_{0} \leq \tfrac{1}{2}$ 
it holds that $ r_{0}^{2}(u_2, u_{0}) \leq 2\, r_{0}^{2}(u_1, u_{0})$.  
 \item For each $\epsilon > 0$ the set 
 $U(u_{0}, \epsilon) = \big\{u \in [0,\tfrac{1}{2}]: \rho_{0}(u,u_{0}) 
\leq \epsilon \big\} $ is contained in a set $[u_{L},u_{U}]$ such that 
 $r_{0}(u_{L},u_{U}) \leq 2 \epsilon$. 
\end{enumerate}
\begin{proof}
 The proof of \textit{(i)} follows directly from the definition of the function $g$, as 
\[
 g_{R}(u) = r_{0}^{2}(u_{0},u) 
= \tfrac{u - u_{0}}{u^{2b}} 
 + \big(\tfrac{1}{u_0^{b}}-\tfrac{1}{u^{b}}\big)^2 u_{0} 
= u^{1-2b} + u_{0}^{1-2b} - \tfrac{2u_{0}^{1-b}}{u^{b}},  
\]
which is evidently an increasing function on $(u_{0},\frac{1}{2}]$.

\medskip 

For the proof of \textit{(ii)} rewrite 
\[
 g_{L}(u) = r_{0}^{2}(u,u_{0}) 
= \tfrac{u_{0} - u}{u_{0}^{2b}} 
 + \big(\tfrac{1}{u_{0}^{b}} - \tfrac{1}{u_{0}^{b}}\big)^2 u 
= u_{0}^{1-2b} + u^{1-2b} - \tfrac{2u^{1-b}}{u_{0}^{b}}.   
% u0^(1-2*beta) + u^(1-2*beta) - 2*u^(1-beta)/u0^(beta);
\]
Now it is straightforward to find that the function~$g_{L}$ has exactly  
one local maximum in the point $u_{*}$ and meets the claimed properties.  

\medskip 

Now we show \textit{(iii)}. Note that thanks to \textit{(ii)} 
the function~$g_{L}(u)$ is decreasing on $(u_{*},u_{0})$, 
thus the inequality trivially holds if $u_{*}\leq u_1 < u_2$. 
 
Thus suppose that $u_{1} < u_{*}$. From \textit{(ii)} 
we further know that $u_{*} = u_{0}\,a$, where $a < 1$. 
Thus we can bound 
\begin{align*}
 r_{0}^{2}\big(u_2, u_{0}\big)  
 &\leq 
 r_{0}^{2}\big(u_{*}, u_{0}\big) 
 = u_{0}^{1-2b}\big[1 - a  
+ (1-a^{b})^2 a^{1-2b}\big]
\\
 &\leq 2\,u_{0}^{1-2b} 
 = 2 r_{0}^{2}\big(0, u_{0}\big) \leq 2 r_{0}^{2}\big(u_{1}, u_{0}\big), 
\end{align*}
which was to be proved.
%  to verify \eqref{eq: ineq for rho second}. 

\medskip 

To prove \textit{(iv)} first note that from \textit{(i)} there 
exists $u_{U}$ such that 
\[
 \big\{u \in [u_{0},\tfrac{1}{2}]: \rho_{0}(u,u_{0}) \leq \epsilon \big\} 
 = [u_{0},u_{U}] 
 \quad \text{and} \quad  \rho_{0}(u,u_{U}) \leq \epsilon. 
\]
When searching for $u_{L}$ one has to be more careful as the function 
$g_{L}$ is not decreasing on $(0,u_{0})$. We need to distinguish two 
cases. First, let $\epsilon < r_{0}(0,u_{0})$. Then one can find $u_{L}$ in a similar 
way as $u_{U}$ was found. Second, suppose that $\epsilon \geq  r_{0}(0,u_{0})$. Then we take simply $u_{L}=0$. 

Now it remains to check that $r_{0}(u_{L},u_{U}) \leq 2 \epsilon$. To do that 
bound 
\begin{align*}
 r_{0}^{2}(u_{L},u_{U}) &=  \tfrac{u_{U} - u_{L}}{u_{U}^{2b}} 
 + \big(\tfrac{1}{u_{L}^{b}} - \tfrac{1}{u_{U}^{b}}\big)^2 u_{L} 
\\
&\leq  \tfrac{u_{U} - u_{0}}{u_{U}^{2b}} 
 +  \tfrac{u_{0}-u_{L}}{u_{0}^{2b}} 
 +  2 \big(\tfrac{1}{u_{L}^{b}} - \tfrac{1}{u_{0}^{b}}\big)^2 u_{L} 
 + 2 \big(\tfrac{1}{u_{0}^{b}} - \tfrac{1}{u_{U}^{b}}\big)^2 u_{0} 
\\
& \leq 2\,r_{0}^{2}(u_{L},u_{0}) + 2\,r_{0}^{2}(u_{0},u_{U}) \leq 4 \epsilon^{2}. 
\end{align*}

% 
% But thanks to \textit{(ii)}

\end{proof}

\end{lemma}

\bibliographystyle{apalike}
\bibliography{short,ReferencesU}

\end{document}